\documentclass{amsart}
\usepackage{amscd,amsmath,xypic,amssymb,combelow,tikz-cd,etoolbox,calligra,mathrsfs,enumitem,mathtools,hyperref,comment,graphicx}

\makeatletter
\patchcmd{\@settitle}{\uppercasenonmath\@title}{}{}{}
\makeatother

\newtheorem{theorem}{Theorem}[section]

\newtheorem{proposition}[theorem]{Proposition}
\newtheorem{lemma}[theorem]{Lemma}

\newtheorem{definition}[theorem]{Definition}
\newtheorem{claim}[theorem]{Claim}

\newtheorem{remark}[theorem]{Remark}

\def\fb{{\mathfrak{b}}}
\def\fg{{\mathfrak{g}}}

\def\fsl{{\mathfrak{sl}}}
\def\fgl{{\mathfrak{gl}}}

\def\BC{{\mathbb{C}}}
\def\BK{{\mathbb{K}}}

\def\BN{{\mathbb{N}}}

\def\BP{{\mathbb{P}}}
\def\BR{{\mathbb{R}}}

\def\BZ{{\mathbb{Z}}}

\def\CA{{\mathcal{A}}}
\def\CB{{\mathcal{B}}}

\def\CO{{\mathcal{O}}}
\def\CP{{\mathcal{P}}}

\def\CR{{\mathcal{R}}}
\def\CS{{\mathcal{S}}}

\def\CV{{\mathcal{V}}}

\def\ph{\varphi}

\def\Sym{\textrm{Sym}}

\def\U{\mathbf{U}}
\def\Up{\mathbf{U}^+}
\def\Upm{\mathbf{U}^\pm}

\def\Um{\mathbf{U}^-}
\def\Ug{\mathbf{U}^\geq}
\def\Ul{\mathbf{U}^\leq}

\def\tU{\mathbf{\widetilde{U}}}
\def\tUp{\mathbf{\widetilde{U}}^+}
\def\tUpm{\mathbf{\widetilde{U}}^\pm}

\def\tUm{\mathbf{\widetilde{U}}^-}
\def\tUg{\mathbf{\widetilde{U}}^\geq}
\def\tUl{\mathbf{\widetilde{U}}^\leq}

\def\UU{U_q(L\fg)}

\def\br{{\mathbf{r}}}

\def\bs{\boldsymbol{\varsigma}}

\def\nn{{\mathbb{N}^I}}
\def\zz{{\mathbb{Z}^I}}
\def\rr{{\mathbb{R}^I}}

\def\bmu{{\boldsymbol{\mu}}}

\def\bpsi{{\boldsymbol{\psi}}}

\def\bm{{\boldsymbol{m}}}
\def\bn{{\boldsymbol{n}}}
\def\bp{{\boldsymbol{p}}}

\def\bx{{\boldsymbol{x}}}

\def\b0{{\boldsymbol{0}}}
\def\bone{{\boldsymbol{1}}}

\def\loccit{\emph{loc.~cit.~}}
\def\loccitt{\emph{loc.~cit.}}

\def\hdeg{\text{hdeg }}
\def\vdeg{\text{vdeg }}

\def\bx{\boldsymbol{x}}

\def\vac{|\varnothing\rangle}

\def\bone{{\boldsymbol{1}}}

\def\binfty{\boldsymbol{\infty}}

\def\Sym{\text{Sym}}

\def\ho{\stackrel{\mathsf{H}}{\otimes}}
\def\vo{\stackrel{\mathsf{V}}{\otimes}}

\def\bla{\boldsymbol{\lambda}}
\def\bmu{\boldsymbol{\mu}}
\def\bga{\boldsymbol{\gamma}}
\def\bth{\boldsymbol{\theta}}

\begin{document}

\title[A new new coproduct on quantum loop algebras]{\Large{\textbf{A new new coproduct on quantum loop algebras}}} 

\author[Andrei Negu\cb t]{Andrei Negu\cb t}

\address{École Polytechnique Fédérale de Lausanne (EPFL), Lausanne, Switzerland \newline \text{ } \ \ Simion Stoilow Institute of Mathematics (IMAR), Bucharest, Romania} 

\email{andrei.negut@gmail.com}

\maketitle
	
\begin{abstract} Quantum loop algebras generalize $U_q(\widehat{\fg})$ for simple Lie algebras $\fg$, and they include examples such as quantum affinizations of Kac-Moody Lie algebras, $K$-theoretic Hall algebras of quivers, and BPS algebras for toric Calabi-Yau threefolds. In the present paper, we define a coproduct on general quantum loop algebras, which coincides with the Drinfeld-Jimbo coproduct in the particular case of $U_q(\widehat{\fg})$. We use our construction to prove fundamental facts about representations of quantum loop algebras, such as the rationality of $R$-matrices, multiplicativity of $q$-characters, and polynomiality of theta series.

\end{abstract}

\bigskip

\section{Introduction}
\label{sec:intro}

\medskip

\subsection{The title}
\label{sub:title}

In the title of this paper, ``new" (unsurprisingly)  stands for something that is different from ``old". Historically, quantum affine algebras $U_q(\widehat{\fg})$ were first endowed with an old coproduct by Drinfeld and Jimbo (\cite{Dr 0, J}), and then Drinfeld defined a new coproduct that was quite different from (and in a certain sense, orthogonal to) the old one. In the present paper, we work with quantum loop algebras in the greater generality of \cite{N Arbitrary}, in which one can define an analogue of Drinfeld's new coproduct, but there is no sense in which the Drinfeld-Jimbo construction applies. Instead, we define a ``new new" coproduct on quantum loop algebras that is different from (and in a certain sense, orthogonal to) the new one. The motivation behind this construction is that
$$
\Big(\text{new new coproduct} \Big) = \Big( \text{old coproduct} \Big)
$$
in the particular case of quantum affine algebras $U_q(\widehat{\fg})$ for simple Lie algebras $\fg$. At a slightly deeper level, one can say that each word ``new" in the title corresponds to rotating the loop root lattice by $90^{\circ}$. Therefore, just like the Drinfeld-Jimbo coproduct differs from Drinfeld's new coproduct by a rotation by $90^{\circ}$, so does Drinfeld's new coproduct from the new new coproduct we introduce in this paper.

\medskip

\subsection{Algebras}
\label{sub:algebras intro}

Fix a finite set $I$, a field $\BK$ of characteristic 0, and rational functions
\begin{equation}
\label{eqn:zeta intro}
\zeta_{ij}(x) \in \frac {\BK[x^{\pm 1}]}{(1-x)^{\delta_{ij}}}
\end{equation}
such that
\begin{equation}
\label{eqn:assumption}
\lim_{x \rightarrow \infty} \frac {\zeta_{ij}(x)}{\zeta_{ji}(x^{-1})} < \infty
\end{equation}
for all $i, j \in I$. Under this assumption, we defined a quantum loop algebra in \cite{N Arbitrary}
\begin{equation}
\label{eqn:quantum loop algebra intro}
\U = \BK \Big \langle e_{i,d}, f_{i,d}, \ph^+_{i,d'}, \ph^-_{i,d'} \Big \rangle_{i \in I, d \in \BZ, d' \geq 0} \Big/ \Big(\text{relations \eqref{eqn:rel quantum 1}-\eqref{eqn:rel quantum 9}}\Big) 
\end{equation}
The motivating example of this construction arises from a simple Lie algebra $\fg$: we take $I$ to be a set of simple roots, $\BK = \BC$ (hereafter fix $q \in \BC^* \backslash \sqrt[\BN]{1}$) and define
\begin{equation}
\label{eqn:zeta intro particular}
\zeta_{ij}(x) = \frac {(q^{-d_{ij}}-x)(-x)^{-\delta_{i>j}}}{(1-x)^{\delta_{ij}}}
\end{equation}
for some total order $<$ on $I$, where
\begin{equation}
\label{eqn:cartan matrix intro}
C = \left(c_{ij} = \frac {2d_{ij}}{d_{ii}} \in \BZ \right)_{i,j \in I}
\end{equation}
is the Cartan matrix of $\fg$. In this particular case, the algebra $\U$ is none other than
\begin{equation}
\label{eqn:quantum affine intro}
\Big( \U \text{ for the choice \eqref{eqn:zeta intro particular}} \Big) = U_q(L\fg) \cong U_q(\widehat{\fg})_{c=1}
\end{equation}
The isomorphism on the right was claimed by Drinfeld (\cite{Dr}) and proved by Beck (\cite{B}, using work of \cite{Da0}). However, the quantum loop algebras \eqref{eqn:quantum loop algebra intro} are significantly more general than $\UU$. They include examples such as quantum affinizations of Kac-Moody Lie algebras, $K$-theoretic Hall algebras of doubled quivers, BPS algebras associated to toric Calabi-Yau threefolds, and Hall algebras of curves over finite fields (which have a long history of study by numerous authors, but we refer to \cite{N Symmetric, N Wheel, N Reduced, NSS} respectively for a treatment in the language of the present paper).

\medskip

\subsection{Simple modules}
\label{sub:simples intro}

In the context of the present paper, loop weights
$$
\bpsi = \left(\psi_i(z) = \sum_{d=0}^{\infty} \frac {\psi_{i,d}}{z^d} \in \BK[[z^{-1}]]^\times \right)_{i \in I}
$$
are $I$-tuples of power series in $\BK$. We are interested in studying simple modules
\begin{equation}
\label{eqn:simple intro}
`` \ \U \curvearrowright L(\bpsi) \ "
\end{equation}
which are generated by a single vector $\vac$ modulo the relations
$$
e_{i,d} \cdot \vac = 0
$$
$$
\ph^+_{i,d'} \cdot \vac = \psi_{i,d'} \vac
$$
for all $i \in I$, $d \in \BZ$, $d' \geq 0$. However, as is already apparent in the situation of quantum affine algebras \eqref{eqn:quantum affine intro} (studied in \cite{CP, HJ}), such modules are not acted on by the entire quantum loop algebra $\U$. Instead, we construct triangular decompositions
\begin{equation}
\label{eqn:triangular intro}
\U = \CA^{\geq \bp} \otimes \CA^{\leq \bp}
\end{equation}
for any $\bp \in \rr$, generalizing the construction of \cite{N Cat, N Char} (which in turn generalizes the triangular decomposition of $U_q(\widehat{\fg})$ into positive and negative Borel subalgebras, in the particular case of \eqref{eqn:quantum affine intro}). For any loop weight $\bpsi$, we will define a simple module
\begin{equation}
\label{eqn:simple actual intro}
\CA^{\geq \bp} \curvearrowright L(\bpsi)
\end{equation}
in Section \ref{sec:modules}, thus generalizing the main construction of \cite{N Cat}, which itself generalizes \cite{HJ}. More systematically, we consider the following analogue of the main construction of \loccitt: a module
$$
\CA^{\geq \bp} \curvearrowright V
$$
is said to be in category $\CO$ if it is diagonalizable with respect to the finite Cartan subalgebra generated by $\{\ph_{i,0}^+\}_{i \in I}$, with finite-dimensional eigenspaces that are non-zero only when the eigenvalues lie in a certain union of cones (see Definition \ref{def:category o}). The prime example of a module in category $\CO$ is the simple module $L(\bpsi)$ of \eqref{eqn:simple actual intro} if all the constituent power series of $\bpsi = (\psi_i(z))_{i\in I}$ are expansions of rational functions. Such loop weights $\bpsi$ will be called rational.

\medskip

\subsection{Coproducts}
\label{sub:coproducts intro}

However, what was missing from \cite{N Cat, N Char} was a notion of tensor product of modules, and this is the main construction of the present paper. 

\medskip

\begin{theorem}
\label{thm:main intro}

(subsumed by Theorem \ref{thm:main}) There is a topological coproduct
\begin{equation}
\label{eqn:coproduct intro}
\U \xrightarrow{\Delta_{\bp}} \U \ho \U
\end{equation}
(see Subsection \ref{sub:gradings} for the definition of $\ho$) which preserves the subalgebra $\CA^{\geq \bp}$.

\end{theorem}

\medskip

\noindent The completion $\ho$ is such that the coproduct \eqref{eqn:coproduct intro} gives rise to a well-defined $\CA^{\geq \bp}$-module structure on $V \otimes W$ for any modules $\CA^{\geq \bp} \curvearrowright V,W$ in category $\CO$.

\medskip 

\begin{theorem}
\label{thm:affine intro}
	
(Theorem \ref{thm:affine}) In the particular case of $\U$ that appears in \eqref{eqn:quantum affine intro}, the coproduct $\Delta_{\b0}$ matches the Drinfeld-Jimbo coproduct on quantum affine algebras. 

\end{theorem}

\medskip

\noindent For any $\bp$, the subalgebra $\CA^{\geq \bp} \subset \U$ contains the positive loop Cartan subalgebra
\begin{equation}
\label{eqn:half loop cartan}
\BK[\ph^+_{i,d}]_{i \in I, d \geq 0}
\end{equation}
which is commutative. We will consider those modules $\CA^{\geq \bp} \curvearrowright V$ in category $\CO$, that decompose into generalized eigenspaces for the subalgebra \eqref{eqn:half loop cartan}
\begin{equation}
\label{eqn:decompose intro}
V = \bigoplus_{\text{loop weights }\bpsi} V_{\bpsi}
\end{equation}
such that every $\ph^+_{i,d}$ acts via the generalized eigenvalue $\psi_{i,d}$ on $V_\bpsi$ (the decomposition \eqref{eqn:decompose intro} exists on general grounds if $\BK$ is algebraically closed). Following \cite{FR}, we define the $q$-character as
\begin{equation}
\label{eqn:q-character intro}
\chi_q(V) = \sum_{\text{loop weights }\bpsi} \dim_{\BK} \left( V_{\bpsi} \right) [\bpsi]
\end{equation}
for various formal symbols $[\bpsi]$ associated to loop weights. If we define the product of these formal symbols component-wise (i.e. $[\bpsi][\bpsi'] = [\bpsi\bpsi']$), then we show in Proposition \ref{prop:multiplicative} that $q$-characters are multiplicative with respect to tensor products, thus generalizing a result of \cite{FR} for quantum affine algebras. More broadly speaking, the coproduct $\Delta_{\bp}$ makes category $\CO$ into a tensor category, and the $q$-character gives an (injective, on general grounds) ring homomorphism from the Grothendieck group of category $\CO$ to the ring of formal linear combinations of the symbols $[\bpsi]$.
 
\medskip 

\subsection{R-matrices}
\label{sub:R-matrices intro}

For any $\bp \in \rr$, Theorem \ref{thm:main} proves that the decomposition
\begin{equation}
\label{eqn:triangular intro 2}
\U = \CA^{\geq \bp} \otimes \CA^{\leq \bp}
\end{equation}
is a particular case of Drinfeld's quantum double construction, with respect to the coproduct $\Delta_{\bp}$ and the pairing \eqref{eqn:footnote}. Thus, there exists a universal $R$-matrix
\begin{equation}
\label{eqn:intro r-matrix 1}
_{\bar{\bp}}\CR_{\bp} \in \U \ \widehat{\otimes} \ \U, \qquad  \left(_{\bar{\bp}}\CR_{\bp}\right) \cdot \Delta_{\bp}(-) = \Delta_{\bp}^{\text{op}}(-) \cdot \left(_{\bar{\bp}}\CR_{\bp}\right)
\end{equation}
where we define completions $\widehat{\otimes}$ and $\bar{\otimes}$ in Subsection \ref{sub:universal 2}. More generally, we will construct the following objects for all $\bp^1,\bp^2 \in \rr$
\begin{align}
&_{\bp^2}\CR_{\bp^1} \in \U  \ \bar{\otimes} \ \U, \qquad  \left( _{\bp^2}\CR_{\bp^1} \right) \cdot \Delta_{\bp^1}(-) = \Delta_{\bp^2}(-) \cdot \left( _{\bp^2}\CR_{\bp^1} \right) \label{eqn:intro r-matrix 2}\\
&_{\bar{\bp}^2}\CR_{\bp^1} \in \U \ \widehat{\otimes} \ \U, \qquad  \left( _{\bar{\bp}^2}\CR_{\bp^1} \right) \cdot \Delta_{\bp^1}(-) = \Delta_{\bp^2}^{\text{op}}(-) \cdot \left( _{\bar{\bp}^2}\CR_{\bp^1} \right) \label{eqn:intro r-matrix 3}
\end{align}
The existence of $_{\bar{\bp}^2}\CR_{\bp^1}$ is subject to the usual caveats involving the Cartan subalgebra, that the reader may find recalled in Subsection \ref{sub:infinite slope 2}. We may evaluate the above $R$-matrices in tensor products of modules as follows. Generalizing a phenomenon that has long been known for quantum affine algebras with the Drinfeld-Jimbo coproduct, we show in Proposition \ref{prop:regular} that for rational loop weights $\bpsi$ which are regular$^{\neq 0}$ (i.e. the rational functions $\psi_i(z)$ are regular and non-zero at $z=0$ for all $i \in I$), the action $\CA^{\geq \bp} \curvearrowright L(\bpsi)$ extends to an action
\begin{equation}
\label{eqn:extended action intro}
\U \curvearrowright L(\bpsi)
\end{equation}
which is independent of $\bp$. Thus, \eqref{eqn:extended action intro} is an example of a module
\begin{equation}
\label{eqn:integrable intro}
\U \curvearrowright V
\end{equation}
in category $\CO$, namely one whose weight spaces are finite-dimensional and non-zero only in a finite union of cones. For any two modules $\U \curvearrowright V,W$ in category $\CO$, the coproducts $\Delta_{\bp}$ and $\Delta_{\bp}^{\text{op}}$ for various $\bp \in\rr$ give rise to module structures
\begin{equation}
\label{eqn:tensor action intro}
\U \curvearrowright V \otimes_{\bp} W \qquad \text{and} \qquad \U \curvearrowright  V \otimes^{\text{op}}_{\bp} W
\end{equation}
on the tensor product $V \otimes W$. The tensors \eqref{eqn:intro r-matrix 2} and \eqref{eqn:intro r-matrix 3} produce $\U$-intertwiners
\begin{equation} 
\label{eqn:intro r-matrix in rep 1}
_{\bp^2}R_{\bp^1}(u): V^u \otimes_{\bp^1} W \rightarrow V^u \otimes_{\bp^2} W[u^{\pm 1}]
\end{equation}
\begin{equation}  
\label{eqn:intro r-matrix in rep 2}
_{\bar{\bp}^2}R_{\bp^1}(u): V^u \otimes_{\bp^1} W \rightarrow V^u \otimes^{\text{op}}_{\bp^2} W((u))
\end{equation} 
where $V^u$ refers to twisting the action of $V$ by powers of $u$ (see Subsection \ref{sub:r-matrices} for details). This opens the door to asking, in the current generality of $\U$, a host of interesting questions that have been studied for $U_q(\widehat{\fg})$: factorization of $R$-matrices, calculation of transfer matrices, relation to XXZ Hamiltonians, Baxter's equations, the Bethe ansatz, relations to $W$-algebras etc (\cite{FJMM, FJMM2, FT, FH, FR}).

\medskip

\subsection{Techniques} 
\label{sub:techniques}

Let us discuss our approach in more technical detail. We will heavily use the shuffle algebra incarnation of quantum loop algebras, namely
$$
\U \cong \CS^+ \otimes \BK[\ph^+_{i,d}, \ph^-_{i,d}]_{i \in I, d\geq 0} \otimes \CS^-
$$
where $\CS^\pm$ are defined in Subsection \ref{sub:shuffle}. We also recall (the natural generalization of) Drinfeld's new topological coproduct $\Delta$ on $\U$, and the pairing
\begin{equation}
\label{eqn:pair shuffle intro}
\CS^+ \otimes \CS^- \xrightarrow{\langle \cdot, \cdot \rangle} \BK
\end{equation}
with respect to which $\U$ is a Drinfeld double. The coproduct and pairing above are the essential inputs that we use to define the coproducts $\Delta_{\bp}$ on $\CA^{\geq \bp}$ and $\CA^{\leq \bp}$, and to show that $\U$ is the Drinfeld double of the algebras $\CA^{\geq \bp}$ and $\CA^{\leq \bp}$ for any $\bp \in \rr$. With this in mind, our approach produces explicit formulas for $\Delta_{\bp}$ only inasmuch as one has explicit answers to the following problems.

\bigskip

\noindent \textbf{Problem 1.} \emph{Find explicit descriptions of $\CS^\pm$ as sets of polynomials satisfying some collection of ``wheel conditions", e.g. \eqref{eqn:wheel} (see the general discussion of \cite{N Arbitrary}}).

\bigskip

\noindent \textbf{Problem 2.} \emph{Find explicit descriptions of the pairing \eqref{eqn:pair shuffle intro}, see Remark \ref{rem:pairing}.}

\bigskip

\noindent So far, Problem 1 has been solved in the following particular cases: quantum affine algebras (\cite{NT}), quantum affinizations of simply-laced Kac-Moody Lie algebras (\cite{N Symmetric}), $K$-theoretic Hall algebras of doubled quivers (\cite{N Wheel}) and BPS algebras associated to toric Calabi-Yau threefolds (\cite{N Reduced}). Problem 2 has also been solved in the latter three settings in the works referenced above; we will provide a solution to this problem in the case of quantum affine algebras in Lemma \ref{lem:contour}. However, a general solution to Problems 1 and 2 is open and would be extremely interesting.

\medskip 

\subsection{Acknowledgements} I would like to thank David Hernandez and Alexander Tsymbaliuk for many great conversations on quantum loop algebras and their representation theory. I gratefully acknowledge the support of the Swiss National Science Foundation grant 10005316.

\bigskip

\section{Quantum loop and shuffle algebras}
\label{sec:algebras}

\medskip

\noindent We recall the definition of quantum loop algebras in the generality of Subsection \ref{sub:algebras intro}, as well as their shuffle algebra incarnation. Special emphasis will be placed on tools and techniques which will be used in later sections, such as the bialgebra structure, Drinfeld double, and various completions of the aforementioned algebras.

\medskip

\subsection{Basic notations}
\label{sub:basic}

The set $\BN$ is henceforth thought to contain 0. We fix a finite set $I$, a field $\BK$ of characteristic 0 and a collection of rational functions $\{\zeta_{ij}(x)\}_{i,j \in I}$ as in Subsection \ref{sub:algebras intro}. The abelian group $\zz$ plays the role of the root lattice for our quantum algebras, with
$$
\bs^i = \underbrace{(0,\dots,0,1,0,\dots,0)}_{1 \text{ on }i\text{-th spot}}
$$
playing the role of simple roots. For any $\bm = (m_i)_{i \in I}, \bn = (n_i)_{i \in I} \in \rr$, we let
\begin{equation}
\label{eqn:dot product}
\bm \cdot \bn = \sum_{i \in I} m_i n_i
\end{equation}
For $\bm,\bn \in \rr$, we will write $\bm \leq \bn$ if $m_i \leq n_i$ for all $i \in I$. We will also use the notation $\b0 = (0,\dots,0)$ and $\bone = (1,\dots,1)$. For any $\bm = (m_i)_{i \in I} \in \rr$, we write
\begin{equation}
\label{eqn:coordinate sum}
|\bm| = \sum_{i \in I} m_i
\end{equation}
Because of assumption \eqref{eqn:assumption}, we can write for all $i,j \in I$
\begin{equation}
\label{eqn:magnitude}
\zeta_{ij}(x) = \frac {c_{ij} x^{\#_{ij}+\delta_{ij}} + \dots + c'_{ij} x^{-\#_{ji}}}{(x-1)^{\delta_{ij}}}
\end{equation}
for various scalars $c_{ij}, c_{ij}' \in \BK^*$ and integers $\#_{ij}$, such that $c_{ij} x^{\#_{ij}}$ is the leading order term of $\zeta_{ij}(x)$ as $x \rightarrow \infty$. The numbers $\#_{ij}$ give rise to a bilinear form
\begin{equation}
\label{eqn:euler form}
\langle \bm, \bn \rangle = \sum_{i,j \in I} m_i n_j \#_{ij}
\end{equation}
for all $\bm = (m_i)_{i \in I}, \bn = (n_i)_{i \in I}$. The following scalars will come up in Section \ref{sec:slopes}
\begin{equation}
\label{eqn:scalars}
\gamma_{\bm, \bn} = \prod_{i,j \in I} \left[ (-1)^{\delta_{ij}}\frac {c_{ij}}{c_{ji}'} \right]^{m_i n_j} \in \BK^*
\end{equation}

\medskip 

\subsection{The pre-quantum loop algebra}
\label{sub:quad}

The following notion is motivated by the setting of \eqref{eqn:quantum affine intro}, in which it $q$-deforms the Lie bracket on $\fg[t^{\pm 1}]$ for simple $\fg$.

\medskip 

\begin{definition}
\label{def:quad}
	
The (positive part of the) \emph{pre-quantum loop algebra} $\tUp$ is 
$$
\tUp = \BK \Big \langle e_{i,d} \Big \rangle _{i \in I, d \in \BZ} \Big/ \Big(\text{relations \eqref{eqn:rel quad}} \Big)
$$
where for every $i,j \in I$ we set
\begin{equation}
\label{eqn:rel quad}
e_i(z) e_j(w) \zeta_{ji} \left(\frac wz\right) = e_j(w) e_i(z) \zeta_{ij} \left( \frac zw \right)
\end{equation}
Above and henceforth, we consider the formal series
$$
e_i(z) = \sum_{d \in \BZ} \frac {e_{i,d}}{z^d}
$$
for all $i \in I$, and interpret relation \eqref{eqn:rel quad} as an infinite collection of relations obtained by equating the coefficients of all $\{z^aw^b\}_{a,b\in \BZ}$ in the left and right-hand sides (if $i = j$, one clears the denominators $z-w$ from \eqref{eqn:rel quad} before equating coefficients). 
	
\end{definition}

\medskip 

\noindent We also consider the negative part of the pre-quantum loop algebra
$$
\tUm = \tU^{+,\text{op}}
$$
with generators denoted by $f_{i,d}$ instead of $e_{i,d}$. They satisfy the relations
\begin{equation}
\label{eqn:rel quad opp}
f_i(z) f_j(w) \zeta_{ij} \left( \frac zw \right) = f_j(w) f_i(z) \zeta_{ji} \left(\frac wz\right)
\end{equation}
for all $i,j \in I$, where $f_i(z) = \sum_{d \in \BZ} \frac {f_{i,d}}{z^d}$. 

\medskip

\subsection{The shuffle algebra}
\label{sub:shuffle}

The following construction generalizes the trigonometric version (due to \cite{E}) of the Feigin-Odesskii elliptic shuffle algebras of \cite{FO}. Consider the vector space of Laurent polynomials in arbitrarily many variables
\begin{equation}
\label{eqn:big shuffle}
\CV = \bigoplus_{\bn \in \nn} \CV_{\bn}, \quad \text{where} \quad \CV_{(n_i \geq 0)_{i \in I}} = \BK[z_{i1}^{\pm 1},\dots,z_{in_i}^{\pm 1}]^{\text{sym}}_{i \in I} 
\end{equation}
Above, ``sym" refers to Laurent polynomials which are color-symmetric, i.e. symmetric in the variables $z_{i1},\dots,z_{in_i}$ for each $i \in I$ separately. The vector space $\CV$ is called the \emph{big shuffle algebra} when endowed with the following shuffle product:
\begin{equation}
	\label{eqn:mult}
	E( z_{i1}, \dots, z_{i n_i})_{i \in I} * E'(z_{i1}, \dots,z_{i n'_i})_{i \in I} = 
\end{equation}
$$
\textrm{Sym} \left[ \frac {E(z_{i1}, \dots, z_{in_i}) E'(z_{i,n_i+1}, \dots, z_{i,n_i+n'_i})}{\bn! \bn'!}
\prod_{i,j \in I} \mathop{\prod_{1 \leq a \leq n_i}}_{n_j < b \leq n_j+n_j'} \zeta_{ij} \left( \frac {z_{ia}}{z_{jb}} \right) \right]
$$
The word ``Sym" in \eqref{eqn:mult} denotes symmetrization with respect to the
\begin{equation*}
	(\bn+\bn')! := \prod_{i\in I} (n_i+n'_i)!
\end{equation*}
permutations of the variables $\{z_{i1}, \dots, z_{i,n_i+n'_i}\}$ for each $i$ independently. Let
$$
\CV^+ = \CV \qquad \text{and} \qquad \CV^- = \CV^{\text{op}}
$$
The reason for formula \eqref{eqn:mult} is to ensure that there exist algebra homomorphisms
\begin{equation}
\label{eqn:tupsilon}
\widetilde{\Upsilon}^\pm : \tUpm \rightarrow \CV^{\pm}
\end{equation}
given by sending $e_{i,d}$ and $f_{i,d}$ (respectively) to $z_{i1}^d \in \CV_{\bs^i} = \CV^{\text{op}}_{\bs^i}$, for all $i \in I$, $d \in \BZ$.

\medskip

\begin{definition}
\label{def:shuffle}

Define the positive/negative \emph{shuffle algebras} as
\begin{equation}
\label{eqn:spherical def}
\CS^{\pm} = \emph{Im }\widetilde{\Upsilon}^{\pm}
\end{equation}
and define the positive/negative parts of the \emph{quantum loop algebra} as
\begin{equation}
\label{eqn:quantum loop}
\U^{\pm} = \tU^\pm \Big/ \emph{Ker }\widetilde{\Upsilon}^\pm 
\end{equation}
Then we have induced isomorphisms 
\begin{equation}
\label{eqn:upsilon}
\Upsilon^\pm : \U^\pm \xrightarrow{\sim} \CS^\pm
\end{equation}

\end{definition}

\medskip 

\noindent Elements of $\CS^\pm$ will be called shuffle elements. By definition, any $E \in \CS^+$ can be written as a linear combination of shuffle elements of the form
\begin{equation}
\label{eqn:spherical}
E = \text{Sym} \left[ \nu(z_1,\dots,z_n) \prod_{1 \leq a < b \leq n} \zeta_{i_ai_b} \left(\frac {z_a}{z_b} \right) \right] 
\end{equation} 
as $\nu$ goes over all Laurent polynomials. In formula \eqref{eqn:spherical}, we consider any $i_1,\dots,i_n \in I$ and use the notation $z_a$ as a placeholder for the variable $z_{i_a\bullet_a}$, where $\bullet_1,\dots,\bullet_n$ denote the minimal positive integers such that $\bullet_a < \bullet_b$ if $a<b$ and $i_a = i_b$.

\medskip 

\subsection{Extended algebras}
\label{sub:extended}

To make $\tU^\pm$, $\CV^\pm$, $\U^\pm \cong \CS^\pm$ into bialgebras, we need to extend them by introducing commuting loop Cartan elements. The vector spaces
\begin{align}
	&\tUg = \tUp \otimes \BK[\ph^+_{i,d}]_{i \in I, d \geq 0} \label{eqn:tug} \\
	&\tUl = \BK[\ph^-_{i,d}]_{i \in I, d \geq 0} \otimes \tUm \label{eqn:tul}
\end{align}
can be made into algebras by imposing the commutation relations
\begin{align}
	&\ph^+_i(z) e_j(w) = e_j(w) \ph^+_i(z) \frac {\zeta_{ij} \left(\frac zw \right)}{\zeta_{ji} \left(\frac wz \right)} \label{eqn:uug} \\
	&f_j(w)\ph^-_i(z)  = \ph^-_i(z)  f_j(w) \frac {\zeta_{ij} \left(\frac zw \right)}{\zeta_{ji} \left(\frac wz \right)} \label{eqn:uul}
\end{align}
for all $i,j \in I$, where 
$$
\ph^+_i(z) = \sum_{d=0}^{\infty} \frac {\ph^+_{i,d}}{z^d} \qquad \text{and} \qquad \ph^-_i(z) = \sum_{d=0}^{\infty} \ph^-_{i,d} z^d
$$
Note that the assumption \eqref{eqn:assumption} implies that $\frac {\zeta_{ij} (x)}{\zeta_{ji} (x^{-1})}$ is regular and non-zero at both $x = 0$ and $x = \infty$, for all $i,j \in I$. With this in mind, one interprets the relations in \eqref{eqn:uug} (respectively \eqref{eqn:uul}) by expanding them as power series in negative (respectively positive) powers of $\frac zw$. Similarly, we define
\begin{align}
&\CV^{\geq} = \CV^+ \otimes \BK[\ph^+_{i,d}]_{i \in I, d \geq 0} \label{eqn:vg} \\
&\CV^{\leq} =  \BK[\ph^-_{i,d}]_{i \in I, d \geq 0} \otimes \CV^- \label{eqn:vl}
\end{align}
and make them into algebras by imposing the commutation relations
\begin{align}
&\ph^+_i(y) E(z_{j1},\dots,z_{jn_j})_{j \in I} = E(z_{j1},\dots,z_{jn_j})_{j \in I} \ph^+_i(y) \prod_{j \in I} \prod_{a=1}^{n_j}\frac {\zeta_{ij} \left(\frac y{z_{ja}} \right)}{\zeta_{ji} \left(\frac {z_{ja}}y \right)} \label{eqn:ssg} \\
&F(z_{j1},\dots,z_{jn_j})_{j \in I} \ph^-_i(y)  = \ph^-_i(y) F(z_{j1},\dots,z_{jn_j})_{j \in I} \prod_{j \in I} \prod_{a=1}^{n_j}\frac {\zeta_{ij} \left(\frac y{z_{ja}} \right)}{\zeta_{ji} \left(\frac {z_{ja}}y \right)} \label{eqn:ssl}
\end{align}
for all $E \in \CV^+$, $F \in \CV^-$ and $i \in I$. Because the homomorphisms $\widetilde{\Upsilon}^\pm$ respect all the constructions above, we obtain natural algebra structures on
\begin{align}
	&\Ug = \Up \otimes \BK[\ph^+_{i,d}]_{i \in I, d \geq 0} \label{eqn:ug} \\
	&\Ul = \BK[\ph^-_{i,d}]_{i \in I, d \geq 0} \otimes \Um \label{eqn:ul}
\end{align}
and 
\begin{align}
&\CS^{\geq} = \CS^+ \otimes \BK[\ph^+_{i,d}]_{i \in I, d \geq 0} \label{eqn:sg} \\
&\CS^{\leq} =  \BK[\ph^-_{i,d}]_{i \in I, d \geq 0} \otimes \CS^- \label{eqn:sl}
\end{align}
which are isomorphic to each other with respect to \eqref{eqn:upsilon}.

\medskip

\subsection{Gradings and completions}
\label{sub:gradings}

All algebras in this paper are graded by 
$$
\deg = (\text{hdeg}, \text{vdeg}) \in \zz \times \BZ
$$
with the $\zz$ component called \emph{horizontal} degree (denoted by hdeg) and the $\BZ$ component called \emph{vertical} degree (denoted by vdeg). Explicitly, we have
\begin{align*} 
&\deg e_{i,d} = (\bs^i, d) \\
&\deg f_{i,d} = (-\bs^i, d) \\
&\deg \ph^+_{i,d} = (0, d) \\
&\deg \ph^-_{i,d} = (0, -d) \\
&\deg E = (\bn, \text{homogeneous degree of }E) \\
&\deg F = (-\bn, \text{homogeneous degree of }F)
\end{align*}
for any $E \in \CV_{\bn}$ and $F \in \CV^{\text{op}}_{\bn}$ which are homogeneous in all their variables. Write
\begin{align}
&\CV^+ = \bigoplus_{\bn \in \nn} \CV_{\bn} = \bigoplus_{(\bn,d) \in \nn \times \BZ} \CV_{\bn,d} \label{eqn:summand plus} \\
&\CV^- = \bigoplus_{\bn \in \nn} \CV_{-\bn} = \bigoplus_{(\bn,d) \in \nn \times \BZ} \CV_{-\bn,d} \label{eqn:summand minus}
\end{align} 
for the graded summands of $\CV^\pm$ (and analogously for $\CS^\pm, \tUpm, \Upm$). For any algebra $A$ which is graded by $\zz \times \BZ$, we may define the completions
\begin{align}
&A \vo A = \bigoplus_{(\bn,d) \in \zz \times \BZ} (A \vo A)_{\bn,d} \label{eqn:vo} \\
&A \ho A = \bigoplus_{(\bn,d) \in \zz \times \BZ} (A \ho A)_{\bn,d} \label{eqn:ho}
\end{align}
where $(A \vo A)_{\bn,d}$ (respectively $(A \ho A)_{\bn,d}$) consists of infinite sums of tensors $x \otimes y$ where all but finitely many summands have the property that $\vdeg x \geq N$ and $\vdeg y \leq -N$ (respectively $|\hdeg x| \leq -N$ and $|\hdeg y| \geq N$, recall the notation \eqref{eqn:coordinate sum}) for any $N \in \BN$. It is a straightforward exercise, which we leave to the reader, to check that \eqref{eqn:vo} and \eqref{eqn:ho} are indeed algebras, i.e. the product of infinite sums of tensors is a well-defined infinite sum of tensors in the sense above. 

\medskip 

\subsection{The topological coproduct}
\label{sub:coproduct}

Our reason for defining the extended algebras in the previous Subsection is to make them into bialgebras. We start with the natural generalization of Drinfeld's new coproduct on quantum affine algebras
\begin{equation}
\label{eqn:coproduct tu}
\Delta : \tUg \rightarrow \tUg \vo \tUg \qquad \text{and} \qquad \Delta : \tUl \rightarrow \tUl \vo \tUl 
\end{equation}
given by
\begin{equation}
\label{eqn:coproduct h}
\Delta(\ph^\pm_i(z)) = \ph^\pm_i(z) \otimes \ph^\pm_i(z)  
\end{equation}
\begin{equation}
\label{eqn:coproduct e}
\Delta(e_i(z)) = \ph^+_i(z) \otimes e_i(z) + e_i(z) \otimes 1  
\end{equation}
\begin{equation}
\label{eqn:coproduct f}
\Delta(f_i(z)) = 1 \otimes f_i(z) + f_i(z) \otimes \ph^-_i(z)
\end{equation}
Similarly, the following coproducts 
\begin{equation}
\label{eqn:coproduct v}
\Delta : \CV^{\geq} \rightarrow \CV^{\geq} \vo \CV^{\geq} \qquad \text{and} \qquad \Delta : \CV^{\leq} \rightarrow \CV^{\leq} \vo \CV^{\leq}
\end{equation}
are natural generalizations of those of \cite{N Shuffle}: $\Delta(\ph^\pm_i(z)) = \ph^\pm_i(z) \otimes \ph^\pm_i(z)$ and
\begin{align}
\Delta(E) = \sum_{\b0 \leq \bm \leq \bn} \frac {\prod^{j \in I}_{m_j < b \leq n_j} \ph^+_j(z_{jb}) E(z_{i1},\dots , z_{im_i} \otimes z_{i,m_i+1}, \dots, z_{in_i})}{\prod^{i \in I}_{1\leq a \leq m_i} \prod^{j \in I}_{m_j < b \leq n_j} \zeta_{ji} \left( \frac {z_{jb}}{z_{ia}} \right)} \label{eqn:coproduct shuffle plus} \\
\Delta(F) = \sum_{\b0 \leq \bm \leq \bn} \frac {F(z_{i1},\dots , z_{im_i} \otimes z_{i,m_i+1}, \dots, z_{in_i}) \prod^{j \in I}_{1 \leq b \leq m_j} \ph^-_j(z_{jb})}{\prod^{i \in I}_{1\leq a \leq m_i} \prod^{j \in I}_{m_j < b \leq n_j} \zeta_{ij} \left( \frac {z_{ia}}{z_{jb}} \right)} \label{eqn:coproduct shuffle minus} 
\end{align}
for all $E \in \CV_{\bn}$, $F \in \CV_{-\bn}$. To make sense of the right-hand side of formulas \eqref{eqn:coproduct shuffle plus} and \eqref{eqn:coproduct shuffle minus}, we expand the denominator as a power series in the range $|z_{ia}| \ll |z_{jb}|$, and place all the powers of $z_{ia}$ to the left of the $\otimes$ sign and all the powers of $z_{jb}$ to the right of the $\otimes$ sign (for all $i,j \in I$, $1 \leq a \leq m_i$ and $m_j < b \leq n_j$). It is easy to see that the maps \eqref{eqn:tupsilon} respect the coproducts, and thus descend to 
\begin{equation}
	\label{eqn:coproduct u}
	\Delta : \Ug \rightarrow \Ug \vo \Ug \quad \text{ and } \quad \Delta : \Ul \rightarrow \Ul \vo \Ul 
\end{equation}
\begin{equation}
	\label{eqn:coproduct s}
	\Delta : \CS^{\geq} \rightarrow \CS^{\geq} \vo \CS^{\geq} \qquad \text{and} \quad \ \ \Delta : \CS^{\leq} \rightarrow \CS^{\leq} \vo \CS^{\leq} 
\end{equation}
which match each other under the isomorphisms $\Ug \cong \CS^{\geq}$ and $\Ul \cong \CS^{\leq}$.

\medskip 

\subsection{The pairing}
\label{sub:pairing}

Consider the following notation for all rational functions $G$:
\begin{equation}
	\label{eqn:contour integral}
	\int_{|z_1| \gg \dots \gg |z_n|} G(z_1,\dots,z_n)
\end{equation}
denotes the constant term in the expansion of $G$ as a power series in 
$$
\frac {z_2}{z_1}, \dots, \frac {z_n}{z_{n-1}}
$$
This notation is motivated by the fact that when $\BK = \BC$, one could compute the constant term as the contour integral of $G(z_1,\dots,z_n)\prod_{a=1}^n \frac{dz_a}{2\pi i z_a}$ over concentric circles centered at the origin. We define $\int_{|z_1| \ll \dots \ll |z_n|} G(z_1,\dots,z_n)$ analogously.

\medskip

\begin{definition}
\label{def:pair}
	
There exist bilinear pairings
\begin{align}
&\tUp \otimes \CV^- \xrightarrow{\langle \cdot, \cdot \rangle} \BK \label{eqn:pair} \\
&\CV^+ \otimes \tUm \xrightarrow{\langle \cdot, \cdot \rangle} \BK \label{eqn:pair opposite}
\end{align}
given for all $E \in \CV_{\bn}$, $F \in \CV_{-\bn}$ and all $i_1,\dots,i_n \in I$, $d_1,\dots,d_n \in \BZ$ by
\begin{align}
&\Big \langle e_{i_1,d_1} \cdots e_{i_n,d_n}, F \Big \rangle = \int_{|z_1| \gg \dots \gg |z_n|} \frac {z_1^{d_1}\dots z_n^{d_n} F(z_1,\dots,z_n)}{\prod_{1\leq a < b \leq n} \zeta_{i_bi_a} \left(\frac {z_b}{z_a} \right)} \label{eqn:pair formula} \\
&\Big \langle E, f_{i_1,d_1} \cdots  f_{i_n,d_n} \Big \rangle = \int_{|z_1| \ll \dots \ll |z_n|} \frac {z_1^{d_1}\dots z_n^{d_n} E(z_1,\dots,z_n)}{\prod_{1\leq a < b \leq n} \zeta_{i_ai_b} \left(\frac {z_a}{z_b} \right)} 
\label{eqn:pair formula opposite}
\end{align}
if $\bs^{i_1}+\dots +\bs^{i_n} = \bn$, and 0 otherwise. 
	
\end{definition}

\medskip

\noindent In the right-hand sides of \eqref{eqn:pair formula} and \eqref{eqn:pair formula opposite}, we implicitly identify
\begin{equation}
	\label{eqn:relabeling}
	z_a \quad \text{with} \quad z_{i_a\bullet_a}, \quad \forall a \in \{1,\dots, n\}
\end{equation}
where $\bullet_1,\dots,\bullet_n$ are the minimal positive integers such that $\bullet_a < \bullet_b$ if $a < b$ and $i_a = i_b$. Note that the pairings \eqref{eqn:pair} and \eqref{eqn:pair opposite}  are non-zero only on elements of opposite degree in $\zz \times \BZ$. It was shown in \cite{N Arbitrary} that there exist descended pairings
\begin{equation}
\label{eqn:pair descended}
\Up \otimes \CS^- \xrightarrow{\langle \cdot, \cdot \rangle} \BK \qquad \text{and} \qquad \CS^+ \otimes \Um \xrightarrow{\langle \cdot, \cdot \rangle} \BK
\end{equation}
which are non-degenerate and coincide under the isomorphisms \eqref{eqn:upsilon}.

\medskip 

\begin{remark}
\label{rem:pairing}

The isomorphisms \eqref{eqn:upsilon} allow us to rewrite \eqref{eqn:pair descended} as a pairing
\begin{equation}
\label{eqn:pair shuffle}
\CS^+ \otimes \CS^- \xrightarrow{\langle \cdot, \cdot \rangle} \BK
\end{equation} 
By definition, to compute $\langle E,F\rangle$ for any $E \in \CS^+$, $F \in \CS^-$, one needs to express $E$ as a linear combination of elements \eqref{eqn:spherical}, and then apply \eqref{eqn:pair formula}. It would be very interesting to obtain a formula for $\langle E, F \rangle$ that takes as input the Laurent polynomials $E$ and $F$ directly. In our experience, such a formula strongly depends on the particular choice of $(I,\BK,\zeta_{ij}(x))$, and is only known in the following cases:

\medskip

\begin{itemize}[leftmargin=*]
	
\item for quantum affinizations of Kac-Moody Lie algebras of simply-laced type in \cite{N Symmetric}

\medskip 	
	
\item for $K$-theoretic Hall algebras associated to doubled quivers in \cite{N Wheel}

\medskip

\item for BPS algebras associated to toric Calabi-Yau threefolds in \cite{N Reduced}
	
\end{itemize}

\medskip

\noindent Moreover, for quantum affine algebras associated to an arbitrary simple Lie algebra $\fg$, we will give such a formula for the pairing in Lemma \ref{lem:contour}.

\end{remark}

\medskip 

\subsection{Doubles}
\label{sub:doubles}

In what follows, we will use Sweedler's notation for coproducts
$$
\Delta(a) = a_1 \otimes a_2 \quad \text{instead of the more precise} \quad \Delta(a) = \sum_k a_{1,k} \otimes a_{2,k}
$$
Suppose we have bialgebras $A$ and $B$ over the field $\BK$. A pairing
\begin{equation}
\label{eqn:bialgebra pairing}
A \otimes B \xrightarrow{\langle \cdot, \cdot \rangle} \BK
\end{equation}
is called a bialgebra pairing if it satisfies
\begin{align}
&\Big \langle a,b'b'' \Big \rangle = \Big \langle \Delta(a), b' \otimes b'' \Big \rangle = \langle a_1,b'\rangle \langle a_2,b''\rangle\label{eqn:bialgebra 1} \\
&\Big \langle a' a'' ,b \Big \rangle = \Big \langle a' \otimes a'', \Delta^{\text{op}}(b) \Big \rangle = \langle a',b_2\rangle \langle a'',b_1\rangle\label{eqn:bialgebra 2} 
\end{align}
for all $a,a',a'' \in A$ and $b,b',b'' \in B$ ($\Delta^{\text{op}}$ is the opposite coproduct). The notion above is sometimes called a skew-bialgebra pairing, as it identifies the coproduct on $A$ with the dual of the product on $B$ and the opposite coproduct on $B$ with the dual of the product on $A$. Whenever we have bialgebras $A$ and $B$ with a bialgebra pairing \eqref{eqn:bialgebra pairing} as above, the \emph{Drinfeld double} construction makes the vector space
\begin{equation}
\label{eqn:drinfeld double}
D = A \otimes B
\end{equation}
into a bialgebra which contains $A = A \otimes 1$ and $B = 1 \otimes B$ as sub-bialgebras. Indeed, the multiplication in $D$ is governed by the relation
\begin{equation}
\label{eqn:drinfeld double relation}
a_1 b_1 \langle a_2,b_2 \rangle = \langle a_1,b_1\rangle b_2a_2, \qquad \forall a \in A = A \otimes 1, \ b \in B = 1 \otimes B
\end{equation}
Practically, \eqref{eqn:drinfeld double relation} allows one to take an arbitrary product of $a$'s and $b$'s and convert it into a linear combination of products of the form $ab$, which are elements of \eqref{eqn:drinfeld double}. 

\medskip

\begin{remark}
\label{rem:hopf}

All bialgebras considered in the present paper are actually Hopf, but we do not recall the antipode $S$ because we have no useful formula for it. Moreover, all bialgebra pairings are actually Hopf pairings, and formula \eqref{eqn:drinfeld double} is equivalent to the more conventionally written formulas for multiplication in Drinfeld doubles
\begin{align}
&ab = \langle a_1,b_1 \rangle b_2 a_2	\langle a_3, S(b_3) \rangle \label{eqn:dd1} \\
&ba = \langle a_1,S(b_1) \rangle a_2 b_2	\langle a_3, b_3 \rangle \label{eqn:dd2}
\end{align} 
$\forall a \in A$, $b \in B$. However, for all algebras considered in the present paper, relation \eqref{eqn:drinfeld double relation} will be enough to reorder arbitrary products of $a$'s and $b$'s, and we do not need \eqref{eqn:dd1} and \eqref{eqn:dd2}. This is because for all elements $x$ of the algebras considered in the present paper, we have
$$
\Delta(x) = x \otimes (\text{Cartan element}) + \dots + (\text{Cartan element}) \otimes x
$$
where the ellipsis denotes terms of horizontal degree strictly contained between $0$ and $\emph{hdeg} x$.

\end{remark} 

\medskip

\noindent If the pairing \eqref{eqn:bialgebra pairing} is non-degenerate, then it induces injective maps
\begin{equation}
\label{eqn:hooks}
A \hookrightarrow B^* \qquad \text{and} \qquad B \hookrightarrow A^*
\end{equation}
If $A$ and $B$ are finite-dimensional over $\BK$, a non-degenerate pairing is perfect, i.e. the inclusions in \eqref{eqn:hooks} are actually isomorphisms. The algebras considered in the present paper are certainly not finite-dimensional, but we will often encounter algebras with finite-dimensional graded summands. In this case, we make the convention that all our duals will always be considered in the graded sense, i.e.
$$
\text{if } A = \bigoplus_{(\bn,d) \in \zz \times \BZ} A_{\bn,d}, \quad \text{then we define } A^* = \bigoplus_{(\bn,d) \in \zz \times \BZ} A_{\bn,d}^*
$$
In this case, a non-degenerate pairing (which respects the grading, by which we mean that it only pairs non-trivially elements of opposite degrees) between graded bialgebras with finite dimensional graded summands is always perfect. 

\medskip

\subsection{Universal $R$-matrices}
\label{sub:universal}

One of the big motivations for the introduction of Drinfeld doubles is that they are endowed with a \emph{universal $R$-matrix}
$$
\CR \in D \otimes D
$$
which satisfies the properties
\begin{equation}
\label{eqn:universal 1}
\CR \cdot \Delta(-) = \Delta^{\text{op}}(-) \cdot \CR
\end{equation}
\begin{equation}
\label{eqn:universal 2}
\left( \Delta \otimes \text{Id}_D \right)(\CR) = \CR_{13} \CR_{23} 
\end{equation}
\begin{equation}
\label{eqn:universal 3}
\left( \text{Id}_D \otimes \Delta \right)(\CR) = \CR_{13} \CR_{12}
\end{equation}  
(in the right-hand side of the expressions above, we write $\CR_{12}, \CR_{13}, \CR_{23}$ for the tensors in $D \otimes D \otimes D$ which are equal to $\CR$ on the two indices in the subscript and 1 on the third index). Indeed, we have the following general result.

\medskip

\begin{lemma}
\label{lem:universal}

If the pairing \eqref{eqn:bialgebra pairing} is perfect, then its canonical tensor
\begin{equation}
\label{eqn:universal} 
\CR  = \sum_k a_k \otimes b_k \in A \otimes B \subset D \otimes D
\end{equation}
(with respect to any dual bases $\{a_k\} \subset A$ and $\{b_k\} \subset B$) is a universal $R$-matrix.

\end{lemma}

\medskip

\begin{proof} The defining property of the canonical tensor is that
\begin{equation}
\label{eqn:defining}
\langle - \otimes a, \CR \rangle = a \quad \text{and} \quad \langle \CR, b \otimes - \rangle = b
\end{equation}
for all $a \in A$ and $b \in B$. Because formula \eqref{eqn:universal 1} is multiplicative in $-$, it suffices to prove it for $a \in A$ and $b \in B$ separately, i.e.
\begin{align}
&\sum_k a_k a_1 \otimes b_k a_2 = \sum_k a_2 a_k \otimes a_1 b_k \label{eqn:oi 1} \\ 
&\sum_k a_k b_1 \otimes b_k b_2 = \sum_k b_2 a_k \otimes b_1 b_k \label{eqn:oi 2}
\end{align}
(we write $a_1,a_2,b_1,b_2$ for the tensor factors of the coproducts of $a,b$, respectively, and they are not to be confused with the tensor factors $a_k,b_k$ of $\CR$). By the non-degeneracy of the pairing, to prove the formulas above it suffices to show that the two sides of each equation have the same pairing with an element of the form $b \otimes -$ and $- \otimes a$, respectively, for arbitrary $a \in A$, $b\in B$. Thus, we have
\begin{multline*}
\langle \text{LHS of \eqref{eqn:oi 1}}, b \otimes - \rangle = \sum_k \langle a_1,b_1\rangle \langle a_k,b_2\rangle b_ka_2 = \langle a_1,b_1\rangle b_2a_2 = \\ = a_1b_1 \langle a_2,b_2\rangle = \sum_k \langle a_2, b_2\rangle \langle a_k,b_1\rangle a_1b_k = \langle \text{RHS of \eqref{eqn:oi 1}}, b \otimes - \rangle 
\end{multline*}
\begin{multline*}
\langle \text{LHS of \eqref{eqn:oi 2}}, - \otimes a \rangle = \sum_k a_k b_1 \langle a_1,b_k\rangle \langle a_2,b_2\rangle = a_1b_1 \langle a_2,b_2\rangle = \\ = \langle a_1,b_1\rangle b_2a_2 = \sum_k b_2a_k \langle a_1, b_1\rangle \langle a_2,b_k\rangle = \langle \text{RHS of \eqref{eqn:oi 2}}, - \otimes a \rangle 
\end{multline*}
with the first and last equalities in each equation due to \eqref{eqn:bialgebra 1} and \eqref{eqn:bialgebra 2} and the middle equalities due to \eqref{eqn:drinfeld double relation}. Properties \eqref{eqn:universal 2} and \eqref{eqn:universal 3} are equivalent to
\begin{align} 
&\sum_k a_{k,1} \otimes a_{k,2} \otimes b_k = \sum_{k,k'} a_k \otimes a_{k'} \otimes b_k b_{k'} \label{eqn:universal 2 equiv} \\
&\sum_k a_k \otimes b_{k,1} \otimes b_{k,2} = \sum_{k,k'} a_{k'} a_k \otimes b_{k} \otimes b_{k'}  \label{eqn:universal 3 equiv}
\end{align} 
By the non-degeneracy of the pairing, to prove the formulas above it suffices to show that the two sides of each equation have the same pairing with an arbitrary element of the form $b' \otimes b'' \otimes a$ and $b \otimes a'' \otimes a'$, respectively, where $a,a',a'' \in A$, $b,b',b'' \in B$. These pairings can then be calculated using \eqref{eqn:defining}. Thus, we have
$$
\langle \text{LHS of \eqref{eqn:universal 2 equiv}}, b' \otimes b'' \otimes a \rangle = \langle \Delta(a), b' \otimes b'' \rangle = \langle a, b'b'' \rangle = \langle \text{RHS of \eqref{eqn:universal 2 equiv}},  b' \otimes b'' \otimes a \rangle
$$
$$
\langle \text{LHS of \eqref{eqn:universal 3 equiv}}, b \otimes a'' \otimes a' \rangle = \langle a'' \otimes a', \Delta(b) \rangle = \langle a'a'', b \rangle = \langle \text{RHS of \eqref{eqn:universal 3 equiv}},  b \otimes a'' \otimes a' \rangle
$$
with the middle equalities being \eqref{eqn:bialgebra 1} and \eqref{eqn:bialgebra 2}. This establishes \eqref{eqn:universal 2} and \eqref{eqn:universal 3}. \end{proof}

\medskip 

\subsection{The quantum loop algebra}
\label{sub:quantum loop algebra}

We conclude this Section by assembling all the ingredients above into the hodge-podge that is the definition of quantum loop algebras. We work with any $(I,\BK,\zeta_{ij}(x))$, as in Subsection \ref{sub:algebras intro}. Recall the pairing
\begin{equation}
\label{eqn:quantum pairing 1}
\Up \otimes \Um \xrightarrow{\langle \cdot, \cdot \rangle} \BK
\end{equation}
induced by \eqref{eqn:pair descended} and the isomorphisms \eqref{eqn:upsilon}. It is straightforward to show that one can extend \eqref{eqn:quantum pairing 1} to a bialgebra pairing
\begin{equation}
\label{eqn:quantum pairing 2}
\Ug \otimes \Ul \xrightarrow{\langle \cdot, \cdot \rangle} \BK
\end{equation}
by requiring for all $i,j \in I$ that
\begin{equation}
\label{eqn:pairing h's}
\Big \langle \ph^+_i(z), \ph^-_j(w) \Big \rangle = \frac {\zeta_{ij}\left(\frac zw \right)}{\zeta_{ji}\left(\frac wz \right)}
\end{equation}
(and that all pairings between $\ph^\pm$ on one hand and $e,f$ on the other hand vanish).

\medskip

\begin{definition}
\label{def:quantum}

The \emph{quantum loop algebra} is the Drinfeld double
\begin{equation}
\label{eqn:quantum 1}
\U = \Ug \otimes \Ul
\end{equation}
defined with respect to the coproducts \eqref{eqn:coproduct u} and the pairing \eqref{eqn:quantum pairing 2}.

\end{definition}

\medskip

\noindent Explicitly, $\U$ has generators $\{e_{i,d}, f_{i,d}, \ph^+_{i,d'}, \ph^-_{i,d'} \}_{i \in I, d \in \BZ, d' \geq 0}$ modulo relations
\begin{equation}
\label{eqn:rel quantum 1}
e_i(z) e_j(w) \zeta_{ji} \left(\frac wz\right) = e_j(w) e_i(z) \zeta_{ij} \left( \frac zw \right)
\end{equation}
\begin{equation}
\label{eqn:rel quantum 2}
 f_i(z) f_j(w) \zeta_{ij} \left( \frac zw \right) = f_j(w) f_i(z) \zeta_{ji} \left(\frac wz\right) 
\end{equation}
\begin{equation}
\label{eqn:rel quantum 3}
\Big(\text{any element of Ker }\widetilde{\Upsilon}^\pm \Big) = 0
\end{equation} 
\begin{equation}
\label{eqn:rel quantum 4}
\left[\ph^+_{i}(z), \ph^+_{j}(w)\right] = \left[\ph^+_{i}(z), \ph^-_{j}(w)\right] = \left[\ph^-_{i}(z), \ph^-_{j}(w)\right] = 0
\end{equation} 
\begin{align}
&\ph^+_i(z) e_j(w) = e_j(w) \ph^+_i(z) \frac {\zeta_{ij} \left(\frac zw \right)}{\zeta_{ji} \left(\frac wz \right)} \label{eqn:rel quantum 5} \\
&\ph^-_i(z) e_j(w) = e_j(w) \ph^-_i(z) \frac {\zeta_{ij} \left(\frac zw \right)}{\zeta_{ji} \left(\frac wz \right)} \label{eqn:rel quantum 6} \\
&f_j(w)\ph^+_i(z)  = \ph^+_i(z)  f_j(w) \frac {\zeta_{ij} \left(\frac zw \right)}{\zeta_{ji} \left(\frac wz \right)} \label{eqn:rel quantum 7} \\
&f_j(w)\ph^-_i(z)  = \ph^-_i(z)  f_j(w) \frac {\zeta_{ij} \left(\frac zw \right)}{\zeta_{ji} \left(\frac wz \right)} \label{eqn:rel quantum 8}
\end{align}
\begin{equation}
\label{eqn:rel quantum 9}
\Big[e_{i}(z), f_{j}(w)\Big] = \delta_{ij} \delta \left(\frac zw \right) \Big( \ph^-_j(w) - \ph^+_i(z)\Big)
\end{equation} 
for all $i,j \in I$, where $\delta(x) = \sum_{d \in \BZ} x^d$ is a formal series. Similarly, we define
\begin{equation}
\label{eqn:double shuffle}
\CS = \CS^{\geq} \otimes \CS^{\leq}
\end{equation}
as a Drinfeld double with respect to the coproducts \eqref{eqn:coproduct s} and the bialgebra pairing
\begin{equation}
\label{eqn:pair shuffle extended}
\CS^{\geq} \otimes \CS^{\leq} \xrightarrow{\langle \cdot, \cdot \rangle} \BK
\end{equation} 
induced by \eqref{eqn:pair shuffle} and \eqref{eqn:pairing h's}. Clearly, we have $\CS \cong \U$ and so we henceforth identify 
\begin{equation}
\label{eqn:identify}
\CS = \U
\end{equation} 

\medskip

\begin{remark} 
\label{rem:central extension}
	
It is possible to enlarge $\U$ by a central element $c$ which governs the failure of $\ph^+_{i,d}$ and $\ph^-_{j,d'}$ to commute. To do so, one keeps the algebras $\U^{\geq}$ and $\U^{\leq}$ as we defined them, but appropriately inserts powers of $c$ in the definition of their coproduct. The resulting Drinfeld double is called centrally extended $\U$. However, since in the present paper we are interested only in modules on which the central element $c$ acts trivially, we will not bother with writing down the central extension.
	
\end{remark}

\medskip

\subsection{Cartan subalgebras}
\label{sub:cartan}

The subalgebra 
\begin{equation}
\label{eqn:loop cartan}
\BK[\ph_{i,d}^+, \ph_{i,d}^-]_{i \in I, d \geq 0} \subset \U
\end{equation}
is called the \emph{loop Cartan subalgebra}, and its subalgebras
$$
\BK[\ph_{i,d}^+]_{i \in I, d \geq 0} \qquad \text{and} \qquad \BK[\ph_{i,d}^-]_{i \in I, d \geq 0}
$$
are called the positive and negative loop Cartan, respectively. The leading terms $\ph_{i,0}^\pm$ will act by non-zero elements of $\BK$ on all (homogeneous) elements of all modules considered in the present paper, and so it makes sense to assume
\begin{equation}
\label{eqn:power series}
\ph_{i}^\pm(z) = \kappa_i^\pm \exp \left(\sum_{d=1}^{\infty} \frac {p_{i,\pm d}}{d z^{\pm d}}\right)  
\end{equation}
In other words, we simply assume that all the $\ph_{i,d}^\pm$ are multiples of $\ph_{i,0}^\pm = \kappa_i^\pm$, which could be ensured (for example) by inverting the latter elements. The subalgebra
$$
\BK[\kappa_i^{+}, \kappa_i^{-}]_{i \in I} \subset \U
$$
is called the (finite) \emph{Cartan subalgebra}. An easy consequence \footnote{Strictly speaking, \eqref{eqn:coproduct k}-\eqref{eqn:coproduct p} follow from \eqref{eqn:coproduct h} and \eqref{eqn:pairing k}-\eqref{eqn:pairing p} follow from \eqref{eqn:pairing h's} only if we assume that $\kappa_i^+$ and $\kappa_i^-$ are invertible. We will tacitly tolerate this slight abuse in the logical flow of our constructions, and henceforth assume that \eqref{eqn:coproduct k}-\eqref{eqn:pairing p} hold as stated.} of \eqref{eqn:coproduct h} is
\begin{equation}
\label{eqn:coproduct k}
\Delta(\kappa_i^\pm) = \kappa_i^\pm \otimes \kappa_i^\pm 
\end{equation}
\begin{equation}
\label{eqn:coproduct p}
\Delta(p_{i,\pm d}) = p_{i,\pm d} \otimes 1 + 1 \otimes p_{i,\pm d}
\end{equation}
for all $i \in I$ and $d \geq 1$. Moreover, \eqref{eqn:pairing h's} implies
\begin{equation}
\label{eqn:pairing k}
\langle \kappa_i^+, \kappa_j^- \rangle = \gamma_{\bs^i,\bs^j}
\end{equation}
\begin{equation}
\label{eqn:pairing p}
\langle p_{i,d}, p_{j,-d} \rangle = d \alpha_{ij}^{(d)}
\end{equation} 
for all $i,j \in I$ and $d \geq 1$, where $\alpha_{ij}^{(d)} \in \BK$ are defined by the power series expansion
\begin{equation}
\label{eqn:notation}
\frac {\zeta_{ij}(x)}{\zeta_{ji}\left(x^{-1}\right)} = \gamma_{\bs^i,\bs^j} \exp \left(\sum_{d=1}^{\infty} \frac {\alpha_{ij}^{(d)}}{dx^d} \right)
\end{equation}
(see \eqref{eqn:scalars} for the definition of $\gamma_{\bs^i,\bs^j}$). All pairings between $\kappa$'s and $p$'s other than \eqref{eqn:pairing k} and \eqref{eqn:pairing p} vanish for degree reasons. We will call the zeta functions \eqref{eqn:magnitude} \emph{fully factored} if their numerators are fully factored into linear terms, i.e.
\begin{equation}
\label{eqn:fully factored}
\zeta_{ij}(x) = \frac {c_{ij} x^{-\#_{ji}}(x-s_{ij|1})\dots (x-s_{ij|\#_{ij}+\#_{ji}+\delta_{ij}})}{(x-1)^{\delta_{ij}}}
\end{equation}
for all $i,j \in I$ and various $s_{ij|b} \in \BK^*$ (the fully factored condition is automatically satisfied if $\BK$ is algebraically closed). In this case, an easy consequence of \eqref{eqn:notation} is
\begin{equation}
\label{eqn:consequence}
\alpha_{ij}^{(d)} = \sum_{b=1}^{\#_{ij}+\#_{ji}+\delta_{ij}} (s_{ji|b}^{-d} - s_{ij|b}^d)  
\end{equation}

\bigskip

\section{Slopes and coproducts}
\label{sec:slopes}

\medskip

\noindent We begin by recalling the discussion of slopes in shuffle algebras and quantum loop algebras, following \cite{FHHSY, N Shuffle, N R-matrix}, which will allow us to construct the subalgebras
$$
\CA^{\geq \bp}, \CA^{\leq \bp} \subset \U
$$
for any $\bp \in \rr$. We then define topological coproducts $\Delta_{\bp}$ on $\CA^{\geq \bp}$ and $\CA^{\leq \bp}$, and show that $ \U$ is their Drinfeld double with respect to a pairing induced by \eqref{eqn:quantum pairing 2}. We call $\Delta_{\bp}$ thus defined ``new new" coproducts, and in the particular case of quantum affine algebras \eqref{eqn:quantum affine intro}, we show that $\Delta_{\b0}$ matches the Drinfeld-Jimbo coproduct. We also deduce various factorizations of universal $R$-matrices. 

\medskip 

\subsection{Slopes}
\label{sub:slope}

Our references for slope subalgebras are \cite{N R-matrix} and \cite{N Char}; while both these works are written for particular choices of $(I,\BK,\zeta_{ij}(x))$, all the results therein actually hold in our current level of generality. Recall the notation \eqref{eqn:dot product}-\eqref{eqn:euler form}.

\medskip

\begin{definition}
\label{def:slopes}

For any $\bp \in \rr$, we will say that

\medskip

\begin{itemize}[leftmargin=*]

\item $E \in \CS_\bn$ has slope $\geq \bp$ if the following limit is finite for all $\b0 < \bm \leq \bn$
\begin{equation}
\label{eqn:slope e geq}
\lim_{\xi \rightarrow 0} \frac {E(\xi z_{i1},\dots,\xi z_{im_i},z_{i,m_{i+1}},\dots,z_{in_i})}{\xi^{\bp \cdot \bm - \langle \bn - \bm, \bm \rangle}}
\end{equation}

\medskip
	
\item $E \in \CS_\bn$ has slope $\leq \bp$ if the following limit is finite for all $\b0 < \bm \leq \bn$
\begin{equation}
\label{eqn:slope e leq}
\lim_{\xi \rightarrow \infty} \frac {E(\xi z_{i1},\dots,\xi z_{im_i},z_{i,m_{i+1}},\dots,z_{in_i})}{\xi^{\bp \cdot \bm + \langle \bm, \bn - \bm \rangle}}
\end{equation}
	
\medskip
	
\item $F \in \CS_{-\bn}$ has slope $\leq \bp$ if the following limit is finite for all $\b0 < \bm \leq \bn$
\begin{equation}
\label{eqn:slope f leq}
\lim_{\xi \rightarrow 0} \frac {F(\xi z_{i1},\dots,\xi z_{im_i},z_{i,m_{i+1}},\dots,z_{in_i})}{\xi^{-\bp \cdot \bm - \langle \bn - \bm, \bm \rangle}}
\end{equation}
	
\medskip
	
\item $F \in \CS_{-\bn}$ has slope $\geq \bp$ if the following limit is finite for all $\b0 < \bm \leq \bn$
\begin{equation}
\label{eqn:slope f geq}
\lim_{\xi \rightarrow \infty} \frac {F(\xi z_{i1},\dots,\xi z_{im_i},z_{i,m_{i+1}},\dots,z_{in_i})}{\xi^{-\bp \cdot \bm + \langle \bm, \bn - \bm \rangle}}
\end{equation}
	
\end{itemize}

\end{definition}

\medskip

\noindent If moreover the limits in \eqref{eqn:slope e geq}-\eqref{eqn:slope f geq} are all 0, then we will say that $E$ and $F$ therein have slopes $<\bp$ or $>\bp$, respectively. We will write
\begin{equation}
\label{eqn:slopes}
\CS^\pm_{\geq \bp}, \ \CS^\pm_{\leq \bp}, \ \CS^\pm_{> \bp}, \ \CS^\pm_{< \bp}
\end{equation}
for the subsets of elements of $\CS^\pm$ of slope $\geq \bp$, $\leq \bp$, $> \bp$, $< \bp$, respectively. It is elementary to show that all the sets which appear in \eqref{eqn:slopes} are actually subalgebras of $\CS^\pm$ with respect to the shuffle product. As proved in \cite[Proposition 3.4]{N R-matrix}, the coproduct of Subsection \ref{sub:coproduct} interacts with the subalgebras defined above as follows
\begin{align}
&\Delta (\CS^+_{\geq \bp} ) \subset \CS^\geq_{\geq \bp} \vo \CS^+ \label{eqn:interact 1} \\
&\Delta (\CS^+_{\leq \bp} ) \subset \CS^\geq \vo \CS^+_{\leq \bp} \label{eqn:interact 2} \\
&\Delta (\CS^-_{\leq \bp} ) \subset \CS^-_{\leq \bp} \vo \CS^\leq \label{eqn:interact 3} \\
&\Delta (\CS^-_{\geq \bp} ) \subset \CS^- \vo \CS^\leq_{\geq \bp} \label{eqn:interact 4}
\end{align}
where $\CS^\geq_{\geq \bp}$ and $\CS^\leq_{\geq \bp}$ are obtained from $\CS^+_{\geq \bp}$ and $\CS^-_{\geq \bp}$ by adding loop Cartan elements as in \eqref{eqn:sg} and \eqref{eqn:sl}.

\medskip

\subsection{Slope subalgebras}
\label{sub:slope subalgebras}

If a shuffle element $X$ simultaneously has slope $\leq \bp$ and $\geq \bp$, then $\vdeg X = \bp \cdot \hdeg X$. The set of such elements is denoted by
\begin{equation}
	\label{eqn:slope subalgebra}
	\CB_{\bp}^\pm = \bigoplus_{\bn \in \nn \text{ s.t. } \bp \cdot \bn \in \BZ} \CB_{\bp|\pm \bn}
\end{equation}
and is called a \emph{slope subalgebra}. Recall the notation $\kappa^\pm_i = \ph_{i,0}^\pm$. The vector spaces
\begin{align}
&\CB_{\bp}^\geq = \CB_{\bp}^+ \otimes \BK[\kappa^+_i]_{i \in I} \label{eqn:enlarged slope 1} \\
&\CB_{\bp}^\leq = \BK[\kappa^-_i]_{i \in I}  \otimes \CB_{\bp}^- \label{eqn:enlarged slope 2}
\end{align}
are subalgebras of $\CS^\geq$ and $\CS^\leq$, respectively, as long as we impose the following commutation relations derived from taking the leading order terms of \eqref{eqn:ssg}-\eqref{eqn:ssl}
\begin{equation}
\label{eqn:cartan commutation plus}
\kappa^+_i X = X \kappa^+_i \gamma_{\bs^i, \hdeg X}
\end{equation}
\begin{equation}
\label{eqn:cartan commutation minus}
\kappa^-_i X = X \kappa^-_i \gamma_{\hdeg X, \bs^i}^{-1}
\end{equation}
for any $X \in \CS$ (recall that the scalars $\gamma$ are defined in \eqref{eqn:scalars}). We can further make the vector spaces \eqref{eqn:enlarged slope 1} and \eqref{eqn:enlarged slope 2} into bialgebras, using $\Delta_{\bp}(\kappa_i^\pm) = \kappa_i^\pm \otimes \kappa_i^\pm$ and 
\begin{align}
	&\Delta_\bp(E) = \sum_{\b0 \leq \bm \leq \bn} (\kappa^+_{\bn-\bm} \otimes 1) (\text{value of the limit \eqref{eqn:slope e geq}}) \label{eqn:coproduct slope plus} \\
	&\Delta_\bp(F) = \sum_{\b0 \leq \bm \leq \bn} (\text{value of the limit \eqref{eqn:slope f leq}}) (1 \otimes \kappa^-_{\bm}) \label{eqn:coproduct slope minus}
\end{align} 
for all $E \in \CB^+_\bp$ and $F \in \CB^-_\bp$, where we write $\kappa^\pm_{\bm} = \prod_{i \in I} (\kappa^\pm_i)^{m_i}$ for all $\bm \in \nn$.

\medskip

\begin{remark}
\label{rem:double}

With respect to the bialgebra structure above and the pairing 
\begin{equation}
\label{eqn:pairing restricted}
\CB^{\geq}_{\bp} \otimes \CB^{\leq}_{\bp} \xrightarrow{\langle \cdot, \cdot \rangle} \BK 
\end{equation}
obtained by restricting \eqref{eqn:pair shuffle extended} to the slope subalgebras, the Drinfeld double 
\begin{equation}
\label{eqn:double slope}
\CB_\bp = \CB_\bp^\geq \otimes \CB_\bp^\leq 
\end{equation}
is a subalgebra of $\CS$, although not a sub-bialgebra. As the author learned from Andrei Okounkov and Olivier Schiffmann, in many cases $\CB_{\bp}$ is expected to be a quantum Borcherds algebra, so not much can be said about it beside its graded dimension. 

\end{remark}

\medskip 

\subsection{Infinite slope}
\label{sub:infinite slope 1}

We will also define slope $\binfty$ versions of slope subalgebras, 
\begin{equation}
\label{eqn:slope infinity}
\CB^{\geq}_{\binfty} = \BK[\ph^+_{i,d}]_{i \in I, d\geq 0} \qquad \text{and} \qquad \CB^{\leq}_{\binfty} = \BK[\ph^-_{i,d}]_{i \in I, d\geq 0} 
\end{equation}
\begin{equation}
\label{eqn:slope infinity 2}
\CB^{+}_{\binfty} = \BK[p_{i,d}]_{i \in I, d> 0} \qquad \text{and} \qquad \CB^{-}_{\binfty} = \BK[p_{i,-d}]_{i \in I, d> 0} 
\end{equation}
with the notation as in Subsection \ref{sub:cartan}. The coproduct $\Delta_{\binfty}$ is then defined as restriction of $\Delta$ to the above commutative subalgebras. We assume that
\begin{equation}
\label{eqn:pairing slope infinity}
\CB^\geq_{\binfty} \otimes \CB^\leq_{\binfty} \xrightarrow{\langle \cdot, \cdot \rangle} \BK
\end{equation}
is non-degenerate, which boils down to the condition that the $|I| \times |I|$ matrix with entries RHS of \eqref{eqn:pairing p} is invertible for all $d$, and that the scalars \eqref{eqn:scalars} are linearly independent in both $\bm \in \nn$ and $\bn \in \nn$ separately. If the latter condition fails, all is not lost; the usual solution is to enlarge the finite Cartan subalgebra, i.e. to add new symbols $\kappa_j^\pm$ which interact with $X \in \CS$ according to \eqref{eqn:cartan commutation plus}-\eqref{eqn:cartan commutation minus}, for newly chosen scalars 
$$
\Big\{ \gamma_{\bs^j,\bn} \Big\}_{\bn \in \zz} \qquad \text{and} \quad \Big\{ \gamma_{\bn,\bs^j}^{-1} \Big\}_{\bn \in \zz}
$$
that are additive in $\bn$ and sufficiently generic.

\medskip

\subsection{Factorizations}
\label{sub:factorizations}

Slope subalgebras are important because they are the building blocks of shuffle algebras, as we will now recall. A parameterized curve 
$$
\BR \rightarrow \rr, \quad t \mapsto \bp(t) = (p_i(t))_{i \in I}
$$
will be called \emph{catty-corner} if
\begin{equation}
	\label{eqn:catty-corner 1}
	t_1 < t_2 \quad \text{implies} \quad p_i(t_1) < p_i(t_2) 
\end{equation}
and
\begin{equation}
	\label{eqn:catty-corner 2}
	\lim_{t \rightarrow \pm \infty} p_i(t) = \pm \infty
\end{equation}
for all $i \in I$. Note that condition \eqref{eqn:catty-corner 1} is stronger than $\bp(t_1) < \bp(t_2)$, because the latter just means $\bp(t_1) \leq \bp(t_2)$ and $\bp(t_1) \neq \bp(t_2)$. We may also define catty-corner curves on bounded intervals of $\BR$, in which case condition \eqref{eqn:catty-corner 2} would be dropped.

\medskip

\begin{proposition}
\label{prop:factor}
	
(\cite{N Char})	For any catty-corner curve $\bp(t)$, we have an isomorphism
\begin{equation}
\label{eqn:factorization 1}	
\bigotimes^{\rightarrow}_{t \in \BR} \CB^\pm_{\bp(t)} \xrightarrow{\sim} \CS^\pm   
\end{equation}
given by multiplication, where $\rightarrow$ means that we take the tensor product in increasing order of $t$. This isomorphism preserves the pairing \eqref{eqn:pair shuffle}, in the sense that
\begin{equation}
	\label{eqn:pair slopes basic}
	\left \langle \prod_{t \in \BR}^{\rightarrow} E_t, \prod_{t \in \BR}^{\rightarrow} F_t \right \rangle = \prod_{t \in \BR} \langle E_t, F_t \rangle 
\end{equation}
for all $\{E_t \in \CB^+_{\bp(t)}, F_t \in \CB^-_{\bp(t)}\}_{t \in \BR}$ (almost all of which are 1). 

\medskip

\end{proposition}

\medskip

\noindent By the same token as in Proposition \ref{prop:factor}, we have
\begin{equation}
	\label{eqn:factorization 2}
	\bigotimes^{\rightarrow }_{t \in [t_1,t_2]} \CB^\pm_{\bp(t)} \xrightarrow{\sim} \CS^\pm_{\geq \bp(t_1)} \cap  \CS^\pm_{\leq \bp(t_2)} 
\end{equation}
for any $t_1 \leq t_2$, as well as the natural analogues of \eqref{eqn:factorization 2} when some endpoints of the intervals may be open instead of closed; while the proof of these statements given in \cite[Proposition 3.14]{N R-matrix} is written in the particular case of $K$-theoretic Hall algebras of double quivers, the argument therein is completely general. Thus, we obtain a host of factorizations of the so-called wedge subalgebras \footnote{We also make the convention that
	\begin{align} 
		&\CS^+_{[\bp,\binfty]}  := \CS^+_{\geq \bp} \otimes \CB^{\geq}_{\binfty} \label{eqn:extended plus} \\ 
		&\CS^-_{[\bp,\binfty]}  := \CS^-_{\geq \bp} \otimes \CB^{\leq}_{\binfty} \label{eqn:extended minus}
	\end{align}
	i.e. we include the finite Cartan elements $\kappa_i^\pm$ in the algebras $\CS^\pm_{[\bp,\binfty]}$, alongside $\CB_{\binfty}^{\pm} = \BK[p_{i,\pm d}]$.}
$$
\CS^\pm_{[\bp^1 , \bp^2]} = \CS^\pm_{\geq \bp^1} \cap  \CS^\pm_{\leq \bp^2}
$$
for any $\bp^1 \leq \bp^2$ in $\rr$, as well as their natural analogues when some endpoints of the intervals may be open instead of closed (the word ``host" is due to the fact that there is a great degree of freedom in choosing a catty-corner curve joining $\bp^1$ and $\bp^2$). The following result is easy, so we leave it as an exercise to the reader.

\medskip

\begin{lemma}
\label{lem:finite-dimensional}

The subalgebras 
$$
\CB^\pm_\bp, \ \CS^\pm_{ \geq \bp }, \ \CS^\pm_{\leq \bp}, \ \CS^\pm_{[\bp^1 , \bp^2]}
$$
have finite-dimensional $\zz \times \BZ$ graded summands for all $\bp$ and $\bp^1 \leq \bp^2$ in $\rr$.

\end{lemma}

\medskip

\noindent In a nutshell, the condition that a shuffle element $X$ lies in either $\CS^\pm_{\geq \bp}$ or $\CS^\pm_{\leq \bp}$ imposes a bound (either lower or upper) on the degrees of all of its variables; then fixing $(\hdeg X,\vdeg X)$ leads to a finite number of monomials that may appear in $X$. As a consequence of Lemma \ref{lem:finite-dimensional}, the restrictions of the pairing \eqref{eqn:pair shuffle} to
\begin{equation}
\label{eqn:perfect pairing 1}
\CB^+_{\bp} \otimes \CB^-_{\bp} \rightarrow \BK, \qquad \CS^+_{[\bp^1 , \bp^2]} \otimes \CS^-_{[\bp^1 , \bp^2]} \rightarrow \BK 
\end{equation}
\begin{equation}
	\label{eqn:perfect pairing 2}
\CS^+_{\geq \bp} \otimes \CS^-_{\geq \bp} \rightarrow \BK, \qquad \qquad \CS^+_{\leq \bp} \otimes \CS^-_{\leq \bp} \rightarrow \BK 
\end{equation}
are perfect for all $\bp$ and $\bp^1 \leq \bp^2 \in \rr$; this is because Proposition \ref{prop:factor} implies that the pairings above inherit their non-degeneracy from that of \eqref{eqn:pair shuffle extended}, and then we invoke the last sentence of Subsection \ref{sub:doubles} to go from non-degenerate to perfect.

\medskip

\subsection{More factorizations}
\label{sub:more}
	
As a consequence of Proposition \ref{prop:factor}, for any $\bp \in \rr$ we have an isomorphism induced by multiplication 
\begin{equation}
\label{eqn:two factor}
\CS^{\pm}_{(-\binfty,\bp)} \otimes \CS^{\pm}_{[\bp,\binfty]} \xrightarrow{\sim} \CS^\geq 
\end{equation}
(recall \eqref{eqn:extended plus}-\eqref{eqn:extended minus}). Formula \eqref{eqn:two factor} allows us to uniquely write any $E \in \CS^\geq$, $F \in \CS^\leq$ as
$$
E = \sum_{k} c^+_k \cdot A^+_k * B^+_k \quad \text{and} \quad F = \sum_{k} c^-_k \cdot A^-_k * B^-_k 
$$
for any bases $\{A^\pm_k\}$ of $\CS^{\pm}_{(-\binfty,\bp)}$ and $\{B_{k}^\pm\}$ of $\CS^\pm_{[\bp,\binfty]}$, and various coefficients $c^\pm_k \in \BK$. Moreover, given elements $E$ and $F$ as above, formula \eqref{eqn:pair slopes basic} implies that
\begin{equation}
\label{eqn:pair slopes advanced}
\langle E,F \rangle = \sum_{k,\ell} c^+_k c^-_{\ell} \cdot \langle A^+_k, A^-_{\ell} \rangle  \langle B^+_k, B^-_{\ell} \rangle 
\end{equation}
In particular, if $(E, F) \in \CS^+_{[\bp,\binfty]} \times \CS^-_{(-\binfty,\bp)}$ or $(E, F) \in \CS^+_{(-\binfty,\bp)} \times \CS^-_{[\bp,\binfty]}$, then
\begin{equation}
\label{eqn:epsilon}
\langle E, F \rangle = \varepsilon(E) \varepsilon(F)
\end{equation}
with $\varepsilon$ denoting the counit. 
By \eqref{eqn:pair slopes advanced} and the perfectness of the pairings \eqref{eqn:perfect pairing 1}-\eqref{eqn:perfect pairing 2}, we may uniquely define for any $E \in \CS^\geq$ and $F \in \CS^\leq$ the elements
\begin{align} 
&[E]_{\geq \bp} \in \CS^+_{[\bp,\binfty]} \quad \ \text{ by} \quad \langle [E]_{\geq \bp}, Y \rangle = \langle E, Y \rangle, \ \forall Y \in \CS^-_{[\bp,\binfty]} \label{eqn:unique 1} \\
&[F]_{< \bp} \in \CS^-_{(-\binfty,\bp)} \quad \text{by} \quad \langle X, [F]_{< \bp}  \rangle = \langle X, F \rangle, \ \forall X \in \CS^+_{(-\binfty,\bp)} \label{eqn:unique 2} \\
&[F]_{\geq \bp} \in \CS^-_{[\bp,\binfty]} \quad \ \text{ by} \quad \langle X, [F]_{\geq \bp} \rangle = \langle X, F \rangle, \ \forall X \in \CS^+_{[\bp,\binfty]} \label{eqn:unique 3} \\
&[E]_{< \bp} \in \CS^+_{(-\binfty,\bp)} \quad \text{by} \quad \langle [E]_{<\bp}, Y \rangle = \langle E, Y \rangle, \ \forall Y \in \CS^-_{(-\binfty,\bp)} \label{eqn:unique 4} 
\end{align}
\footnote{A little care must be taken with the formulas above if $E = E'\kappa^+$ and $F \in F' \kappa^-$ where
$$
E' \in \CS^+ \otimes \CB_{\binfty}^+ \quad \text{and} \quad F' \in \CS^- \otimes \CB_{\binfty}^-
$$
and $\kappa^\pm$ denote various products of finite Cartan elements $\kappa^\pm_i$. In this case, we set
$$
[E]_{\geq \bp} = [E']_{\geq \bp} \kappa^+, \qquad [F]_{<\bp} = [F']_{<\bp}, \qquad [F]_{\geq \bp} = [F']_{\geq \bp} \kappa^-, \qquad [E]_{<\bp} = [E']_{<\bp}
$$
with $[E']_{\geq \bp}, [F']_{<\bp}, [F']_{\geq \bp}, [E']_{<\bp}$ uniquely determined by \eqref{eqn:unique 1}-\eqref{eqn:unique 4}.} Moreover, if we write the coproduct as $\Delta(E) = E_1 \otimes E_2$, $\Delta(F) = F_1 \otimes F_2$, then
\begin{align}
&E = [E_1]_{< \bp} [E_2]_{\geq \bp} \label{eqn:equation 1} \\
&F = [F_2]_{< \bp} [F_1]_{\geq \bp} \label{eqn:equation 2}
\end{align}
for all $E \in \CS^\geq$, $F \in \CS^\leq$. Formula \eqref{eqn:equation 1} is proved by pairing both sides with an arbitrary $XY$ where $X \in \CS^-_{(-\binfty,\bp)}$, $Y \in \CS^-_{[\bp,\binfty]}$ and then evaluating the left-hand side with \eqref{eqn:bialgebra 1} and the right-hand side with \eqref{eqn:pair slopes advanced}. Formula \eqref{eqn:equation 2} is proved similarly.

\medskip

\subsection{The half subalgebras}
\label{sub:slopes}

Consider any $\bp \in \rr$. The straightforward generalization of \cite[Proposition 3.18]{N Cat} yields subalgebras
\begin{align}
	&\CA^{\geq \bp} =  \CS^+_{[\bp,\binfty]} \otimes \CS^-_{(-\binfty,\bp)} \label{eqn:slope p plus} \\
	&\CA^{\leq \bp} = \CS^-_{[\bp,\binfty]}  \otimes \CS^+_{(-\binfty,\bp)} \label{eqn:slope p minus}
\end{align}
of $\CS$, which provide an isomorphism of vector spaces (cf. \cite[Proposition 3.21]{N Cat})
\begin{equation}
	\label{eqn:a triangular}
	\CA^{\geq \bp} \otimes \CA^{\leq \bp} \xrightarrow{\sim} \CS = \U 
\end{equation}
Note that the loop Cartan subalgebras $\CB_{\binfty}^\geq$ and $\CB_{\binfty}^\leq$ are contained in $\CA^{\geq \bp}$ and $\CA^{\leq \bp}$, respectively. The initial motivation for this construction is that in the case of simple Lie algebras $\fg$ (i.e. the quantum loop algebra \eqref{eqn:quantum affine intro}), the subalgebras $\CA^{\geq \b0}$ and $\CA^{\leq \b0}$ correspond to the Borel subalgebras of the quantum affine algebra $U_q(\widehat{\fg})$ under the Drinfeld-Beck isomorphism, see Subsection \ref{sub:quantum affine}. However, the importance of this construction for our purposes lies in the following strengthening of Theorem \ref{thm:main intro}.

\medskip 

\begin{theorem}
\label{thm:main}

For any $\bp \in \rr$, there exist ``new new" topological coproducts
\begin{equation}
\label{eqn:main coproduct}
\Delta_{\bp} : \CA^{\geq \bp} \rightarrow \CA^{\geq \bp} \ho \CA^{\geq \bp} \qquad \text{and} \qquad \Delta_{\bp} : \CA^{\leq \bp} \rightarrow \CA^{\leq \bp} \ho \CA^{\leq \bp} 
\end{equation}
which extend the coproducts \eqref{eqn:coproduct slope plus} and \eqref{eqn:coproduct slope minus}. With respect to $\Delta_{\bp}$, the pairing
\begin{equation}
\label{eqn:main pairing}
\CA^{\geq \bp} \otimes \CA^{\leq \bp} \xrightarrow{\langle \cdot, \cdot \rangle_{\bp}} \BK 
\end{equation}
to be introduced in \eqref{eqn:footnote} is a bialgebra pairing. The corresponding Drinfeld double 
\begin{equation}
\label{eqn:main double}
\CA^{\geq \bp} \otimes \CA^{\leq \bp}
\end{equation}
is isomorphic to $\CS = \U$ as an algebra.

\end{theorem} 

\medskip

\noindent Thus, there are as many topological coproducts on $\U$ as there are elements $\bp \in \rr$ (beside the Drinfeld new coproduct $\Delta$, which morally corresponds to $\bp = \binfty$). In geometric situations such as critical $K$-theoretic Hall algebras of quivers, we expect our coproducts to match the coproducts defined by \cite{COZZ1, COZZ2} using the FRT formalism and stable envelopes. In the particular case of preprojective $K$-theoretic Hall algebras of quivers that was developed in \cite{OS}, the coincidence between our construction and that of \loccit follows from \cite[Sections 6.5.1, 6.5.2]{Z}. Finally, when $\U$ is defined using the zeta functions \eqref{eqn:zeta intro particular} for an affine Lie algebra $\fg$, it would be interesting to compare the coproduct \eqref{eqn:main coproduct} with the coproduct on quantum toroidal algebras defined in \cite{L} using double affine braid group actions.

\medskip

\begin{proof} \emph{of Theorem \ref{thm:main}:} By the very definition of our half subalgebras in \eqref{eqn:slope p plus} and \eqref{eqn:slope p minus}, general elements of $\CA^{\geq \bp}$ and $\CA^{\leq \bp}$ are linear combinations of
\begin{equation}
\label{eqn:as in}
\textcolor{red}{E}\textcolor{blue}{F} \quad \text{and} \quad \textcolor{blue}{F'}\textcolor{red}{E'}
\end{equation}
respectively, where we make the following color codes
\begin{align*} 
&\text{red non-primed letters:} \qquad \textcolor{red}{E}, \textcolor{red}{X} \in \CS^+_{[\bp,\binfty]} \\
&\text{blue non-primed letters:} \ \ \quad \textcolor{blue}{F}, \textcolor{blue}{Y} \in \CS^-_{(-\binfty,\bp)} \\
&\text{red primed letters:} \qquad \qquad \textcolor{red}{E'}, \textcolor{red}{X'} \in \CS^+_{(-\binfty,\bp)} \\
&\text{blue primed letters:} \ \ \quad \qquad \textcolor{blue}{F'}, \textcolor{blue}{Y'} \in \CS^-_{[\bp,\binfty]} 
\end{align*} 
If we represent horizontal and vertical degree on horizontal and vertical axes, then the elements above occupy the wedges indicated in the following picture

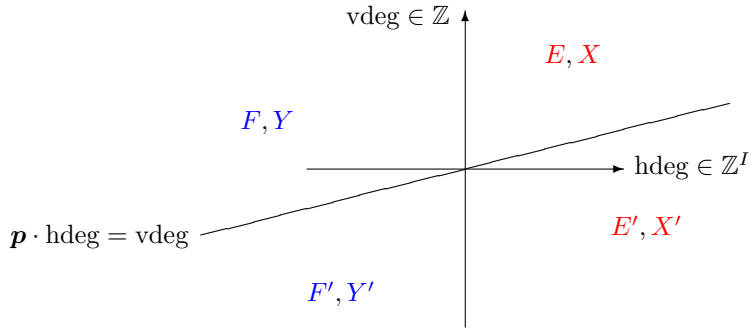
\begin{picture}(100,150)(-110,-20)
\label{pic:par}
	
\put(60,0){\vector(0,1){120}}
\put(0,60){\vector(1,0){120}}
\put(-40,35){\line(4,1){200}}
	
\put(15,115){$\text{vdeg} \in \BZ$}
\put(124,58){$\text{hdeg} \in \zz$}
\put(-112,32){$\bp \cdot \text{hdeg} = \text{vdeg}$}

\put(90,100){$\textcolor{red}{E}, \textcolor{red}{X}$}
\put(115,35){$\textcolor{red}{E'}, \textcolor{red}{X'}$}
\put(-25,75){$\textcolor{blue}{F}, \textcolor{blue}{Y}$}
\put(0,10){$\textcolor{blue}{F'}, \textcolor{blue}{Y'}$}
	
\end{picture}

\noindent Black letters with tildes such as $\tilde{E}, \tilde{F}, \tilde{E}', \tilde{F}'$ correspond to elements which are free to run over the same subalgebras as $\textcolor{red}{E}, \textcolor{blue}{F}, \textcolor{red}{E'}, \textcolor{blue}{F'}$, and they will be used as dummy arguments of various linear functionals. Then we declare
\begin{align} 
&\Delta_{\bp}(\textcolor{red}{E}\textcolor{blue}{F}) = \textcolor{red}{E_1} \textcolor{blue}{Y} \otimes \textcolor{red}{X} \textcolor{blue}{F_1} \label{eqn:delta p plus} \\
&\Delta_{\bp}(\textcolor{blue}{F'}\textcolor{red}{E'}) = \textcolor{blue}{Y'}\textcolor{red}{E'_2} \otimes \textcolor{blue}{F'_2} \textcolor{red}{X'} \label{eqn:delta p minus} 
\end{align} 
where the indices $1$ and $2$ refer to Sweedler notation for the coproduct $\Delta$ of Subsection \ref{sub:coproduct}, and the tensors $\textcolor{blue}{Y} \otimes \textcolor{red}{X}$ and $\textcolor{blue}{Y'} \otimes \textcolor{red}{X'}$ are defined by 
\begin{align} 
&\textcolor{red}{X} \langle \tilde{E}', \textcolor{blue}{Y} \rangle =\left[ \textcolor{red}{E_2} \tilde{E}'_1 \right]_{\geq \bp} \langle \tilde{E}'_2, \textcolor{blue}{F_2}\rangle  \label{eqn:declare 1 def} \\
&\textcolor{blue}{Y'} \langle \textcolor{red}{X'}, \tilde{F} \rangle  = \left[ \textcolor{blue}{F_1'} \tilde{F}_2 \right]_{\geq \bp}  \langle \textcolor{red}{E_1'}, \tilde{F}_1 \rangle \label{eqn:declare 2 def}
\end{align} 
The fact that the formulas above are well-defined is due to the perfectness of $\CS^+_{(-\binfty,\bp)} \otimes \CS^-_{(-\binfty,\bp)} \rightarrow \BK$. Using \eqref{eqn:unique 1}-\eqref{eqn:unique 4},  formulas \eqref{eqn:declare 1 def}-\eqref{eqn:declare 2 def} are equivalent to
\begin{align} 
&\langle \textcolor{red}{X}, \tilde{F}' \rangle \langle \tilde{E}', \textcolor{blue}{Y} \rangle = \langle \tilde{E}'_1, \tilde{F}'_1\rangle \langle \tilde{E}'_2, \textcolor{blue}{F_2}\rangle \langle \textcolor{red}{E_2}, \tilde{F}_2' \rangle \label{eqn:declare 1} \\
&\langle \textcolor{red}{X'}, \tilde{F} \rangle \langle \tilde{E}, \textcolor{blue}{Y'} \rangle = \langle \tilde{E}_2, \tilde{F}_2\rangle \langle \tilde{E}_1, \textcolor{blue}{F'_1}\rangle \langle \textcolor{red}{E_1'}, \tilde{F}_1 \rangle \label{eqn:declare 2}
\end{align} 
due to the non-degeneracy of the pairing  $\CS^+_{[\bp,\binfty]} \otimes \CS^-_{[\bp,\binfty]} \rightarrow \BK$. 

\medskip 

\noindent The fact that the terms $\textcolor{red}{E_1} \textcolor{blue}{Y},\textcolor{red}{X} \textcolor{blue}{F_1}$ in the right-hand side of \eqref{eqn:delta p plus} and $\textcolor{blue}{Y'}\textcolor{red}{E'_2}, \textcolor{blue}{F'_2} \textcolor{red}{X'}$ in the right-hand side of \eqref{eqn:delta p minus} lie in $\CA^{\geq \bp}$ and $\CA^{\leq \bp}$, respectively, follows from \eqref{eqn:interact 1}-\eqref{eqn:interact 4}. Moreover, the right-hand sides of \eqref{eqn:delta p plus}-\eqref{eqn:delta p minus} lie in the completion \eqref{eqn:ho} because the horizontal degrees of $\textcolor{red}{E_1}, \textcolor{red}{E_2'}, \textcolor{blue}{F_1}, \textcolor{blue}{F_2'}$ are bounded on both sides, while the horizontal degrees of $\textcolor{red}{X}, \textcolor{red}{X'}$ are bounded below and those of $\textcolor{blue}{Y}, \textcolor{blue}{Y'}$ are bounded above. In what follows, we spell out the following particular cases of \eqref{eqn:delta p plus}-\eqref{eqn:delta p minus}:
\begin{align} 
&\Delta_{\bp}(\textcolor{red}{E}) = \textcolor{red}{E_1} \textcolor{blue}{Y} \otimes \textcolor{red}{X}, \qquad \quad \langle \textcolor{red}{X}, \tilde{F}' \rangle \langle \tilde{E}', \textcolor{blue}{Y} \rangle  = \langle \textcolor{red}{E_2} \tilde{E}', \tilde{F}'\rangle \label{eqn:two 1} \\ 
&\Delta_{\bp}(\textcolor{blue}{F}) =  \textcolor{blue}{Y} \otimes \textcolor{red}{X} \textcolor{blue}{F_1}, \qquad \quad \langle \textcolor{red}{X}, \tilde{F}' \rangle \langle \tilde{E}', \textcolor{blue}{Y} \rangle  = \langle \tilde{E}', \tilde{F}' \textcolor{blue}{F_2}\rangle  \label{eqn:two 2} \\ 
&\Delta_{\bp}(\textcolor{red}{E'}) = \textcolor{blue}{Y'}\textcolor{red}{E'_2} \otimes \textcolor{red}{X'}, \qquad \langle \textcolor{red}{X'}, \tilde{F} \rangle \langle \tilde{E}, \textcolor{blue}{Y'} \rangle  = \langle \tilde{E}  \textcolor{red}{E_1'}, \tilde{F} \rangle\label{eqn:two 3} \\ 
&\Delta_{\bp}(\textcolor{blue}{F'}) = \textcolor{blue}{Y'} \otimes \textcolor{blue}{F'_2} \textcolor{red}{X'}, \qquad \langle \textcolor{red}{X'}, \tilde{F} \rangle \langle \tilde{E}, \textcolor{blue}{Y'} \rangle = \langle \tilde{E}, \textcolor{blue}{F'_1} \tilde{F}\rangle  \label{eqn:two 4} 
\end{align}

\medskip

\begin{claim} 
\label{claim:extend}

The coproducts \eqref{eqn:delta p plus}-\eqref{eqn:delta p minus} extend the coproducts \eqref{eqn:coproduct slope plus}-\eqref{eqn:coproduct slope minus}. 

\end{claim} 

\medskip

\begin{proof} Plug $\textcolor{red}{E} \in \CB_\bp^{\geq}$, $\textcolor{blue}{F}  = 1$, $\textcolor{red}{E'}  = 1$, $\textcolor{blue}{F'} \in \CB_{\bp}^{\leq}$ in formulas \eqref{eqn:delta p plus}-\eqref{eqn:delta p minus}. Then
\begin{align} 
&\textcolor{red}{X} \langle \tilde{E}', \textcolor{blue}{Y} \rangle =\left[ \textcolor{red}{E_2} \tilde{E}' \right]_{\geq \bp} \label{eqn:temp 1} \\
&\textcolor{blue}{Y'} \langle \textcolor{red}{X'}, \tilde{F} \rangle  = \left[ \textcolor{blue}{F_1'} \tilde{F} \right]_{\geq \bp}  \label{eqn:temp 2}
\end{align} 
in formulas \eqref{eqn:declare 1 def}-\eqref{eqn:declare 2 def}. As in \cite[Subsection 3.7]{N R-matrix}, we have
\begin{align}
	&\textcolor{red}{E_1} \otimes \textcolor{red}{E_2} \in \CB_{\bp}^{\geq} \otimes \CB_{\bp}^{\geq} + \CS^\geq_{>\bp} \otimes \CS^\geq_{<\bp}
	\label{eqn:extend 1} \\
	&\textcolor{blue}{F'_1} \otimes \textcolor{blue}{F'_2} \in \CB_{\bp}^{\leq} \otimes \CB_{\bp}^{\leq}  +  \CS^\leq_{<\bp} \otimes \CS^\leq_{>\bp}
	\label{eqn:extend 2} 
\end{align} 
Because $\tilde{E}'$ and $\tilde{F}$ in \eqref{eqn:temp 1} and \eqref{eqn:temp 2} have slope $<\bp$, the right-most summands in the RHS of \eqref{eqn:extend 1}-\eqref{eqn:extend 2} do not contribute anything to $\textcolor{blue}{Y} \otimes \textcolor{red}{X}$ and $\textcolor{blue}{Y'} \otimes \textcolor{red}{X'}$. By the same token, the left-most summands in the RHS of \eqref{eqn:extend 1}-\eqref{eqn:extend 2} only produce a non-trivial contribution in \eqref{eqn:temp 1}-\eqref{eqn:temp 2} if $\tilde{E}' = \tilde{F} = 1$, which implies that
\begin{align*} 
&\textcolor{blue}{Y} \otimes \textcolor{red}{X} = 1 \otimes \Big(\text{those }\textcolor{red}{E_2} \text{ in }\CB_{\bp}^{\geq} \Big) \\ &\textcolor{blue}{Y'} \otimes \textcolor{red}{X'} = \Big(\text{those }\textcolor{blue}{F_1'} \text{ in }\CB_{\bp}^{\leq} \Big) \otimes 1
\end{align*}
Plugging the formulas above in \eqref{eqn:delta p plus}-\eqref{eqn:delta p minus} yields
\begin{align*}
&\Delta_{\bp}(\textcolor{red}{E}) = \Big( \text{those } \textcolor{red}{E_1} \otimes \textcolor{red}{E_2} \text{ in } \CB_\bp^{\geq} \otimes \CB_\bp^{\geq} \Big) \\
&\Delta_{\bp}(\textcolor{blue}{F'}) = \Big( \text{those } \textcolor{blue}{F'_1} \otimes \textcolor{blue}{F'_2} \text{ in } \CB_\bp^{\leq} \otimes \CB_\bp^{\leq} \Big) 
\end{align*}
which coincides with the right-hand sides of \eqref{eqn:coproduct slope plus} and \eqref{eqn:coproduct slope minus}, respectively. \end{proof}

\medskip

\noindent Let us return to the proof of Theorem \ref{thm:main}. We define the pairing
\begin{equation}
\label{eqn:footnote}
\CA^{\geq \bp} \otimes \CA^{\leq \bp} \xrightarrow{\langle \cdot, \cdot \rangle_{\bp}} \BK, \qquad \langle \textcolor{red}{E}\textcolor{blue}{F}, \textcolor{blue}{F'}\textcolor{red}{E'} \rangle_{\bp} = \langle \textcolor{red}{E}, \textcolor{blue}{F'} \rangle \langle \textcolor{red}{E'}, \textcolor{blue}{F}\rangle
\end{equation}
for any $\textcolor{red}{E} \in \CS^+_{[\bp,\binfty]}$, $\textcolor{blue}{F} \in \CS^-_{(-\binfty,\bp)}$, $\textcolor{red}{E'} \in \CS^+_{(-\binfty, \bp)}$, $\textcolor{blue}{F'} \in \CS^-_{[\bp,\binfty]}$. 

\medskip 

\begin{claim}
\label{claim:bialgebra}

$\langle \cdot, \cdot \rangle_{\bp}$ is a bialgebra pairing with respect to $\Delta_{\bp}$, i.e.
\begin{align} 
&\langle \textcolor{red}{E}\textcolor{blue}{F}, \textcolor{blue}{F'}\textcolor{red}{E'} \textcolor{blue}{\tilde{F}'}\textcolor{red}{\tilde{E}'} \rangle_{\bp} = \langle \textcolor{red}{E_1} \textcolor{blue}{Y}, \textcolor{blue}{F'}\textcolor{red}{E'} \rangle_{\bp} \langle \textcolor{red}{X} \textcolor{blue}{F_1}, \textcolor{blue}{\tilde{F}'}\textcolor{red}{\tilde{E}'} \rangle_{\bp} \label{eqn:bialgebra 1 big} \\
&\langle \textcolor{red}{E}\textcolor{blue}{F} \textcolor{red}{\tilde{E}}\textcolor{blue}{\tilde{F}}, \textcolor{blue}{F'}\textcolor{red}{E'} \rangle_{\bp} = \langle \textcolor{red}{E}\textcolor{blue}{F},  \textcolor{blue}{F'_2} \textcolor{red}{X'} \rangle_{\bp} \langle \textcolor{red}{\tilde{E}}\textcolor{blue}{\tilde{F}}, \textcolor{blue}{Y'}\textcolor{red}{E'_2} \rangle_{\bp} \label{eqn:bialgebra 2 big}
 \end{align}
 for all $\textcolor{red}{E}, \textcolor{red}{\tilde{E}} \in \CS^+_{[\bp,\binfty]}$, $\textcolor{blue}{F}, \textcolor{blue}{\tilde{F}} \in \CS^-_{(-\binfty,\bp)}$, $\textcolor{red}{E'}, \textcolor{red}{\tilde{E}'} \in \CS^+_{(-\binfty,\bp)}$, $\textcolor{blue}{F'}, \textcolor{blue}{\tilde{F}'} \in \CS^-_{[\bp,\binfty]}$.

\end{claim}

\medskip

\begin{proof} The Drinfeld double relation \eqref{eqn:drinfeld double relation} states that
$$
\textcolor{red}{E'_1} \textcolor{blue}{\tilde{F}'_1} \langle \textcolor{red}{E'_2} , \textcolor{blue}{\tilde{F}'_2} \rangle = \langle \textcolor{red}{E'_1}, \textcolor{blue}{\tilde{F}'_1} \rangle \textcolor{blue}{\tilde{F}'_2}\textcolor{red}{E'_2}
$$
$$
\langle \textcolor{red}{\tilde{E}_1}, \textcolor{blue}{F_1} \rangle \textcolor{blue}{F_2} \textcolor{red}{\tilde{E}_2} = \textcolor{red}{\tilde{E}_1} \textcolor{blue}{F_1} \langle \textcolor{red}{\tilde{E}_2} , \textcolor{blue}{F_2} \rangle
$$
Because of \eqref{eqn:interact 1}-\eqref{eqn:interact 4} and \eqref{eqn:epsilon}, the left-hand sides of the equations above simplify to
\begin{equation}
\label{eqn:simplify 1}
\textcolor{red}{E'} \textcolor{blue}{\tilde{F}'} = \langle \textcolor{red}{E'_1}, \textcolor{blue}{\tilde{F}'_1} \rangle \textcolor{blue}{\tilde{F}'_2}\textcolor{red}{E'_2}
\end{equation}
\begin{equation}
\label{eqn:simplify 2}
\textcolor{blue}{F} \textcolor{red}{\tilde{E}} = \textcolor{red}{\tilde{E}_1} \textcolor{blue}{F_1} \langle \textcolor{red}{\tilde{E}_2} , \textcolor{blue}{F_2} \rangle
\end{equation}
Plugging \eqref{eqn:simplify 1}-\eqref{eqn:simplify 2} into the left-hand sides of \eqref{eqn:bialgebra 1 big}-\eqref{eqn:bialgebra 2 big} makes the latter equivalent to
\begin{align*} 
&\langle \textcolor{red}{E}\textcolor{blue}{F}, \textcolor{blue}{F'} \textcolor{blue}{\tilde{F}_2'} \textcolor{red}{E'_2} \textcolor{red}{\tilde{E}'} \rangle_{\bp} \langle \textcolor{red}{E'_1}, \textcolor{blue}{\tilde{F}'_1} \rangle = \langle \textcolor{red}{E_1} \textcolor{blue}{Y}, \textcolor{blue}{F'}\textcolor{red}{E'} \rangle_{\bp} \langle \textcolor{red}{X} \textcolor{blue}{F_1}, \textcolor{blue}{\tilde{F}'}\textcolor{red}{\tilde{E}'} \rangle_{\bp} \\
&\langle \textcolor{red}{E} \textcolor{red}{\tilde{E}_1} \textcolor{blue}{F_1}\textcolor{blue}{\tilde{F}}, \textcolor{blue}{F'}\textcolor{red}{E'} \rangle_{\bp} \langle \textcolor{red}{\tilde{E}_2} , \textcolor{blue}{F_2} \rangle = \langle \textcolor{red}{E}\textcolor{blue}{F},  \textcolor{blue}{F'_2} \textcolor{red}{X'} \rangle_{\bp} \langle \textcolor{red}{\tilde{E}}\textcolor{blue}{\tilde{F}}, \textcolor{blue}{Y'}\textcolor{red}{E'_2} \rangle_{\bp} 
\end{align*}
By \eqref{eqn:footnote}, the equalities above are equivalent to
\begin{align*} 
&\langle \textcolor{red}{E}, \textcolor{blue}{F'} \textcolor{blue}{\tilde{F}_2'} \rangle \langle  \textcolor{red}{E'_2} \textcolor{red}{\tilde{E}'}, \textcolor{blue}{F} \rangle \langle \textcolor{red}{E'_1}, \textcolor{blue}{\tilde{F}'_1} \rangle = \langle \textcolor{red}{E_1}, \textcolor{blue}{F'} \rangle \langle \textcolor{red}{E'}, \textcolor{blue}{Y} \rangle \langle \textcolor{red}{X} , \textcolor{blue}{\tilde{F}'} \rangle \langle \textcolor{red}{\tilde{E}'}, \textcolor{blue}{F_1} \rangle \\
&\langle \textcolor{red}{E} \textcolor{red}{\tilde{E}_1}, \textcolor{blue}{F'} \rangle \langle   \textcolor{red}{E'}, \textcolor{blue}{F_1}\textcolor{blue}{\tilde{F}} \rangle \langle \textcolor{red}{\tilde{E}_2} , \textcolor{blue}{F_2} \rangle = \langle \textcolor{red}{E},  \textcolor{blue}{F'_2} \rangle \langle \textcolor{red}{X'}, \textcolor{blue}{F}\rangle \langle \textcolor{red}{\tilde{E}}, \textcolor{blue}{Y'} \rangle  \langle \textcolor{red}{E'_2}, \textcolor{blue}{\tilde{F}} \rangle
\end{align*}
After applying \eqref{eqn:bialgebra 1}-\eqref{eqn:bialgebra 2} to the left-hand sides, the expressions above match \eqref{eqn:declare 1}-\eqref{eqn:declare 2}, and are therefore proved. \end{proof} 

\medskip

\noindent To conclude the proof of Theorem \ref{thm:main}, we must show that the Drinfeld double relation \eqref{eqn:drinfeld double relation} holds in $\CS = \U$ with respect to the coproduct \eqref{eqn:delta p plus}-\eqref{eqn:delta p minus} and the pairing \eqref{eqn:footnote}, for any $\bp \in \rr$. Since the Drinfeld double relation is multiplicative in $a$ and $b$, it suffices to check \eqref{eqn:drinfeld double relation} for $(a,b)$ among
\begin{equation}
\label{eqn:check}
(\textcolor{red}{E}, \textcolor{red}{E'}), \quad (\textcolor{red}{E}, \textcolor{blue}{F'}), \quad (\textcolor{blue}{F}, \textcolor{red}{E'}), \quad (\textcolor{blue}{F}, \textcolor{blue}{F'})
\end{equation}
For the first check, we apply formulas \eqref{eqn:two 1} and \eqref{eqn:two 3} and we need to show that
\begin{equation}
\label{eqn:first check}
\textcolor{red}{E_1} \textcolor{blue}{Y}\textcolor{blue}{Y'}\textcolor{red}{E'_2} \langle \textcolor{red}{X}, \textcolor{red}{X'} \rangle_{\bp} = \langle \textcolor{red}{E_1} \textcolor{blue}{Y}, \textcolor{blue}{Y'}\textcolor{red}{E'_2} \rangle_{\bp} \textcolor{red}{X'}  \textcolor{red}{X}
\end{equation}
where 
\begin{align*} 
&\langle \textcolor{red}{X}, \tilde{F}' \rangle \langle \tilde{E}', \textcolor{blue}{Y} \rangle = \langle \textcolor{red}{E_2} \tilde{E}', \tilde{F}'\rangle \quad \Rightarrow \quad \textcolor{red}{X} \langle \tilde{E}', \textcolor{blue}{Y} \rangle = \left[\textcolor{red}{E_2}\tilde{E}' \right]_{\geq \bp} \quad \text{and} \quad \textcolor{blue}{Y} \varepsilon(\textcolor{red}{X}) = \varepsilon(\textcolor{red}{E_2}) \\
&\langle \textcolor{red}{X'}, \tilde{F} \rangle \langle \tilde{E}, \textcolor{blue}{Y'} \rangle  = \langle \tilde{E}  \textcolor{red}{E_1'}, \tilde{F} \rangle \quad \Rightarrow \quad \textcolor{red}{X'} \langle \tilde{E}, \textcolor{blue}{Y'} \rangle = \left[\tilde{E} \textcolor{red}{E_1'} \right]_{< \bp} \quad \text{and} \quad \textcolor{blue}{Y'} \varepsilon(\textcolor{red}{X'}) = \varepsilon(\textcolor{red}{E_1'})
\end{align*}
By \eqref{eqn:footnote}, we have
\begin{align*}
&\text{LHS of \eqref{eqn:first check}} = \textcolor{red}{E_1} \textcolor{blue}{Y}\textcolor{blue}{Y'}\textcolor{red}{E'_2} \varepsilon(\textcolor{red}{X}) \varepsilon(\textcolor{red}{X'}) = \textcolor{red}{E_1} \varepsilon(\textcolor{red}{E_2}) \varepsilon(\textcolor{red}{E_1'})\textcolor{red}{E'_2}  = \textcolor{red}{E}\textcolor{red}{E'} \\
&\text{RHS of \eqref{eqn:first check}} =  \langle \textcolor{red}{E_1} , \textcolor{blue}{Y'} \rangle \langle \textcolor{red}{E'_2} , \textcolor{blue}{Y} \rangle \textcolor{red}{X'}  \textcolor{red}{X} = [\textcolor{red}{E_1}\textcolor{red}{E'_1}]_{<\bp} [\textcolor{red}{E_2}\textcolor{red}{E'_2}]_{\geq\bp} 
\end{align*} 
The equality of the two expressions above is due to \eqref{eqn:equation 1}. 

\medskip

\noindent For the second check in \eqref{eqn:check}, we apply \eqref{eqn:two 1} and \eqref{eqn:two 4} and we need to show that
\begin{equation}
\label{eqn:second check}
\textcolor{red}{E_1} \textcolor{blue}{Y} \textcolor{blue}{Y'} \langle \textcolor{red}{X}, \textcolor{blue}{F'_2} \textcolor{red}{X'} \rangle_{\bp} = \langle \textcolor{red}{E_1} \textcolor{blue}{Y}, \textcolor{blue}{Y'} \rangle_{\bp} \textcolor{blue}{F'_2} \textcolor{red}{X'}\textcolor{red}{X}
\end{equation}
where
\begin{align*} 
&\langle \textcolor{red}{X}, \tilde{F}' \rangle \langle \tilde{E}', \textcolor{blue}{Y} \rangle =  \langle \textcolor{red}{E_2} \tilde{E}', \tilde{F}'\rangle \quad \Rightarrow \quad \textcolor{blue}{Y} \langle \textcolor{red}{X}, \tilde{F}' \rangle = [\tilde{F}_1']_{<\bp} \langle \textcolor{red}{E_2}, \tilde{F}_2' \rangle \ \text{ and } \ \textcolor{red}{X}\varepsilon(\textcolor{blue}{Y}) = [\textcolor{red}{E_2}]_{\geq \bp}  \\
&\langle \textcolor{red}{X'}, \tilde{F} \rangle \langle \tilde{E}, \textcolor{blue}{Y'} \rangle = \langle \tilde{E}, \textcolor{blue}{F'_1} \tilde{F}\rangle \quad \Rightarrow \quad \textcolor{red}{X'}  \langle \tilde{E}, \textcolor{blue}{Y'} \rangle = [\tilde{E}_2]_{<\bp} \langle \tilde{E}_1, \textcolor{blue}{F_1'} \rangle \ \text{ and } \ \textcolor{blue}{Y'} \varepsilon(\textcolor{red}{X'}) = [\textcolor{blue}{F_1'}]_{\geq \bp} 
\end{align*}
By \eqref{eqn:footnote}, we have 
\begin{align*} 
&\text{LHS of \eqref{eqn:second check}} = \textcolor{red}{E_1} \textcolor{blue}{Y} \textcolor{blue}{Y'} \langle \textcolor{red}{X}, \textcolor{blue}{F'_2} \rangle \varepsilon(\textcolor{red}{X'}) = \textcolor{red}{E_1} \textcolor{blue}{Y} [\textcolor{blue}{F_1'}]_{\geq \bp} \langle \textcolor{red}{X}, \textcolor{blue}{F'_2} \rangle = \textcolor{red}{E_1} [\textcolor{blue}{F_2'}]_{<\bp} [\textcolor{blue}{F_1'}]_{\geq \bp} \langle \textcolor{red}{E_2}, \textcolor{blue}{F_3'} \rangle  \\
&\text{RHS of \eqref{eqn:second check}} = \varepsilon(\textcolor{blue}{Y}) \langle \textcolor{red}{E_1}, \textcolor{blue}{Y'} \rangle \textcolor{blue}{F'_2} \textcolor{red}{X'}\textcolor{red}{X} = \langle \textcolor{red}{E_1}, \textcolor{blue}{Y'} \rangle \textcolor{blue}{F'_2} \textcolor{red}{X'} [\textcolor{red}{E_2}]_{\geq \bp}  = \textcolor{blue}{F'_2} [\textcolor{red}{E_2}]_{<\bp} [\textcolor{red}{E_3}]_{\geq \bp} \langle \textcolor{red}{E_1}, \textcolor{blue}{F_1'} \rangle
\end{align*}
By \eqref{eqn:equation 1} and \eqref{eqn:equation 2}, the right-hand sides of the expressions above are equal to $\textcolor{red}{E_1} \textcolor{blue}{F_1'} \langle \textcolor{red}{E_2}, \textcolor{blue}{F_2'} \rangle$ and $\textcolor{blue}{F'_2} \textcolor{red}{E_2} \langle \textcolor{red}{E_1}, \textcolor{blue}{F_1'} \rangle$ respectively, which are equal to each other by \eqref{eqn:drinfeld double relation}. 

\medskip

\noindent For the third check in \eqref{eqn:check}, we apply \eqref{eqn:two 2} and \eqref{eqn:two 3} and we need to show that
\begin{equation}
\label{eqn:third check}
\textcolor{blue}{Y}  \textcolor{blue}{Y'}\textcolor{red}{E'_2} \langle \textcolor{red}{X} \textcolor{blue}{F_1}, \textcolor{red}{X'}\rangle_{\bp} = \langle  \textcolor{blue}{Y} , \textcolor{blue}{Y'}\textcolor{red}{E'_2} \rangle_{\bp}\textcolor{red}{X'}\textcolor{red}{X} \textcolor{blue}{F_1}
\end{equation}
where
\begin{align*} 
&\langle \textcolor{red}{X}, \tilde{F}' \rangle \langle \tilde{E}', \textcolor{blue}{Y} \rangle = \langle \tilde{E}', \tilde{F}' \textcolor{blue}{F_2}\rangle \quad \Rightarrow \quad \textcolor{red}{X} \langle \tilde{E}', \textcolor{blue}{Y}  \rangle = [\tilde{E}'_1]_{\geq \bp} \langle \tilde{E}'_2, \textcolor{blue}{F_2} \rangle \ \text{ and } \  \textcolor{blue}{Y} \varepsilon( \textcolor{red}{X}) = [\textcolor{blue}{F_2}]_{<\bp}   \\
&\langle \textcolor{red}{X'}, \tilde{F} \rangle \langle \tilde{E}, \textcolor{blue}{Y'} \rangle = \langle \tilde{E}  \textcolor{red}{E_1'}, \tilde{F} \rangle \quad \Rightarrow \quad \textcolor{blue}{Y'} \langle \textcolor{red}{X'}, \tilde{F} \rangle = [\tilde{F}_2]_{\geq \bp} \langle \textcolor{red}{E_1'}, \tilde{F}_1 \rangle  \ \text{ and } \ \textcolor{red}{X'}\varepsilon(\textcolor{blue}{Y'}) = [\textcolor{red}{E_1'}]_{<\bp}  
\end{align*} 
By \eqref{eqn:footnote}, we have 
\begin{align*} 
&\text{LHS of \eqref{eqn:third check}} = \textcolor{blue}{Y}  \textcolor{blue}{Y'}\textcolor{red}{E'_2} \langle \textcolor{red}{X'}, \textcolor{blue}{F_1} \rangle \varepsilon(\textcolor{red}{X}) =  [\textcolor{blue}{F_2}]_{<\bp}   \textcolor{blue}{Y'}\textcolor{red}{E'_2} \langle \textcolor{red}{X'}, \textcolor{blue}{F_1} \rangle = [\textcolor{blue}{F_3}]_{<\bp}  [\textcolor{blue}{F_2}]_{\geq \bp} \textcolor{red}{E'_2}  \langle \textcolor{red}{E_1'}, \textcolor{blue}{F_1} \rangle  \\
&\text{LHS of \eqref{eqn:third check}} = \varepsilon(\textcolor{blue}{Y'}) \langle \textcolor{red}{E'_2}, \textcolor{blue}{Y}  \rangle \textcolor{red}{X'}\textcolor{red}{X} \textcolor{blue}{F_1} = \langle \textcolor{red}{E'_2}, \textcolor{blue}{Y}  \rangle [\textcolor{red}{E_1'}]_{<\bp} \textcolor{red}{X} \textcolor{blue}{F_1} =  [\textcolor{red}{E_1'}]_{<\bp} [\textcolor{red}{E'_2}]_{\geq \bp} \textcolor{blue}{F_1} \langle \textcolor{red}{E'_3}, \textcolor{blue}{F_2} \rangle 
\end{align*}
By \eqref{eqn:equation 1} and \eqref{eqn:equation 2}, the right-hand sides of the expressions above are equal to $\textcolor{blue}{F_2} \textcolor{red}{E_2'} \langle \textcolor{red}{E_1'}, \textcolor{blue}{F_1} \rangle$ and $\textcolor{red}{E_1'} \textcolor{blue}{F_1} \langle \textcolor{red}{E_2'}, \textcolor{blue}{F_2} \rangle$ respectively, which are equal to each other by \eqref{eqn:drinfeld double relation}. 

\medskip

\noindent For the fourth check in \eqref{eqn:check}, we apply \eqref{eqn:two 2} and \eqref{eqn:two 4} and we need to show that
\begin{equation}
\label{eqn:fourth check}
 \textcolor{blue}{Y}\textcolor{blue}{Y'} \langle \textcolor{red}{X} \textcolor{blue}{F_1}, \textcolor{blue}{F'_2} \textcolor{red}{X'}\rangle_{\bp} = \langle  \textcolor{blue}{Y}, \textcolor{blue}{Y'}\rangle_{\bp}\textcolor{blue}{F'_2} \textcolor{red}{X'}\textcolor{red}{X} \textcolor{blue}{F_1}
\end{equation}
where
\begin{align*} 
&\langle \textcolor{red}{X}, \tilde{F}' \rangle \langle \tilde{E}', \textcolor{blue}{Y} \rangle =  \langle \tilde{E}', \tilde{F}' \textcolor{blue}{F_2}\rangle \quad \Rightarrow \quad \textcolor{blue}{Y} \langle \textcolor{red}{X}, \tilde{F}' \rangle = \left[ \tilde{F}' \textcolor{blue}{F_2} \right]_{<\bp} \quad \text{and} \quad \textcolor{red}{X} \varepsilon(\textcolor{blue}{Y})  = \varepsilon(\textcolor{blue}{F_2}) \\
&\langle \textcolor{red}{X'}, \tilde{F} \rangle \langle \tilde{E}, \textcolor{blue}{Y'} \rangle = \langle \tilde{E}, \textcolor{blue}{F'_1} \tilde{F}\rangle \quad \Rightarrow \quad \textcolor{blue}{Y'} \langle \textcolor{red}{X'}, \tilde{F} \rangle = \left[ \textcolor{blue}{F_1'} \tilde{F} \right]_{\geq \bp} \quad \text{and} \quad  \textcolor{red}{X'}\varepsilon(\textcolor{blue}{Y'})   = \varepsilon(\textcolor{blue}{F_1'})
\end{align*}
By \eqref{eqn:footnote}, we have
\begin{align*}
&\text{LHS of \eqref{eqn:fourth check}} = \textcolor{blue}{Y}\textcolor{blue}{Y'} \langle \textcolor{red}{X}, \textcolor{blue}{F'_2} \rangle \langle \textcolor{red}{X'}, \textcolor{blue}{F_1}\rangle = [\textcolor{blue}{F_2'F_2}]_{<\bp} [\textcolor{blue}{F_1'F_1}]_{\geq \bp} \\
&\text{RHS of \eqref{eqn:fourth check}} = \textcolor{blue}{F'_2} \textcolor{red}{X'}\textcolor{red}{X} \textcolor{blue}{F_1} \varepsilon(\textcolor{blue}{Y}) \varepsilon(\textcolor{blue}{Y'}) =  \textcolor{blue}{F'_2} \varepsilon(\textcolor{blue}{F_1'}) \varepsilon(\textcolor{blue}{F_2}) \textcolor{blue}{F_1} = \textcolor{blue}{F'} \textcolor{blue}{F}
\end{align*}
The equality of the two expressions above is due to \eqref{eqn:equation 2}. \end{proof}

\medskip

\noindent We observe that the coproduct $\Delta_{\bp}$ preserves degrees. To see this, consider for instance any $\textcolor{red}{E} \textcolor{blue}{F}$ whose coproduct is $\textcolor{red}{E_1} \textcolor{blue}{Y} \otimes \textcolor{red}{X} \textcolor{blue}{F_1}$ as in \eqref{eqn:delta p plus}, and note that
$$
\text{deg}(\textcolor{red}{E} \textcolor{blue}{F}) - \text{deg}(\textcolor{red}{E_1} \textcolor{blue}{Y}) - \text{deg}(\textcolor{red}{X} \textcolor{blue}{F_1}) = \text{deg}(\textcolor{red}{E_2}) + \text{deg}(\textcolor{blue}{F_2}) - \text{deg}(\textcolor{red}{X}) - \text{deg}(\textcolor{blue}{Y}) 
$$
The quantity above is equal to 0 because of \eqref{eqn:declare 1} and the fact that the pairing $\langle \cdot, \cdot \rangle$ is non-zero only on elements of opposite degrees. The latter is also the reason why the pairing \eqref{eqn:footnote} is also non-zero only on elements of opposite degrees. 

\medskip 

\subsection{Universal $R$-matrices}
\label{sub:universal 2}

Let us now consider universal $R$-matrices as in Subsection \ref{sub:universal} for the algebra $\CS = \U$. Since this algebra has numerous coproducts and corresponding structures of a Drinfeld double, we may define partial universal $R$-matrices as follows. We will encounter the notation
\begin{align*} 
&\CS \ \widehat{\otimes} \ \CS = \prod_{(\bn,d) \in \zz \times \BZ} \CS_{\bn,d} \ \widehat{\otimes} \ \CS_{-\bn,-d} \\
&\CS \ \bar{\otimes} \ \CS = \prod_{(\bn,d) \in \zz \times \BZ} \CS_{\bn,d} \ \otimes \ \CS_{-\bn,-d}
\end{align*}
where 
\begin{equation}
\label{eqn:hat}
\CS_{\bn,d} \ \widehat{\otimes} \ \CS_{-\bn,-d} = \left \{ \sum_{k} A_k \otimes B_k \right \}
\end{equation}
in which the sum can go over infinitely many $A_k \in \CS_{\bn,d}$ and $B_k \in \CS_{-\bn,-d}$, but for any $N \in \BN$, all but finitely many $k$ have the property that $A_k = A_k'A_k''$ and $B_k = B_k' B_k''$ with $|\hdeg A_k'| \geq N, |\hdeg A_k''| \leq - N, |\hdeg B_k'| \leq - N, |\hdeg B_k''|\geq N$.

\medskip

\begin{proposition}
\label{prop:partial universal}

For any $\bp \in \rr$, the canonical tensors \footnote{See the discussion of Subsection \ref{sub:infinite slope 2} for certain technicalities necessary to make this precise.}
\begin{align}
&\CR_{\bp} \in \CA^{\geq \bp} \ \bar{\otimes} \ \CA^{\leq \bp} \subset \CS \ \bar{\otimes} \ \CS \quad \text{of the pairing} \quad \CS^-_{(-\binfty,\bp)} \otimes \CS^+_{(-\binfty,\bp)} \rightarrow \BK \label{eqn:partial 1} \\
&_{\bar{\bp}}\CR \in \CA^{\geq \bp} \ \bar{\otimes} \ \CA^{\leq \bp} \subset \CS \ \bar{\otimes} \ \CS \quad \text{of the pairing} \quad \CS^+_{[\bp,\binfty]} \otimes \CS^-_{[\bp, \binfty]} \rightarrow \BK \label{eqn:partial 2} 
\end{align} 
satisfy the properties
\begin{equation} 
\label{eqn:intertwine 1}
\CR_{\bp} \cdot \Delta_{\bp}(-) = \Delta(-) \cdot \CR_{\bp}
\end{equation} 
\begin{equation} 
\label{eqn:intertwine 2}
_{\bar{\bp}}\CR \cdot \Delta(-) = \Delta_{\bp}^{\emph{op}}(-) \cdot {_{\bar{\bp}}\CR}
\end{equation} 
(generalizing formulas of \cite{EKP} in the case of quantum affine algebras \eqref{eqn:quantum affine intro} and $\bp = \b0$).

\end{proposition}

\medskip 

\begin{proof} We will need the following technical results
\begin{equation}
\label{eqn:star 1}
[E'E]_{\geq \bp} = [E']_{\geq \bp} E
\end{equation}
\begin{equation}
\label{eqn:star 2}
[FF']_{<\bp} = F [F']_{<\bp}
\end{equation}	
for all $E \in \CS^+_{[\bp,\binfty]}$, $E' \in \CS^\geq$, $F \in \CS^-_{(-\binfty,\bp)}$, $F' \in \CS^\leq$. The formulas above follow from the non-degeneracy of the pairing and the facts that $\forall \tilde{F}' \in \CS^-_{[\bp,\binfty]}, \tilde{E}' \in \CS^+_{(-\binfty,\bp)}$
$$
\langle [E'E]_{\geq \bp}, \tilde{F}' \rangle = \langle E'E, \tilde{F}' \rangle =  \langle E, \tilde{F}'_1 \rangle  \langle E', \tilde{F}'_2 \rangle =  \langle E, \tilde{F}'_1 \rangle  \langle [E']_{\geq \bp}, \tilde{F}'_2 \rangle = \langle [E']_{\geq \bp} E, \tilde{F}' \rangle
$$
$$
\langle \tilde{E}', [FF']_{<\bp} \rangle = \langle \tilde{E}', FF' \rangle =  \langle \tilde{E}'_1, F \rangle  \langle \tilde{E}'_2, F' \rangle = \langle \tilde{E}'_1, F \rangle  \langle \tilde{E}'_2, [F']_{<\bp} \rangle = \langle \tilde{E}', F [F']_{<\bp} \rangle
$$
(the equalities above use $\tilde{E}_2' \in \CS^+_{(-\binfty,\bp)}$, $\tilde{F}'_2 \in \CS^-_{[\bp,\binfty]}$ and \eqref{eqn:unique 1}-\eqref{eqn:unique 4}). Similarly,
\begin{equation}
\label{eqn:star 3}
[E'E]_{< \bp} = E'[E]_{<\bp}
\end{equation}
\begin{equation}
\label{eqn:star 4}
[FF']_{\geq \bp} = [F]_{\geq \bp} F'
\end{equation}	
for all $E \in \CS^\geq$, $E' \in \CS^+_{(-\binfty,\bp)}$, $F \in \CS^\leq$, $F' \in \CS_{[\bp,\binfty]}^-$. The formulas above follow from the non-degeneracy of the pairing and the facts that $\forall \tilde{E} \in \CS^+_{[\bp,\binfty]}, \tilde{F} \in \CS^-_{(-\binfty,\bp)}$	
\begin{align*} 
&\langle [E'E]_{< \bp}, \tilde{F} \rangle = \langle E'E, \tilde{F} \rangle = \langle E, \tilde{F}_1 \rangle\langle E', \tilde{F}_2 \rangle = \langle [E]_{<\bp}, \tilde{F}_1 \rangle\langle E', \tilde{F}_2 \rangle = \langle E'[E]_{< \bp}, \tilde{F} \rangle \\
&\langle \tilde{E}, [FF']_{\geq \bp} \rangle = \langle \tilde{E}, FF' \rangle = \langle \tilde{E}_1, F \rangle \langle \tilde{E}_2, F' \rangle = \langle \tilde{E}_1, [F]_{\geq \bp} \rangle \langle \tilde{E}_2, F' \rangle = \langle \tilde{E}, [F]_{\geq \bp} F' \rangle
\end{align*} 
(the equalities above use $\tilde{E}_1 \in \CS^+_{[\bp,\binfty]}$, $\tilde{F}_1 \in \CS^-_{(-\binfty,\bp)}$ and \eqref{eqn:unique 1}-\eqref{eqn:unique 4}).

\medskip	
	
\noindent Let us now proceed with the proof of Proposition \ref{prop:partial universal}. Because formulas \eqref{eqn:intertwine 1} and \eqref{eqn:intertwine 2} are multiplicative in $-$, it suffices to prove them for $-$ equal to either of
\begin{equation}
\label{eqn:four}
\textcolor{red}{E} \in \CS^+_{[\bp,\binfty]}, \qquad \textcolor{blue}{F} \in \CS^-_{(-\binfty,\bp)}, \qquad \textcolor{red}{E'} \in \CS^+_{(-\binfty,\bp)}, \qquad \textcolor{blue}{F'} \in \CS^-_{[\bp,\binfty]}
\end{equation}
We will begin by proving the first two of these cases, and then turn to the remaining two on the next page. Specifically, unraveling equations \eqref{eqn:intertwine 1}-\eqref{eqn:intertwine 2} boils down to 
\begin{align} 
&\sum_k \textcolor{blue}{\tilde{F}_k} \textcolor{red}{E_1} \textcolor{blue}{Y} \otimes \textcolor{red}{\tilde{E}'_k} \textcolor{red}{X}   = \sum_k \textcolor{red}{E_1} \textcolor{blue}{\tilde{F}_k}  \otimes  \textcolor{red}{E_2} \textcolor{red}{\tilde{E}'_k} \label{eqn:tar 1} \\
&\sum_k \textcolor{blue}{\tilde{F}_k} \textcolor{blue}{\tilde{Y}}  \otimes \textcolor{red}{\tilde{E}'_k} \textcolor{red}{\tilde{X}} \textcolor{blue}{F_1}  = \sum_k \textcolor{blue}{F_1} \textcolor{blue}{\tilde{F}_k}  \otimes \textcolor{blue}{F_2} \textcolor{red}{\tilde{E}'_k}   \label{eqn:tar 2} \\
&\sum_k \textcolor{red}{\tilde{E}_k} \textcolor{red}{E_1} \otimes \textcolor{blue}{\tilde{F}'_k} \textcolor{red}{E_2} = \sum_k \textcolor{red}{X} \textcolor{red}{\tilde{E}_k} \otimes \textcolor{red}{E_1} \textcolor{blue}{Y} \textcolor{blue}{\tilde{F}'_k} \label{eqn:tar 3} \\
&\sum_k \textcolor{red}{\tilde{E}_k} \textcolor{blue}{F_1} \otimes \textcolor{blue}{\tilde{F}'_k} \textcolor{blue}{F_2} = \sum_k \textcolor{red}{\tilde{X}} \textcolor{blue}{F_1} \textcolor{red}{\tilde{E}_k} \otimes \textcolor{blue}{\tilde{Y}} \textcolor{blue}{\tilde{F}'_k} \label{eqn:tar 4}
\end{align}
where $\CR_{\bp} = \sum_k \textcolor{blue}{\tilde{F}_k} \otimes \textcolor{red}{\tilde{E}'_k}$ and $_{\bar{\bp}}\CR = \sum_k \textcolor{red}{\tilde{E}_k} \otimes \textcolor{blue}{\tilde{F}'_k}$ denote the canonical tensors of the pairings \eqref{eqn:partial 1} and \eqref{eqn:partial 2}, respectively. Above, $\textcolor{red}{X}, \textcolor{red}{\tilde{X}}, \textcolor{blue}{Y}, \textcolor{blue}{\tilde{Y}}$ are determined by
$$
\langle \textcolor{red}{X}, \tilde{F}' \rangle \langle \tilde{E}', \textcolor{blue}{Y} \rangle = \langle \textcolor{red}{E_2} \tilde{E}', \tilde{F}'\rangle \quad \Rightarrow \quad \textcolor{blue}{Y} \langle \textcolor{red}{X}, \tilde{F}' \rangle = [\tilde{F}'_1]_{<\bp} \langle \textcolor{red}{E_2} , \tilde{F}'_2\rangle
$$
$$
\langle \textcolor{red}{\tilde{X}}, \tilde{F}' \rangle \langle \tilde{E}', \textcolor{blue}{\tilde{Y}} \rangle =  \langle \tilde{E}', \tilde{F}' \textcolor{blue}{F_2}\rangle \quad \Rightarrow \quad \textcolor{red}{\tilde{X}} \langle \tilde{E}', \textcolor{blue}{\tilde{Y}} \rangle = [\tilde{E}'_1]_{\geq \bp} \langle \tilde{E}'_2, \textcolor{blue}{F_2}\rangle
$$
for all $\tilde{E}' \in \CS^+_{(-\binfty,\bp)}$ and $\tilde{F}' \in \CS^-_{[\bp,\binfty]}$. In the formulas below, we will use repeatedly the defining property \eqref{eqn:defining} of the canonical tensor of any pairing.

\medskip

\noindent To prove \eqref{eqn:tar 1}, take an arbitrary $V = FF'$ where $F\in \CS^-_{(-\binfty,\bp)}$, $F' \in \CS^-_{[\bp,\binfty]}$. Then
\begin{align*} 
&\langle \text{LHS of \eqref{eqn:tar 1}}, -  \otimes V \rangle  = F \textcolor{red}{E_1} \textcolor{blue}{Y} \langle \textcolor{red}{X}, F'\rangle =  F \textcolor{red}{E_1} [F'_1]_{<\bp} \langle \textcolor{red}{E_2}, F'_2 \rangle \stackrel{\eqref{eqn:simplify 2}}= \textcolor{red}{E_1} F_1 [F_1']_{<\bp} \langle \textcolor{red}{E_2}, F_2\rangle \langle \textcolor{red}{E_3}, F'_2\rangle \\
&\langle \text{RHS of \eqref{eqn:tar 1}}, - \otimes V \rangle = \textcolor{red}{E_1} [V_1]_{<\bp} \langle \textcolor{red}{E_2}, V_2\rangle = \textcolor{red}{E_1} [F_1F'_1]_{<\bp} \langle \textcolor{red}{E_2}, F_2 \rangle \langle \textcolor{red}{E_3}, F_2' \rangle
\end{align*}
The right-hand sides above are equal due to \eqref{eqn:star 2}, thus proving \eqref{eqn:tar 1}.

\medskip

\noindent To prove \eqref{eqn:tar 2}, pair both sides of the equation with an arbitrary $\tilde{E}' \in \CS^+_{(-\binfty,\bp)}$:
\begin{align*} 
&\langle \tilde{E}' \otimes -, \text{LHS of \eqref{eqn:tar 2}} \rangle = [\tilde{E}'_1]_{<\bp} \textcolor{red}{\tilde{X}} \textcolor{blue}{F_1} \langle \tilde{E}'_2, \textcolor{blue}{\tilde{Y}} \rangle = [\tilde{E}'_1]_{<\bp} [\tilde{E}_2']_{\geq \bp} \textcolor{blue}{F_1} \langle \tilde{E}'_3, \textcolor{blue}{F_2} \rangle \stackrel{\eqref{eqn:equation 1}}= \tilde{E}'_1 \textcolor{blue}{F_1} \langle \tilde{E}'_2, \textcolor{blue}{F_2} \rangle \\
&\langle \tilde{E}' \otimes -, \text{RHS of \eqref{eqn:tar 2}} \rangle = \langle \tilde{E}'_1, \textcolor{blue}{F_1} \rangle \textcolor{blue}{F_2} \tilde{E}'_2
\end{align*}
The right-hand sides above are equal due to \eqref{eqn:drinfeld double relation}, thus proving \eqref{eqn:tar 2}. 

\medskip

\noindent To prove \eqref{eqn:tar 3}, pair both sides of the equation with an arbitrary $\tilde{F}' \in \CS^-_{[\bp,\binfty]}$:
\begin{align*} 
&\langle \text{LHS of \eqref{eqn:tar 3}}, \tilde{F}' \otimes - \rangle = \langle \textcolor{red}{E_1}, \tilde{F}_1' \rangle \tilde{F}_2' \textcolor{red}{E_2} \\
&\langle \text{RHS of \eqref{eqn:tar 3}}, \tilde{F}' \otimes - \rangle = \textcolor{red}{E_1} \textcolor{blue}{Y} [\tilde{F}'_1]_{\geq \bp} \langle \textcolor{red}{X}, \tilde{F}'_2 \rangle = \textcolor{red}{E_1} [\tilde{F}_2']_{<\bp} [\tilde{F}'_1]_{\geq \bp} \langle \textcolor{red}{E_2}, \tilde{F}'_3 \rangle \stackrel{\eqref{eqn:equation 2}}= \textcolor{red}{E_1} \tilde{F}_1' \langle \textcolor{red}{E_2}, \tilde{F}'_2 \rangle
\end{align*}
The right-hand sides above are equal due to \eqref{eqn:drinfeld double relation}, thus proving \eqref{eqn:tar 3}. 

\medskip

\noindent To prove \eqref{eqn:tar 4}, take an arbitrary $U = E'E$ where $E'\in \CS^+_{(-\binfty,\bp)}$, $E \in \CS^+_{[\bp,\binfty]}$. Then
\begin{align*} 
&\langle - \otimes U, \text{LHS of \eqref{eqn:tar 4}} \rangle  = [U_1]_{\geq \bp}\textcolor{blue}{F_1} \langle U_2, \textcolor{blue}{F_2} \rangle  = [E_1'E_1]_{\geq \bp}\textcolor{blue}{F_1} \langle E_2, \textcolor{blue}{F_2} \rangle  \langle E'_2, \textcolor{blue}{F_3} \rangle \\
&\langle - \otimes U, \text{RHS of \eqref{eqn:tar 4}} \rangle = \textcolor{red}{\tilde{X}} \textcolor{blue}{F_1} E \langle E',\textcolor{blue}{\tilde{Y}} \rangle = [E'_1]_{\geq \bp} \textcolor{blue}{F_1} E \langle E'_2, \textcolor{blue}{F_2}\rangle \stackrel{\eqref{eqn:simplify 2}}= [E'_1]_{\geq \bp} E_1 \textcolor{blue}{F_1} \langle E_2, \textcolor{blue}{F_2}\rangle \langle E'_2, \textcolor{blue}{F_3}\rangle
\end{align*}
The right-hand sides of the expressions above are equal due to \eqref{eqn:star 1}, thus proving \eqref{eqn:tar 4}.

\medskip

\noindent Let us now deal with the last two cases in \eqref{eqn:four}, for which equations \eqref{eqn:intertwine 1}-\eqref{eqn:intertwine 2} boil down to the identities
\begin{align} 
&\sum_k \textcolor{blue}{\tilde{F}_k}\textcolor{blue}{Y'}\textcolor{red}{E'_2} \otimes \textcolor{red}{\tilde{E}'_k}\textcolor{red}{X'}  = \sum_k \textcolor{red}{E'_1} \textcolor{blue}{\tilde{F}_k}  \otimes  \textcolor{red}{E'_2} \textcolor{red}{\tilde{E}'_k} \label{eqn:tar 5} \\
&\sum_k \textcolor{blue}{\tilde{F}_k} \textcolor{blue}{\tilde{Y}'} \otimes \textcolor{red}{\tilde{E}'_k} \textcolor{blue}{F'_2} \textcolor{red}{\tilde{X}'} = \sum_k \textcolor{blue}{F'_1} \textcolor{blue}{\tilde{F}_k}  \otimes \textcolor{blue}{F'_2} \textcolor{red}{\tilde{E}'_k}   \label{eqn:tar 6} \\
&\sum_k \textcolor{red}{\tilde{E}_k} \textcolor{red}{E'_1} \otimes \textcolor{blue}{\tilde{F}'_k} \textcolor{red}{E'_2} = \sum_k \textcolor{red}{X'} \textcolor{red}{\tilde{E}_k} \otimes \textcolor{blue}{Y'}\textcolor{red}{E'_2} \textcolor{blue}{\tilde{F}'_k} \label{eqn:tar 7} \\
&\sum_k \textcolor{red}{\tilde{E}_k} \textcolor{blue}{F'_1} \otimes \textcolor{blue}{\tilde{F}'_k} \textcolor{blue}{F'_2} = \sum_k \textcolor{blue}{F'_2} \textcolor{red}{\tilde{X}'} \textcolor{red}{\tilde{E}_k} \otimes \textcolor{blue}{\tilde{Y}'}  \textcolor{blue}{\tilde{F}'_k} \label{eqn:tar 8}
\end{align}
where $\CR_{\bp} = \sum_k \textcolor{blue}{\tilde{F}_k} \otimes \textcolor{red}{\tilde{E}'_k}$ and $_{\bar{\bp}}\CR = \sum_k \textcolor{red}{\tilde{E}_k} \otimes \textcolor{blue}{\tilde{F}'_k}$ denote the canonical tensors of the pairings \eqref{eqn:partial 1} and \eqref{eqn:partial 2}, respectively. Above, $\textcolor{red}{X'}, \textcolor{red}{\tilde{X}'}, \textcolor{blue}{Y'}, \textcolor{blue}{\tilde{Y}'}$ are determined by
\begin{align*} 
&\langle \textcolor{red}{X'}, \tilde{F} \rangle \langle \tilde{E}, \textcolor{blue}{Y'} \rangle  = \langle \tilde{E}  \textcolor{red}{E_1'}, \tilde{F} \rangle \quad \Rightarrow \quad \textcolor{blue}{Y'} \langle \textcolor{red}{X'}, \tilde{F} \rangle = [ \tilde{F}_2 ]_{\geq\bp} \langle \textcolor{red}{E_1'}, \tilde{F}_1 \rangle \\ 
&\langle \textcolor{red}{\tilde{X}'}, \tilde{F} \rangle \langle \tilde{E}, \textcolor{blue}{\tilde{Y}'} \rangle = \langle \tilde{E}, \textcolor{blue}{F'_1} \tilde{F}\rangle \quad \Rightarrow \quad \textcolor{red}{\tilde{X}'}  \langle \tilde{E}, \textcolor{blue}{\tilde{Y}'} \rangle = [\tilde{E}_2]_{<\bp} \langle \tilde{E}_1, \textcolor{blue}{F_1'} \rangle
\end{align*}
for all $\tilde{E} \in \CS^+_{[\bp,\binfty]}, \tilde{F} \in \CS^-_{(-\binfty,\bp)}$.

\medskip 

\noindent To prove \eqref{eqn:tar 5}, pair both sides with an arbitrary $\tilde{F} \in \CS^-_{(-\binfty,\bp)}$:
\begin{align*} 
&\langle \text{LHS of \eqref{eqn:tar 5}}, - \otimes \tilde{F} \rangle = [\tilde{F}_2]_{<\bp} \textcolor{blue}{Y'}\textcolor{red}{E'_2} \langle \textcolor{red}{X'}, \tilde{F}_1 \rangle =  [\tilde{F}_3]_{<\bp} [\tilde{F}_2]_{\geq\bp} \textcolor{red}{E'_2} \langle \textcolor{red}{E_1'}, \tilde{F}_1 \rangle \stackrel{\eqref{eqn:equation 2}}=  \tilde{F}_2 \textcolor{red}{E'_2} \langle \textcolor{red}{E_1'}, \tilde{F}_1 \rangle \\
&\langle \text{RHS of \eqref{eqn:tar 5}}, - \otimes \tilde{F} \rangle = \textcolor{red}{E'_1} \tilde{F}_1 \langle \textcolor{red}{E'_2} , \tilde{F}_2 \rangle
\end{align*}
The right-hand sides above are equal due to \eqref{eqn:drinfeld double relation}, thus proving \eqref{eqn:tar 5}. 

\medskip

\noindent To prove \eqref{eqn:tar 6}, take an arbitrary $U = E'E$ where $E'\in \CS^+_{(-\binfty,\bp)}$, $E \in \CS^+_{[\bp,\binfty]}$. Then
\begin{align*}
&\langle U \otimes - , \text{LHS of \eqref{eqn:tar 6}}\rangle = E' \textcolor{blue}{F'_2} \textcolor{red}{\tilde{X}'} \langle E,\textcolor{blue}{\tilde{Y}'}\rangle = E' \textcolor{blue}{F'_2} [E_2]_{<\bp} \langle E_1, \textcolor{blue}{F_1'} \rangle \stackrel{\eqref{eqn:simplify 1}}= \textcolor{blue}{F_3'} E_2'[E_2]_{<\bp} \langle E'_1, \textcolor{blue}{F_2'} \rangle  \langle E_1, \textcolor{blue}{F_1'} \rangle \\
&\langle U \otimes -, \text{RHS of \eqref{eqn:tar 6}} \rangle = \textcolor{blue}{F_2'} [U_2]_{<\bp} \langle U_1, \textcolor{blue}{F_1'} \rangle =  \textcolor{blue}{F_3'} [E'_2E_2]_{<\bp} \langle E_1', \textcolor{blue}{F_2'} \rangle \langle E_1, \textcolor{blue}{F_1'} \rangle
\end{align*}
The right-hand sides above are equal due to \eqref{eqn:star 3}, thus proving \eqref{eqn:tar 6}.

\medskip

\noindent To prove \eqref{eqn:tar 7}, take an arbitrary $V = FF'$ where $F\in \CS^-_{(-\binfty,\bp)}$, $F' \in \CS^-_{[\bp,\binfty]}$. Then
\begin{align*}
&\langle \text{LHS of \eqref{eqn:tar 7}}, V \otimes - \rangle = [V_2]_{\geq \bp} \textcolor{red}{E_2'} \langle \textcolor{red}{E_1'}, V_1 \rangle = [F_2F_2']_{\geq \bp} \textcolor{red}{E_3'} \langle \textcolor{red}{E_1'}, F_1 \rangle \langle \textcolor{red}{E_2'}, F'_1 \rangle \\
&\langle \text{RHS of \eqref{eqn:tar 7}}, V \otimes - \rangle = \textcolor{blue}{Y'} \textcolor{red}{E_2'} F' \langle \textcolor{red}{X'}, F \rangle = [F_2]_{\geq \bp} \textcolor{red}{E_2'} F' \langle \textcolor{red}{E_1'}, F_1 \rangle \stackrel{\eqref{eqn:simplify 1}}=[F_2]_{\geq \bp} F'_2 \textcolor{red}{E_3'} \langle \textcolor{red}{E_1'}, F_1 \rangle \langle \textcolor{red}{E_2'}, F_1' \rangle
\end{align*}
The right-hand sides above are equal due to \eqref{eqn:star 4}, thus proving \eqref{eqn:tar 7}.

\medskip 

\noindent To prove \eqref{eqn:tar 8}, pair both sides with an arbitrary $\tilde{E} \in \CS^+_{[\bp,\binfty]}$:
\begin{align*} 
&\langle - \otimes \tilde{E}, \text{LHS of \eqref{eqn:tar 8}} \rangle = \tilde{E}_1 \textcolor{blue}{F_1'} \langle \tilde{E}_2, \textcolor{blue}{F_2'} \rangle \\
&\langle - \otimes \tilde{E}, \text{RHS of \eqref{eqn:tar 8}} \rangle = \textcolor{blue}{F_2'} \textcolor{red}{\tilde{X}'} [\tilde{E}_2]_{\geq \bp} \langle \tilde{E}_1, \textcolor{blue}{\tilde{Y}'} \rangle = \textcolor{blue}{F_2'} [\tilde{E}_2]_{<\bp} [\tilde{E}_3]_{\geq \bp} \langle \tilde{E}_1, \textcolor{blue}{F_1'} \rangle \stackrel{\eqref{eqn:equation 1}}= \textcolor{blue}{F_2'} \tilde{E}_2 \langle \tilde{E}_1, \textcolor{blue}{F_1'} \rangle
\end{align*}
The right-hand sides above are equal due to \eqref{eqn:drinfeld double relation}, thus proving \eqref{eqn:tar 8}. \end{proof}

\medskip

\noindent For any $\bp^1, \bp^2$ in $\rr$, we may consider 
\begin{align}
&_{\bp^2}\CR_{\bp^1} = {\CR_{\bp^2}^{-1}}  \cdot {\CR_{\bp^1}} \in \CS \ \bar{\otimes} \ \CS \label{eqn:consider 1} \\
&_{\bar{\bp}^2}\CR_{\bp^1} = {_{\bar{\bp}^2}\CR} \cdot {\CR_{\bp^1}} \in \CS \ \widehat{\otimes} \ \CS  \label{eqn:consider 2}
\end{align}
and so \eqref{eqn:intertwine 1}-\eqref{eqn:intertwine 2} imply
\begin{equation}
\label{eqn:intertwine 3}
({_{\bp^2}\CR_{\bp^1}}) \cdot \Delta_{\bp^1}(-) = \Delta_{\bp^2}(-) \cdot ( {_{\bp^2}\CR_{\bp^1}} )
\end{equation}
\begin{equation}
\label{eqn:intertwine 4}
( {_{\bar{\bp}^2}\CR_{\bp^1}}) \cdot \Delta_{\bp^1}(-) = \Delta^{\text{op}}_{\bp^2}(-) \cdot ( {_{\bar{\bp}^2}\CR_{\bp^1}} )
\end{equation}

\medskip

\subsection{Factorizations of $R$-matrices}
\label{sub:factorizations of r-matrices}

To further refine the constructions above, consider for any $\bp \in \rr$ the canonical tensor
\begin{equation}
\label{eqn:slope r-matrix}
\CP_{\bp} \in \CB^{+}_{\bp} \ \bar{\otimes} \ \CB^{-}_{\bp} \subset \CS \ \bar{\otimes} \ \CS
\end{equation}
of the first pairing in \eqref{eqn:perfect pairing 1}. We then use \eqref{eqn:pair slopes basic} to obtain the identities
\begin{align} 
&\CR_{\bp^1} = \prod_{t \in (-\infty,t_1)}^{\rightarrow} \CP_{\bp(t)}^{\text{op}}\label{eqn:factorization r-matrix 1} \\
&_{\bar{\bp}^2}\CR = \prod_{t \in [t_2,\infty]}^{\rightarrow} \CP_{\bp'(t)}  \label{eqn:factorization r-matrix 2} 
\end{align}
for any catty-corner curves $\bp : (-\infty,t_1) \rightarrow \rr$ and $\bp' : [t_2,\infty] \rightarrow \rr$ with $\bp(t_1) = \bp^1$ and $\bp'(t_2) = \bp_2$. In \eqref{eqn:factorization r-matrix 1}, we set $\CP^{\text{op}} = \text{swap}(\CP)$. If we put the formulas above together, we obtain the following factorization of the universal $R$-matrix \eqref{eqn:consider 2}
\begin{equation}
\label{eqn:factorization r-matrix 3}
_{\bar{\bp}^2}\CR_{\bp^1} =  \prod_{t \in [t_2,\infty]}^{\rightarrow} \CP_{\bp'(t)} \prod_{t \in (-\infty,t_1)}^{\rightarrow} \CP_{\bp(t)}^{\text{op}}
\end{equation}
When $\bp^1=\bp^2$, \eqref{eqn:factorization r-matrix 3} generalizes the well-known formulas for factorizations of $R$-matrices of \cite{KT, KR, LS, LSS, R}, see also \cite{COZZ1, COZZ2, N Stable, OS, Z} for a geometric incarnation.

\medskip 

\begin{proposition}
\label{prop:catty}

For any $\bp^1 \geq \bp^2$ in $\rr$, we have the formula
\begin{equation}
\label{eqn:factorization r-matrix 4}
{_{\bp^2}\CR_{\bp^1}} = \prod_{t \in [t_2,t_1)}^{\rightarrow} \CP^{\emph{op}}_{\bp(t)}
\end{equation}
for any catty-corner curve $\bp : [t_2,t_1] \rightarrow \rr$ with $\bp(t_1) = \bp_1$, $\bp(t_2) = \bp_2$. Thus, $_{\bp^2}\CR_{\bp^1}$ is the canonical tensor of the perfect pairing
\begin{equation}
\label{eqn:hr}
\CS^-_{[\bp^2,\bp^1)} \otimes \CS^+_{[\bp^2,\bp^1)}\rightarrow \BK
\end{equation}

\end{proposition}

\medskip 

\begin{proof}

The first statement of the Proposition is an immediate consequence of \eqref{eqn:factorization r-matrix 1}, since any catty-corner curve can be extended to $-\infty$. The second statement of the Proposition is an immediate consequence of \eqref{eqn:pair slopes basic}. \end{proof}

\medskip 

\subsection{Infinite slope revisited}
\label{sub:infinite slope 2}

Because of the description of the universal $R$-matrix ${_{\bp^2}\CR_{\bp^1}}$ as the canonical tensor of the pairing \eqref{eqn:hr}, all of its homogeneous summands are finite sums (this statement actually holds for all $\bp^1,\bp^2 \in \rr$). This is not true for ${_{\bar{\bp}^2}\CR_{\bp^1}}$, whose homogeneous summands lie in the completion
$$
\CS \ \widehat{\otimes} \ \CS
$$
defined in \eqref{eqn:hat}. However, there is a bigger problem than the necessity of completions, which we will now address. While the slope $R$-matrices \eqref{eqn:slope r-matrix} are well-defined for all $\bp \in \rr$, when $\bp = \binfty$ we define
\begin{equation}
\label{eqn:infinite factorization}
\CP_{\binfty} = \CP'_{\binfty} \CP''_{\binfty}
\end{equation}
where $\CP'_{\binfty}$ and $\CP''_{\binfty}$ are the canonical tensors of
\begin{equation}
\label{eqn:first pairing}
\BK[\kappa^+_i]_{i \in I} \otimes \BK[\kappa^-_i]_{i \in I} \xrightarrow{\langle \cdot, \cdot \rangle} \BK
\end{equation}
and 
\begin{equation}
\label{eqn:second pairing}
\BK[p_{i,d}]_{i \in I, d\geq 1} \otimes \BK[p_{i,-d}]_{i \in I, d \geq 1} \xrightarrow{\langle \cdot, \cdot \rangle} \BK
\end{equation}
respectively. And while the canonical tensor $\CP_{\binfty}''$ of \eqref{eqn:second pairing} is well-defined (subject to the non-degeneracy assumption in Subsection \ref{sub:infinite slope 1}) and rather easy to compute, there is no reasonable canonical tensor of the pairing \eqref{eqn:first pairing}. The workaround for this issue is the usual one: replace the ground field $\BK$ with $\BK[[\hbar]]$, and replace
$$
\kappa^\pm_i \quad \text{by} \quad e^{\hbar H^\pm_i}
$$
for primitive elements $\{H^+_i, H^-_i\}_{i \in I}$. If the $|I| \times |I|$ matrix with coefficients $\langle H^+_i, H^-_j \rangle$ is invertible, then the pairing
\begin{equation}
	\label{eqn:third pairing}
	\BK[[\hbar]][H^+_i]_{i \in I} \otimes \BK[[\hbar]][H^-_i]_{i \in I} \xrightarrow{\langle \cdot, \cdot \rangle} \BK((\hbar))
\end{equation}
has a canonical tensor, which serves as a replacement for the factor $\CP'_{\binfty}$ in \eqref{eqn:infinite factorization}. We will ignore this issue in what follows, as both the problem and its workaround are well-known and similar in the case at hand to the classic cases such as $U_q(\fsl_2)$.

\medskip

\subsection{Quantum affine algebras}
\label{sub:quantum affine}

Let $\fg$ be a simple finite-dimensional Lie algebra, with a set of simple roots $\{\alpha_i\}_{i \in I}$ and associated Cartan matrix 
$$
C = \left(c_{ij} = \frac {d_{ij}}{d_{i}} \in \BZ \right)_{i,j \in I}
$$
where we abbreviate $d_{ij} = (\alpha_i, \alpha_j)$ and $d_i = \frac {(\alpha_i, \alpha_i)}2$. We work over the field $\BK = \BC$, fix $q \in \BC^*$ not a root of unity, and consider the rational functions
\begin{equation}
\label{eqn:zeta particular}
\zeta_{ij}(x) = \frac {(q^{-d_{ij}}-x)(-x)^{-\delta_{i>j}}}{(1-x)^{\delta_{ij}}}
\end{equation}
with respect to an arbirary total order $<$ on $I$. Then the corresponding quantum loop algebra is denoted by
\begin{equation}
\label{eqn:quantum loop algebra}
\UU = \Big(\U \text{ for the choice \eqref{eqn:zeta particular}} \Big) 
\end{equation}
\footnote{We tacitly modify the pairing of $\U$ by dividing the RHS of \eqref{eqn:pair formula}-\eqref{eqn:pair formula opposite} by the constant
$$
\left(q_{i_1}^{-1} - q_{i_1} \right) \dots \left( q_{i_n}^{-1} - q_{i_n}\right)
$$
where $q_i = q^{d_i}$, in order to match the existing conventions for quantum affine algebras. This has a trivial effect on our algebras, which can be countered by rescaling either the $e$ or $f$ generators.} Drinfeld constructed an isomorphism
$$
U_q(\widehat{\fg})_{c=1} =  \BC \Big \langle e_i,f_i, \kappa_i, \kappa_i^{-1}, e_0, f_0 \Big \rangle_{i \in I} \Big/ \Big(\text{relations}\Big) 
$$
\begin{equation}
\label{eqn:drinfeld iso}
\xrightarrow{\Xi} \UU \Big/ \Big(\kappa_i^+ \kappa_i^- = 1 \Big)_{i \in I}
\end{equation}
by sending for all $i \in I$
\begin{equation}
\label{eqn:send}
\kappa_i^{\pm 1} \mapsto \kappa_i^\pm, \qquad e_i \mapsto e_{i,0}, \qquad f_i \mapsto f_{i,0} 
\end{equation}
\begin{align} 
&e_{0} \mapsto c (\kappa_{\bth}^+)^{-1} \textcolor{blue}{F} \quad \ \ \text{ where } \quad \textcolor{blue}{F} = [f_{i_1,0},[f_{i_2,0},\dots,[f_{i_{k},0}, f_{j,1}]_q\dots]_q]_q \label{eqn:e bullet} \\
&f_{0} \mapsto c' (\kappa_{\bth}^-)^{-1} \textcolor{red}{E'} \quad \text{ where } \quad \textcolor{red}{E'} = [e_{i_1,0},[e_{i_2,0},\dots,[e_{i_{k},0}, e_{j,-1}]_q\dots]_q]_q \label{eqn:f bullet}
\end{align} 
where $c,c'$ are non-zero constants, we write
$$
\kappa_{\bth}^\pm = \prod_i (\kappa_i^\pm)^{\theta_i}
$$
with $\bth \in \nn$ being the longest root and  $[x,y]_q = xy - q^{(\hdeg x, \hdeg y)} yx$. In the formulas above, $i_1,\dots,i_k,j \in I$ are such that
$$
\bs^{i_1}+\dots+\bs^{i_k} + \bs^j = \bth
$$
and $\bs^{i_a}+\dots+\bs^{i_k} +\bs^j$ is a root $\forall a$. It is clear that $\deg \textcolor{blue}{F} = (-\bth,1)$ and $\deg \textcolor{red}{E'} = (\bth,-1)$. We proved in \cite[Proposition 3.23]{N Cat} that the isomorphism \eqref{eqn:drinfeld iso} sends
\begin{align} 
&\Xi \left(U_q(\widehat{\fb}^+)_{c=1} \right) = \CA^{\geq \b0} \label{eqn:borel positive} \\ 
&\Xi \left(U_q(\widehat{\fb}^-)_{c=1} \right) = \CA^{\leq \b0} \label{eqn:borel negative} 
\end{align}
where $U_q(\widehat{\fb}^+)$ and $U_q(\widehat{\fb}^-)$ denote the Borel subalgebras of $U_q(\widehat{\fg})$ generated by $\{e_i,e_0,\kappa_i\}_{i \in I}$ and $\{f_i, f_0,\kappa_i^{-1}\}_{i \in I}$, respectively.

\medskip

\begin{theorem}
\label{thm:affine}

The isomorphism $\Xi$ intertwines the Drinfeld-Jimbo coproduct on the quantum affine algebra with the coproduct $\Delta_{\b0}$ on $\UU \cong \CS$.

\end{theorem} 

\medskip 

\begin{proof} To show that the Drinfeld-Jimbo coproduct matches the coproduct $\Delta_{\b0}$, it is enough to do so on the generators $\{e_i,f_i\}_{i \in I \sqcup 0}$.  This is obvious for $i \in I$, since relations \eqref{eqn:coproduct slope plus} and \eqref{eqn:coproduct slope minus} give us
\begin{align*} 
&\Delta_{\b0}(e_{i,0}) = \kappa_i^+ \otimes e_{i,0} + e_{i,0} \otimes 1 \\
&\Delta_{\b0}(f_{i,0}) = 1 \otimes f_{i,0} + f_{i,0} \otimes \kappa_i^-
\end{align*}
and this matches the Drinfeld-Jimbo coproduct. For the generators $e_{0}$ and $f_{0}$, it suffices to show that
\begin{align}
&\Delta_{\b0}(\textcolor{blue}{F}) = \textcolor{blue}{F} \otimes \kappa_{\bth}^+ + 1 \otimes \textcolor{blue}{F} \label{eqn:dj 1} \\ 
&\Delta_{\b0}(\textcolor{red}{E'}) = \kappa_{\bth}^- \otimes \textcolor{red}{E'} + \textcolor{red}{E'} \otimes 1 \label{eqn:dj 2}
\end{align}
To prove \eqref{eqn:dj 1}, note that
\begin{equation}
\label{eqn:rob 1}
\Delta(\textcolor{blue}{F}) = \textcolor{blue}{F}_1 \otimes \textcolor{blue}{F}_2 = 1 \otimes \textcolor{blue}{F} + \textcolor{blue}{F} \otimes \kappa^-_{\bth} + \sum_k \textcolor{blue}{G_k} \otimes \textcolor{blue}{H_k}
\end{equation}
for various $\textcolor{blue}{G_k} \in \CA^{\geq \b0}$ and $\textcolor{blue}{H_k} \in \CA^{\leq \b0}$ of horizontal degree contained strictly between $\b0$ and $-\bth$ (or alternatively we may have $\hdeg \textcolor{blue}{H_k} = \b0$ and $\vdeg \textcolor{blue}{H_k}<0$). By formula \eqref{eqn:declare 1 def}, there are various $\textcolor{red}{X} \in \CS^+_{[\b0,\binfty]}$, $\textcolor{blue}{Y} \in \CS^-_{(-\binfty,\b0)}$ such that
\begin{equation}
\label{eqn:rob 2}
\Delta_{\b0}(\textcolor{blue}{F}) =  \textcolor{blue}{Y} \otimes \textcolor{red}{X} \textcolor{blue}{F_1}, \quad \text{where} \quad \textcolor{red}{X} \langle \tilde{E}', \textcolor{blue}{Y} \rangle =[ \tilde{E}'_1 ]_{\geq \b0} \langle \tilde{E}'_2, \textcolor{blue}{F_2}\rangle, \  \forall \tilde{E}' \in \CS^+_{(-\binfty,\b0)}
\end{equation}
Let us plug the three terms in the right-hand side of \eqref{eqn:rob 1} into the right-hand side of \eqref{eqn:rob 2}, and analyze their contributions to $\Delta_{\b0}(\textcolor{blue}{F})$:

\medskip

\begin{itemize}[leftmargin=*]
	
\item $1 \otimes \textcolor{blue}{F}$: for any homogeneous element $\tilde{E}' \in \CS^+_{(-\binfty,\b0)}$, the right-hand side of 
\begin{equation}
\label{eqn:jd 1}
\textcolor{red}{X} \langle \tilde{E}', \textcolor{blue}{Y} \rangle =[ \tilde{E}'_1 ]_{\geq \b0} \langle \tilde{E}'_2, \textcolor{blue}{F}\rangle
\end{equation}
can be non-zero only if $\tilde{E}'$ is proportional to $\textcolor{red}{E'}$ of  \eqref{eqn:f bullet}. This is because
$$
\CA^{\leq \b0} \cong U_q(\widehat{\fb}^-)_{c=1}
$$
has no elements $\tilde{E}'$ of vertical degree $-1$ and horizontal degree $> \bth$, and only such elements could afford an $\tilde{E}'_2$ of degree $(\bth,-1)$. Therefore, the only non-zero contribution in \eqref{eqn:jd 1} arises from
$$
\tilde{E}' \sim \textcolor{red}{E} \quad \text{and} \quad \tilde{E}'_1 \otimes \tilde{E}'_2 \sim \kappa_{\bth}^+ \otimes  \textcolor{red}{E}
$$
We conclude that $\textcolor{blue}{Y} \otimes \textcolor{red}{X} = \textcolor{blue}{F} \otimes \kappa_{\bth}^+$, which leads to the first term in \eqref{eqn:dj 1}. 

\medskip 

\item $\textcolor{blue}{F} \otimes \kappa_{\bth}^-$: in this case, the condition 
$$
\textcolor{red}{X} \langle \tilde{E}', \textcolor{blue}{Y} \rangle =[ \tilde{E}'_1 ]_{\geq \b0} \langle \tilde{E}'_2, \kappa_{\bth}^-\rangle
$$
implies $\textcolor{blue}{Y} \otimes \textcolor{red}{X} = 1 \otimes 1$, which yields the second term in the right-hand side of \eqref{eqn:dj 1}. \\

\item $\textcolor{blue}{G_k} \otimes \textcolor{blue}{H_k}$ for various $\textcolor{blue}{G_k} \in \CA^{\geq \b0}$ and $\textcolor{blue}{H_k} \in \CA^{\leq \b0}$ of horizontal degree contained strictly between $\b0$ and $-\bth$: in this case, we have for all $\tilde{E}' \in \CS^+_{(-\binfty,\b0)}$
$$
\textcolor{red}{X} \langle \tilde{E}', \textcolor{blue}{Y} \rangle =[ \tilde{E}'_1 ]_{\geq \b0} \langle \tilde{E}'_2, \textcolor{blue}{H_k}\rangle = 0
$$ 
because $\CA^{\leq \b0}$ contains no elements of vertical degree $>0$ and negative horizontal degree (which are the only ones that could pair non-trivially with $\tilde{E}_2' \in \CS^+_{(-\binfty,\b0)}$). Thus, in this case there is no contribution to the right-hand side of \eqref{eqn:dj 1}.

\end{itemize}

\medskip

\noindent Now let us prove \eqref{eqn:dj 2}. To this end, note that
\begin{equation}
\label{eqn:rob 3}
\Delta(\textcolor{red}{E'}) = \textcolor{red}{E'_1} \otimes \textcolor{red}{E'_2} = \textcolor{red}{E'} \otimes 1 + \kappa^+_{\bth} \otimes \textcolor{red}{E'} + \sum_k \textcolor{red}{G_k} \otimes \textcolor{red}{H_k}
\end{equation}
for various $\textcolor{red}{G_k} \in \CA^{\geq \b0}$ and $\textcolor{red}{H_k} \in \CA^{\leq \b0}$ of horizontal degree contained strictly between $\b0$ and $\bth$ (or alternatively we may have $\hdeg \textcolor{red}{G_k} = \b0$ and $\vdeg \textcolor{red}{G_k}>0$). By formula \eqref{eqn:declare 2 def}, there are various $\textcolor{red}{X'} \in \CS^+_{(-\binfty,\b0)}$, $\textcolor{blue}{Y'} \in \CS^-_{[\b0,\binfty]}$ such that
\begin{equation}
\label{eqn:rob 4}
\Delta_{\b0}(\textcolor{red}{E'}) =  \textcolor{blue}{Y'} \textcolor{red}{E_2'} \otimes \textcolor{red}{X'}, \quad \text{where} \quad \textcolor{blue}{Y'} \langle \textcolor{red}{X'}, \tilde{F} \rangle  = [ \tilde{F}_2 ]_{\geq \b0}  \langle \textcolor{red}{E_1'}, \tilde{F}_1 \rangle, \  \forall \tilde{F} \in \CS^-_{(-\binfty,\b0)}
\end{equation}
Let us plug the three terms in the right-hand side of \eqref{eqn:rob 3} into the right-hand side of \eqref{eqn:rob 4}, and analyze their contributions to $\Delta_{\b0}(\textcolor{red}{E'})$:

\medskip

\begin{itemize}[leftmargin=*]
	
\item $\textcolor{red}{E'} \otimes 1$: for any homogeneous element $\tilde{F} \in \CS^-_{(-\binfty,\b0)}$, the right-hand side of 
\begin{equation}
\label{eqn:jd 2}
\textcolor{blue}{Y'} \langle \textcolor{red}{X'}, \tilde{F} \rangle  = [ \tilde{F}_2 ]_{\geq \b0}  \langle \textcolor{red}{E'}, \tilde{F}_1 \rangle
\end{equation}
can be non-zero only if $\tilde{F}$ is proportional to $\textcolor{blue}{F}$ of  \eqref{eqn:e bullet}. This is because
$$
\CA^{\geq \b0} \cong U_q(\widehat{\fb}^+)_{c=1}
$$
has no elements $\tilde{F}$ of vertical degree $1$ and horizontal degree $< - \bth$, and only such elements could afford an $\tilde{F}_1$ of degree $(-\bth,1)$. Therefore, the only non-zero contribution in \eqref{eqn:jd 2} arises from
$$
\tilde{F} \sim \textcolor{blue}{F} \quad \text{and} \quad \tilde{F}_1 \otimes \tilde{F}_2 \sim \textcolor{blue}{F} \otimes \kappa_{\bth}^- 
$$
We conclude that $\textcolor{blue}{Y'} \otimes \textcolor{red}{X'} = \kappa_{\bth}^- \otimes \textcolor{red}{E'}$, which leads to the first term in \eqref{eqn:dj 2}. 

\medskip 

\item $\kappa_{\bth}^+ \otimes \textcolor{red}{E'}$: in this case, the condition 
$$
\textcolor{blue}{Y'} \langle \textcolor{red}{X'}, \tilde{F} \rangle  = [ \tilde{F}_2 ]_{\geq \b0}  \langle \kappa_{\bth}^+, \tilde{F}_1 \rangle
$$
implies $\textcolor{blue}{Y'} \otimes \textcolor{red}{X'} = 1 \otimes 1$, which yields the second term in the right-hand side of \eqref{eqn:dj 2}. \\

\item $\textcolor{red}{G_k} \otimes \textcolor{red}{H_k}$ for various $\textcolor{red}{G_k} \in \CA^{\geq \b0}$ and $\textcolor{red}{H_k} \in \CA^{\leq \b0}$ of horizontal degree contained strictly between $\b0$ and $-\bth$: in this case, we have for all $\tilde{F} \in \CS^-_{(-\binfty,\b0)}$
$$
\textcolor{blue}{Y'} \langle \textcolor{red}{X'}, \tilde{F} \rangle  = [ \tilde{F}_2 ]_{\geq \b0}  \langle \textcolor{red}{G_k}, \tilde{F}_1 \rangle = 0
$$ 
because $\CA^{\geq \b0}$ contains no elements of vertical degree $<0$ and positive horizontal degree (which are the only ones that could pair non-trivially with $\tilde{F}_1 \in \CS^-_{(-\binfty,\b0)}$). Thus, in this case there is no contribution to the right-hand side of \eqref{eqn:dj 2}.
	
\end{itemize}

\end{proof} 	

\medskip

\noindent In principle, following the proof above would lead to explicit formulas for $\Delta_{\b0}$ akin to the formulas for the Drinfeld-Jimbo coproduct obtained in \cite[Theorem 10.14]{FiT} for $\fsl_n$ (or for general $\fg$, following in the footsteps of \cite[Theorem 4]{Da}).

\medskip 

\begin{remark}
\label{rem:matrix}

We recall the case $\fg = \widehat{\fsl}_n$, when there exists a well-known two-parameter version of \eqref{eqn:quantum loop algebra} that yields quantum toroidal $\fgl_n$ (see \cite{FT}). In this case, we constructed in \cite{N Tale} a so-called double matrix shuffle algebra $\CA = \CA^\geq \otimes \CA^\leq$. We showed that $\CA$ is a Drinfeld double with respect to bialgebra structures on $\CA^{\geq}$, $\CA^{\leq}$, and a bialgebra pairing between them, that were defined in \loccitt. Using the results of \cite{N PBW}, we showed that there is an algebra isomorphism
\begin{equation}
\label{eqn:matrix iso}
\CA \cong \Big(\text{quantum toroidal }\fgl_n\Big)
\end{equation}
with respect to which the subalgebras $\CA^{\geq}, \CA^{\leq}$ map isomorphically onto the subalgebras $\CA^{\geq \b0}, \CA^{\leq \b0}$ in the present paper. The fact that these isomorphisms respect the pairings is due to the fact that all our algebras factor into slope subalgebras \eqref{eqn:double slope} that are isomorphic to the quantum affine algebra of $\fsl_{\frac nd}^{\otimes d}$ for various $d|n$ (\cite{N PBW, N Tale}), and they inherit the usual bialgebra pairing from the aforementioned quantum affine algebras. The fact that \eqref{eqn:matrix iso} intertwines the coproduct of \loccit and the coproduct $\Delta_{\b0}$ defined in the present paper is due to the fact that both coproducts are dual to compatible algebra structures under one and the same non-degenerate pairing.

\end{remark}

\medskip

\subsection{The pairing explicit}
\label{eqn:pairing}

We conclude this section with an explicit formula for the pairing \eqref{eqn:pair shuffle} in the case of the quantum affine algebra $U_q(\widehat{\fg})$ associated to a simple Lie algebra $\fg$, which will provide an answer to Problem 2 of Subsection \ref{sub:techniques} in this particular case. The contents of the present Subsection will not be used anywhere else in this paper, but we include them to fill a gap in the literature. We keep the notation as in Subsection \ref{sub:quantum affine}, and we recall from \cite[Theorem 5.19]{NT} that
\begin{equation}
\label{eqn:wheel}
\CS^\pm = \left\{E \in \CV^\pm \text{ s.t. } E\Big|_{z_{i1} = z_{j1} q^{d_{ij}}, z_{i2} = z_{j1} q^{d_{ij}+d_{ii}}, \dots,  z_{i, 1-c_{ij}} = z_{j1} q^{-d_{ij}}} = 0, \forall i \neq j\right\}
\end{equation}
(if $E$ above does not have enough variables to make the specialization, the vanishing condition is vacuous). The conditions on $E$ in \eqref{eqn:wheel} are called \emph{wheel conditions} and were introduced by \cite{E}, following the seminal work of \cite{FO}. The goal of Problem 1 in Subsection \ref{sub:techniques} is to find descriptions of $\CS^\pm$ for general $(I,\BK,\zeta_{ij}(x))$ that are similar to \eqref{eqn:wheel} above. As for Problem 2, we will provide in Lemma \ref{lem:contour} an analogue of  \cite[Proposition 2.22]{N Symmetric}, \cite[Proposition 3.3]{N Wheel}, \cite[Proposition 3.10 and formula (3.26)]{N Reduced}. 

\medskip 

\noindent In the upcoming Lemma, we will use the notation
\begin{equation}
\label{eqn:iterated residue}
\underset{(z_1,\dots,z_n) = (w q^{n-1}, \dots, w q^{1-n})}{\text{Res}} = \underset{z_1 = z_2 q^2}{\text{Res}} \  \underset{z_2 = z_3 q^2}{\text{Res}}  \dots \ \underset{z_{n-1}=z_nq^2}{\text{Res}}
\end{equation}
followed by relabeling the variable $z_n$ as $w q^{1-n}$. Moreover, 
$$
\int_{|w| = r} G(w)
$$
denotes the contour integral of $G(w)\frac {dw}{2\pi i w}$ over the circle of radius $r$ centered at the origin. In all subsequent contour integrals, we assume that $|q|>1$, although this is not important due to our result being purely algebraic. Write $q_i = q^{d_i}$ for all $i \in I$.

\medskip

\begin{lemma}
\label{lem:contour}

For any $E \in \CS_{\bn,d}$ and $F \in \CS_{-\bn,-d}$, we have for any $r \in \BR_{>0}$
\begin{multline}
\label{eqn:contour}
\Big \langle E, F \Big \rangle = \sum^{I\text{-tuples of partitions}}_{\left(n_{i1} \geq n_{i2} \geq \dots \right) \vdash n_i, \forall i \in I} \\ \int_{|w_{i1}| = |w_{i2}| = \dots = r} \emph{Res } \left[ \frac {E(z_{i1},\dots,z_{in_i})_{i \in I} F(z_{i1},\dots,z_{in_i})_{i \in I}}{\prod_{(i,a) \neq (j,b)} \zeta_{ij} \left(\frac {z_{ia}}{z_{jb}} \right)} \right]
\end{multline}
where the residue above is defined by taking
$$
\underset{\left(z_{i1},\dots,z_{in_{i1}}\right) = \left(w_{i1} q_i^{n_{i1}-1}, \dots, w_{i1} q_i^{1-n_{i1}} \right)}{\emph{Res}} \ \  \underset{\left(z_{i,n_{i1}+1},\dots,z_{i,n_{i1}+n_{i2}}\right) = \left(w_{i2} q_i^{n_{i2}-1}, \dots, w_{i2} q_i^{1-n_{i2}}\right)}{\emph{Res}} \ \dots
$$
over all $i \in I$, with the notation as in \eqref{eqn:iterated residue}. 

\end{lemma}

\medskip

\begin{proof} By the very definition of $\CS^+$ in \eqref{eqn:spherical def}, it suffices to prove \eqref{eqn:contour} for 
\begin{equation}
\label{eqn:the e}
E = e_{i_1,k_1} \dots e_{i_n,k_n}
\end{equation}
where $i_1,\dots,i_n \in I$, $k_1,\dots,k_n \in \BZ$. For any $m \in \{1,\dots,n\}$, consider the quantity
\begin{equation}
\label{eqn:xm}
X_m = \sum^{\text{fair partition}}_{\{m,\dots,n\}=A_1 \sqcup \dots \sqcup A_t} \int_{|z_1| \gg \dots \gg |z_{m-1}| \gg |w_1| = \dots = |w_t| = r} 
\end{equation}
$$
\left[ \underset{(z_{a_s^{(1)}}, \dots, z_{a_s^{(n_s)}}) = (w_s q_{\iota(A_s)}^{n_s-1}, \dots, w_sq_{\iota(A_s)}^{1-n_s})}{\text{Res}}  \frac {z_1^{k_1}\dots z_n^{k_n} F(z_1,\dots,z_n)}{\prod_{1\leq a < b \leq n} \zeta_{i_bi_a} \left(\frac {z_b}{z_a} \right)} \right]_{\forall s \in \{1,\dots,t\}}
$$
where a \emph{fair partition} (relative to the fixed $i_1,\dots,i_n \in I$ in \eqref{eqn:the e}) consists of sets
$$
A_s = \left\{a_s^{(1)} < \dots < a_s^{(n_s)} \right\} \subseteq \{m,\dots,n\}
$$
of arbitrary length $n_s$, such that 
$$
i_{a_s^{(1)}} = \dots  = i_{a_s^{(n_s)}} =: \iota(A_s)
$$
for all $s \in \{1,\dots,t\}$. 
	
\medskip 
	
\begin{claim}
\label{claim:induction}
		
We have $X_m = X_{m-1}$ for all $m \in \{2,\dots,n\}$.
		
\end{claim}
	
\medskip 	
	
\noindent Let us first show how Claim \ref{claim:induction} implies \eqref{eqn:contour}. By iterating Claim \ref{claim:induction} a number of $n-1$ times, we conclude that $X_n = X_1$, or more explicitly
\begin{equation}
\label{eqn:end induction 1}
\Big \langle e_{i_1,k_1} \dots e_{i_n,k_n}, F \Big \rangle = \sum^{\text{fair partition}}_{\{1,\dots,n\}=A_1 \sqcup \dots \sqcup A_t} \int_{|w_1| = \dots = |w_t| = r} 
\end{equation}
$$
\left[ \underset{(z_{a_s^{(1)}}, \dots, z_{a_s^{(n_s)}}) = (w_s q_{\iota(A_s)}^{n_s-1}, \dots, w_sq_{\iota(A_s)}^{1-n_s})}{\text{Res}} \frac {z_1^{k_1}\dots z_n^{k_n} F(z_1,\dots,z_n)}{\prod_{1\leq a < b \leq n} \zeta_{i_bi_a} \left(\frac {z_b}{z_a} \right)} \right]_{\forall s \in \{1,\dots,t\}} 
$$
However, we claim that the residue on the second line of \eqref{eqn:contour} satisfies
\begin{equation}
\label{eqn:ofofof}
\text{Res} \left[ \text{Sym}\left( \frac {z_1^{k_1}\dots z_n^{k_n} F(z_1,\dots,z_n)}{\prod_{1\leq a < b \leq n} \zeta_{i_bi_a} \left(\frac {z_b}{z_a} \right)} \right) \right] = 
\end{equation}
$$
= \sum^{\text{fair partition}}_{\{1,\dots,n\}=A_1 \sqcup \dots \sqcup A_t} \left[ \underset{(z_{a_s^{(1)}}, \dots, z_{a_s^{(n_s)}})=(w_s q_{\iota(A_s)}^{n_s-1}, \dots, w_sq_{\iota(A_s)}^{1-n_s})}{\text{Res}} \frac {z_1^{k_1}\dots z_n^{k_n} F(z_1,\dots,z_n)}{\prod_{1\leq a < b \leq n} \zeta_{i_bi_a} \left(\frac {z_b}{z_a} \right)} \right]_{\forall s \in \{1,\dots,t\}}
$$
with the sum going over those fair partitions for which the lengths of the various $A_s$ match the parts $n_{i1},n_{i2},\dots$ in \eqref{eqn:contour}. In the left-hand side of \eqref{eqn:ofofof}, we write $\Sym$ for the symmetrization with respect to all pairs of variables $z_a$ and $z_b$ such that $i_a = i_b$. Thus, in the left-hand side of \eqref{eqn:ofofof}, we specialize specific subsets of variables of $\text{Sym}(*)$ to geometric progressions $w  q^\bullet$, where $*$ is shorthand for the rational function that appears on the second line of the equation. This is of course the same as specializing arbitrary subsets of variables of $*$ to geometric progressions $w  q^\bullet$. In the previous sentence, we can restrict attention to those subsets of variables where the indices increase (this is because the only poles of $*$ that involve variables $z_a$ and $z_b$ with $a<b$ and $i_a = i_b$ are $z_a - z_b q_{i_a}^2$) and this precisely yields the right-hand side of \eqref{eqn:ofofof}. Combining \eqref{eqn:end induction 1} and \eqref{eqn:ofofof} yields \eqref{eqn:contour} for $E$ as in \eqref{eqn:the e}.

\medskip 

\noindent It remains to prove Claim \ref{claim:induction}. To this end, consider the residue theorem
$$
\int_{|z| \gg |w|} G(z,w) = \int_{|z| = |w|} G(z,w) + \sum_{|\gamma| > 1} \int \left[ \underset{z = w\gamma}{\text{Res}} G(z,w) \right]
$$
for any rational function $G$, all of whose poles are of the form $z - w\gamma$. Consider formula \eqref{eqn:xm}, and let us zoom in on the summand corresponding to a given fair partition $\{m,\dots,n\} = A_1 \sqcup \dots \sqcup A_t$. As we move the (larger) contour of the variable $z_{m-1}$ toward the (smaller) contours of the variables $w_1,\dots,w_t$, one of two things can happen. The first thing is that the larger contour reaches the smaller ones, which leads to the fair partition
$$
\{m-1,\dots,n\} = A_1\sqcup \dots \sqcup A_t \sqcup \{m-1\}
$$
in formula \eqref{eqn:xm} for $m$ replaced by $m-1$. The second thing is that the variable $z_{m-1}$ gets ``caught" in a pole of the form $z_{m-1} = w_s \gamma$ for some $s \in \{1,\dots,t\}$ and some $|\gamma|>1$. However, the apparent poles of the rational function on the second line of \eqref{eqn:xm} that involve both $z_{m-1}$ and some $w_s$ for $s \in \{1,\dots,t\}$ are of the form
\begin{equation}
\label{eqn:two cases}
\begin{cases} \frac 1{z_{m-1}-w_s q_{\iota(A_s)}^{n_s+1}} &\text{if } i_{m-1} = \iota(A_s) \\ \prod_{\bullet \in \{n_s-1,n_s-3,\dots,3-n_s,1-n_s\}}^{\bullet+c \geq 0} \frac 1{z_{m-1} - w_s q_{\iota(A_s)}^{\bullet + c}} &\text{if }i_{m-1} \neq \iota(A_s)\end{cases}
\end{equation}
where we recall that $c := \frac {d_{i_{m-1}\iota(A_s)}}{d_{\iota(A_s)}}$ is a non-positive integer entry of the Cartan matrix. Let us start with the second option in \eqref{eqn:two cases}. The apparent simple pole at
$$
z_{m-1} = w_s q_{\iota(A_s)}^{\bullet +c}
$$
is precisely canceled by the fact that $F$ vanishes at the specialization
$$
z_{\iota(A_s)1} = w_s q_{\iota(A_s)}^{\bullet+2c}, \quad z_{\iota(A_s)2} = w_s q_{\iota(A_s)}^{\bullet+2c+2}, \quad \dots \quad , \quad z_{\iota(A_s), 1-c} = w_s q_{\iota(A_s)}^{\bullet}
$$
$$
z_{i_{m-1}1} = w_s q_{\iota(A_s)}^{\bullet + c}
$$
due to the condition \eqref{eqn:wheel} (all the powers of $q_{\iota(A_s)}$ on the first line of the equation above lie in the arithmetic progression $\{n_s-1,n_s-3\dots,3-n_s,1-n_s\}$). As for the first option in \eqref{eqn:two cases}, it leads to the fair partition:
$$
\{m-1,\dots,n\} = A_1 \sqcup \dots \sqcup A_{s-1} \sqcup \Big( A_s \sqcup \{m-1\} \Big) \sqcup A_{s+1} \sqcup \dots \sqcup A_t
$$
in formula \eqref{eqn:xm} for $m$ replaced by $m-1$. However, there is a catch: in this new fair partition, the variables that correspond to the $s$-th part are specialized to
$$
w_s q_{\iota(A_s)}^{1-n_s}, \quad \dots \quad  w_s q_{\iota(A_s)}^{n_s-1}, \quad w_s q_{\iota(A_s)}^{n_s+1}
$$
In order to match this with formula \eqref{eqn:xm} for $m$ replaced by $m-1$, we need to move the contour of the variable $w_s$
\begin{equation}
\label{eqn:move}
\text{from} \quad |w_s| = |w_r|, \ \forall r \neq s \quad \text{to} \quad |w_s | = |w_rq^{-1}_{\iota(A_s)}|, \ \forall r \neq s
\end{equation}
It remains to show that no new poles involving $w_s$ and $w_r$ (for an arbitrary $r \neq s$) are produced in the rational function
\begin{equation}
\label{eqn:restricted rational function}
\mathop{\underset{(z_{a_r^{(1)}}, \dots, z_{a_r^{(n_r)}}) = (w_r q_{\iota(A_r)}^{n_r-1}, \dots, w_r q_{\iota(A_r)}^{1-n_r})}{\text{Res}}}_{(z_{m-1}, z_{a_s^{(1)}}, \dots, z_{a_s^{(n_s)}}) = (w_s q_{\iota(A_s)}^{n_s+1}, w_s q_{\iota(A_s)}^{n_s-1}, \dots, w_s q_{\iota(A_s)}^{1-n_s})} \left[ \frac {z_1^{k_1}\dots z_n^{k_n} F(z_1,\dots,z_n)}{\prod_{1\leq a < b \leq n} \zeta_{i_bi_a} \left(\frac {z_b}{z_a} \right)} \right]
\end{equation}
as we move the contours according to \eqref{eqn:move}. At every point during the movement of the contours, let us depict the variables on two horizontal lines, where the midpoint of the top line is $|w_r|$ and any variable whose absolute value is $|w_r|q^k$ is situated $k$ horizontal units to the right of the midpoint. We draw a diagonal going $|d_{\iota(A_r)\iota(A_s)}|$ units to the right from any variable $z_a$ on one line to a variable $z_b$ on another line if $a < b$; every such diagonal is responsible for a simple pole in \eqref{eqn:restricted rational function}.

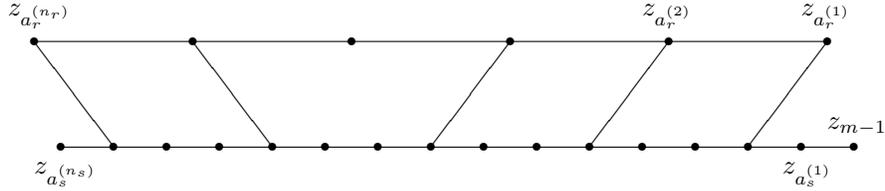
\begin{picture}(100,100)(-50,-30)
\label{pic:zig}
	
\put(-30,0){\line(1,0){300}}
\put(-40,40){\line(1,0){300}}
	
\put(-30,0){\circle*{3}}
\put(-10,0){\circle*{3}}
\put(10,0){\circle*{3}}
\put(30,0){\circle*{3}}
\put(50,0){\circle*{3}}
\put(70,0){\circle*{3}}
\put(90,0){\circle*{3}}
\put(110,0){\circle*{3}}
\put(130,0){\circle*{3}}
\put(150,0){\circle*{3}}
\put(170,0){\circle*{3}}
\put(190,0){\circle*{3}}
\put(210,0){\circle*{3}}
\put(230,0){\circle*{3}}
\put(250,0){\circle*{3}}
\put(270,0){\circle*{3}}
	
\put(-40,40){\circle*{3}}	
\put(20,40){\circle*{3}}
\put(80,40){\circle*{3}}
\put(140,40){\circle*{3}}
\put(200,40){\circle*{3}}
\put(260,40){\circle*{3}}
	
\put(250,50){$z_{a_r^{(1)}}$}
\put(190,50){$z_{a_r^{(2)}}$}
\put(-50,50){$z_{a_r^{(n_r)}}$}

\put(260,7){$z_{m-1}$}	
\put(243,-10){$z_{a_s^{(1)}}$}
\put(-40,-10){$z_{a_s^{(n_s)}}$}
	
\put(-40,40){\line(3,-4){30}}
\put(20,40){\line(3,-4){30}}
	
\put(140,40){\line(-3,-4){30}}
\put(200,40){\line(-3,-4){30}}
\put(260,40){\line(-3,-4){30}}
	
\end{picture}

\noindent However, because we have $a_r^{(n_r)} > \dots > a_r^{(1)}$ and $a_s^{(n_s)} > \dots > a_s^{(1)} > m-1$ by construction, no two diagonals can intersect (not even at the endpoints) lest they produce an impossible closed cycle where each variable has larger index than the next. A basic property of Cartan matrices is that either $d_{\iota(A_r) \iota(A_s)} = 0$ or
\begin{equation}
\label{eqn:crucial}
\text{min}(d_{\iota(A_r)}, d_{\iota(A_s)}) \quad \text{divides} \quad  \text{max}(d_{\iota(A_r)}, d_{\iota(A_s)}) =   |d_{\iota(A_r) \iota(A_s)}| 
\end{equation}
In the case when $d_{\iota(A_r) \iota(A_s)} = 0$, \eqref{eqn:wheel} implies that $F$ is divisible by linear factors corresponding to all the diagonal lines (which will be perfectly vertical) in the figure above, and so \eqref{eqn:restricted rational function} has no poles involving both $w_r$ and $w_s$. In the case of \eqref{eqn:crucial}, the variables involved with the diagonal lines in the figure above all have horizontal coordinate congruent to some fixed $k$ modulo $|d_{\iota(A_r) \iota(A_s)}|$. Thus, we might as well ignore all the variables whose horizontal coordinate is $\not \equiv k$, and then the picture above precisely matches the one in the simply-laced case
$$
d_{\iota(A_r)} =  d_{\iota(A_s)} = 1 = |d_{\iota(A_r) \iota(A_s)}|
$$
In this case, it is well-known that the wheel conditions \eqref{eqn:wheel} cause the numerator of \eqref{eqn:restricted rational function} to vanish to order at least as great as the number of diagonals in the picture above. For instance, one can show this by placing each diagonal $D$ in a triangle $T_D$, such that $T_D$ and $T_{D'}$ have either 0 or 1 vertices in common if $D \neq D'$. Then each triangle $T_D$ corresponds to a vanishing condition \eqref{eqn:wheel} which produces one zero in the numerator of \eqref{eqn:restricted rational function}, thus canceling out the pole caused by the diagonal $D$ \footnote{A crucial detail which ensures the preceding argument works is that the variable $z_{m-1}$ has the smallest index among all variables involved, and so it can only afford a diagonal going out of it to the right; this is in tune with the fact that the variable $z_{m-1}$ is causing the bottom line in the picture to be shifted between $0$ and $|d_{\iota(A_r) \iota(A_s)}|$ units to the right of the top line.}. \end{proof}

\medskip

\subsection{A generalization}
\label{sub:generalization}

The contents of the previous subsection may be generalized as follows. Following a suggestion of David Hernandez, we call a Kac-Moody Lie algebra $\fg$ \emph{strongly symmetrizable} if 
\begin{equation}
\label{eqn:strongly symmetrizable}
d_{ij} \in \{ 0, - \max(d_i,d_j) \}
\end{equation}
for all $i\neq j$. This definition includes all finite and affine type Lie algebras, except for affine $A_1$ (thus, the discussion in the present subsection does not apply to affine $A_1$, in which case one must instead follow the treatment of \cite{N Symmetric}). For a strongly symmetrizable Kac-Moody Lie algebra $\fg$, Lemma \ref{lem:contour} holds as stated, where
\begin{equation}
\label{eqn:def}
\CS^\pm := \left\{E \in \CV^\pm \text{ s.t. } E\Big|_{z_{i1} = z_{j1} q^{d_{ij}}, z_{i2} = z_{j1} q^{d_{ij}+d_{ii}}, \dots,  z_{i, 1-c_{ij}} = z_{j1} q^{-d_{ij}}} = 0, \forall i \neq j\right\}
\end{equation}
Thus, we are in the situation of \cite[Proposition 3.3]{N Wheel}: Lemma \ref{lem:contour} provides a non-degenerate pairing between the subalgebra $\CS^\pm$ of \eqref{eqn:def} and the a priori smaller subalgebra $\text{Im }\widetilde{\Upsilon}^\mp \subseteq \CV^\mp$ linearly spanned by the elements \eqref{eqn:spherical}. Then the argument of \cite[Theorem 2.13]{N Wheel} carries through and implies that these two algebras coincide
$$
\Big( \CS^\pm \text{ of \eqref{eqn:def}} \Big) = \Big(\text{Im }\widetilde{\Upsilon}^\pm \text{ of Definition \ref{def:shuffle}} \Big)
$$
This gives a complete description of the shuffle algebra associated to a strongly symmetrizable Kac-Moody Lie algebra $\fg$. As shown in \cite{NT}, the wheel conditions \eqref{eqn:def} are dual to the so-called Drinfeld-Serre relations
$$
S_{ij}^+ = \sum_{k=0}^{1-c_{ij}} (-1)^k {1-c_{ij} \choose k}_{q_i} \text{Sym}_{z_1,\dots,z_{1-c_{ij}}}  e_i(z_1) \dots e_i(z_k) e_j(w) e_i(z_{k+1}) \dots e_i(z_{1-c_{ij}}) 
$$
$$
S_{ij}^- = \sum_{k=0}^{1-c_{ij}} (-1)^k {1-c_{ij} \choose k}_{q_i} \text{Sym}_{z_1,\dots,z_{1-c_{ij}}}  f_i(z_1) \dots f_i(z_k) f_j(w) f_i(z_{k+1}) \dots f_i(z_{1-c_{ij}}) 
$$
Following the general principle laid out in \cite{N Arbitrary}, this implies that
$$
\text{Ker }\widetilde{\Upsilon}^\pm = \left( S_{ij}^\pm \right)_{i\neq j}
$$
We conclude that the full set of relations in the quantum loop algebra \eqref{eqn:quantum loop algebra} associated to any strongly symmetrizable Kac-Moody Lie algebra consists of
\begin{equation}
\label{eqn:drinfeld-serre}
S_{ij}^\pm = 0, \quad \forall i \neq j
\end{equation}
together with \eqref{eqn:rel quantum 1}, \eqref{eqn:rel quantum 2}, \eqref{eqn:rel quantum 4}, \eqref{eqn:rel quantum 5}, \eqref{eqn:rel quantum 6}, \eqref{eqn:rel quantum 7}, \eqref{eqn:rel quantum 8}, \eqref{eqn:rel quantum 9}. In the particular case of quantum toroidal algebras of type other than $A_1$, this proves that such algebras possess triangular decompositions in terms of their Drinfeld positive and negative halves, and that the usual Hopf pairing between the two halves is non-degenerate (i.e. the aforementioned set of relations is maximal so that the resulting algebra keeps its usual Hopf algebra structure and Hopf pairing); see \cite{N Arbitrary} for details.

\bigskip 

\section{Modules and tensor products}
\label{sec:modules}

\medskip

\noindent We will now introduce simple modules for the subalgebras $\CA^{\geq \bp} \subset \CS = \U$ defined in the previous Section, and we will use the coproducts $\Delta_{\bp}$ to define tensor products. Our main reference will be \cite{N Cat}; although \loccit pertains to a particular choice of $(I,\BK,\zeta_{ij}(x))$, most proofs therein are completely general, and we will provide alternative proofs for those results where the generalization is not straightforward.

\medskip

\subsection{Borel category $\CO$} 
\label{sub:category o}

The following definitions are natural generalizations of classic constructions for quantum affine algebras \eqref{eqn:quantum affine intro}. After finite-dimensional type 1 $U_q(\widehat{\fg})$-modules were classified in \cite{CP}, it was recognized in \cite{HJ} that one can obtain more by restricting to the Borel subalgebra. In light of \eqref{eqn:borel positive}, we generalize this by choosing any $\bp \in \rr$ and considering modules
\begin{equation}
\label{eqn:module}
\CA^{\geq \bp} \curvearrowright V
\end{equation}
for which the action of the Cartan subalgebra $\{\kappa^+_i = \ph^+_{i,0}\}_{i \in I}$ is diagonalizable:
\begin{equation}
\label{eqn:weight decomposition}
V = \bigoplus_{\bla = (\lambda_i)_{i \in I} \in (\BK^*)^I} V_{\bla}
\end{equation} 
where
$$
V_{\bla} = \Big\{ v \in V \Big| \kappa^+_i \cdot v = \lambda_i v, \forall i \in I \Big\}
$$
Above, $\bla$ will be referred to as \emph{weights}. Because of the commutation relation
\begin{equation}
\label{eqn:commute}
\kappa^+_i X = X \kappa^+_i \gamma_{\bs^i, \hdeg X}
\end{equation}
for all $i \in I$ and $X \in \CA^{\geq \bp}$ (see \eqref{eqn:cartan commutation plus}), it is easy to see that algebra elements interact with the weight space decomposition according to the rule
\begin{equation}
\label{eqn:r action}
X : V_{\bla} \rightarrow V_{\bla  \bga^{\hdeg X}}
\end{equation}
where $\bga^{\bn}$ denotes the weight with $i$-th component $\gamma_{\bs^i, \bn}$ for all $\bn \in \zz$. In the right-hand side of \eqref{eqn:r action} we multiply weights component-wise, so $\bla \bmu = (\lambda_i\mu_i)_{i \in I}$ for all weights $\bla = (\lambda_i)_{i \in I}$ and $\bmu = (\mu_i)_{i \in I}$.

\medskip 

\begin{definition}
\label{def:category o}

A module $\CA^{\geq \bp} \curvearrowright V$ is said to be in \emph{category }$\CO$ if it has a weight decomposition \eqref{eqn:weight decomposition} with every $V_{\bla}$ finite-dimensional and non-zero only for
$$
\bla \in \{\bla^1 \bga^{-\bn}, \dots, \bla^k \bga^{-\bn}\}_{\bn \in \nn}
$$
for finitely many weights $\bla^1,\dots,\bla^k$.

\end{definition}

\medskip

\noindent In order for the above definition to be meaningful, we assume throughout that
\begin{equation}
\label{eqn:generic}
\bga^{\bn} \neq \bone
\end{equation}
for all $\bn \in \zz \backslash \b0$, which allows us to define the partial order
\begin{equation}
\label{eqn:order}
\bla \bga^{\bn} > \bla, \quad \forall \text{ weight }\bla, \ \forall \ \bn \in \nn \backslash \b0
\end{equation}
If condition \eqref{eqn:generic} is not satisfied (e.g. in the important case of quantum toroidal $\fgl_1$ studied in \cite{FJMM}) then there is a fix following \cite{H2}: one adds additional finite Cartan elements $\kappa_j^+$ and imposes relations \eqref{eqn:commute} for various constants $\gamma_{\bs^j, \bn}$ that are multiplicative in $\bn$. If the aforementioned constants are chosen generic enough, then we can ensure \eqref{eqn:generic} for all $\bn \in \zz \backslash \b0$.

\medskip 

\begin{proposition}
	\label{prop:tensor product}

For any modules $\CA^{\geq \bp} \curvearrowright V,W$ in category $\CO$, the tensor product
$$
\CA^{\geq \bp} \curvearrowright V \otimes W
$$
defined with respect to the coproduct $\Delta_{\bp}$, is a module in category $\CO$.

\end{proposition}

\medskip 

\begin{proof} The fact that $\kappa^+_i$ are group-like for the coproduct $\Delta_{\bp}$ ensures that $V \otimes W$ satisfies all the axioms of category $\CO$. However, we need to check that infinite sums
$$
\sum_k A_k \otimes B_k \in (\CA^{\geq \bp} \ho \CA^{\geq \bp})_{\bn,d}
$$
act correctly on $V \otimes W$. By the very definition of $\ho$, for any $N \in \BN$ we have $|\hdeg A_k| \leq -N$ and $|\hdeg B_k| \geq N$ for all but finitely many $k$. This means that on any given vector $v \otimes w \in V \otimes W$, all but finitely many $B_k$'s will have the property that $B_k(w) = 0$ due to \eqref{eqn:r action} and the very definition of category $\CO$. \end{proof}

\medskip

\subsection{Loop weights}
\label{sub:loop weights}

Beside the finite Cartan subalgebra generated by $\kappa^+_i = \ph^+_{i,0}$, we have the (positive) loop Cartan subalgebra 
$$
\CB^{\geq}_{\binfty} = \BK[\ph^+_{i,d}]_{i \in I, d \geq 0}
$$
in $\CA^{\geq \bp}$ for any $\bp \in \rr$. The $\ph^+_{i,d}$'s give an infinite family of commuting operators in any module $V$ in category $\CO$, so we may consider their joint generalized eigenspaces 
\begin{equation}
\label{eqn:loop weight decomposition}
V = \bigoplus_{\bpsi = (\psi_i(z))_{i \in I} \in (\BK[[z^{-1}]]^\times)^I} V_{\bpsi}
\end{equation} 
\footnote{The decomposition above exists on general grounds if $\BK$ is algebraically closed, but otherwise we restrict attention only to those $V$'s for which such a decomposition exists.} where
$$
V_{\bpsi} = \Big\{ v \in V \Big| \left( \ph^+_{i,d} - \psi_{i,d}\text{Id}_V \right)^N v = 0, \forall i \in I, d \geq 0 \text{ and }N \text{ large enough} \Big\}
$$
Above, 
\begin{equation}
\label{eqn:loop weight}
\bpsi = \left(\psi_i(z) = \sum_{d=0}^{\infty} \frac {\psi_{i,d}}{z^d} \right)_{i \in I}
\end{equation}
will be referred to as \emph{loop weights}, by analogy with the classic situation of quantum affine algebras. The vector spaces $V_\bpsi$ are finite-dimensional for any $V$ in category $\CO$, so we may define the \emph{$q$-character} following \cite{FR}
\begin{equation}
\label{eqn:q-character}
\chi_q(V) = \sum_{\text{loop weights }\bpsi} \dim_{\BK}(V_{\bpsi}) [\bpsi]
\end{equation}
where $[\bpsi]$ are formal symbols. Recall that $[\bpsi][\bpsi'] = [\bpsi\bpsi']$ with respect to the component-wise multiplication of $I$-tuples of power series.

\medskip

\begin{proposition}
\label{prop:multiplicative}

For any modules $\CA^{\geq \bp} \curvearrowright V,W$ in category $\CO$, we have
\begin{equation}
\label{eqn:multiplicative}
\chi_q(V \otimes W) = \chi_q(V) \chi_q(W)
\end{equation}
\end{proposition}

\medskip

\begin{proof} The Proposition would be obvious if the Cartan series $\ph^+_i(z)$ were group-like with respect to the coproduct $\Delta_{\bp}$, but this is not the case. Instead, \eqref{eqn:two 1} yields
$$
\Delta_{\bp}(\ph^+_i(z)) = \ph^+_i(z) Y \otimes X
$$
where $Y \otimes X \in \CS_{(-\binfty,\bp)}^- \otimes \CS_{[\bp,\binfty]}^+$ is determined by
$$
X \langle E', Y \rangle = \Big[\ph^+_i(z) E'\Big]_{\geq \bp}, \quad \forall E' \in \CS_{(-\binfty,\bp)}^+
$$
When $\hdeg E' > \b0$, the equation above forces $\hdeg X > \b0$ and $\hdeg Y < \b0$, and when $\hdeg E' = \b0$, the equation above forces $Y \otimes X = 1 \otimes \ph_i^+(z)$. We thus have
\begin{equation}
\label{eqn:id}
\Delta_{\bp}(\ph^+_i(z)) \in \ph^+_i(z) \otimes \ph^+_i(z) + \Big(\text{hdeg} < \b0\Big) \otimes \Big(\text{hdeg} > \b0\Big) 
\end{equation}
(the formula above generalizes a well-known formula of \cite{Da} for the Drinfeld-Jimbo coproduct of quantum affine algebras). Therefore, the action of $\Delta_{\bp}(\ph^+_i(z))$ on $V \otimes W$ is block upper triangular with respect to the subspaces $V_{\bpsi} \otimes W_{\bpsi'}$ of $V \otimes W$: the order of the blocks is determined by \eqref{eqn:order}, and the action on the diagonal blocks is given by $\ph^+_i(z) \otimes \ph^+_i(z)$. Since $\chi_q(V \otimes W)$ encodes the generalized eigenspaces of $\Delta_{\bp}(\ph^+_i(z))$ while $\chi_q(V) \chi_q(W)$ encodes the generalized eigenspaces of $\ph^+_i(z) \otimes \ph^+_i(z)$, the aforementioned triangularity implies \eqref{eqn:multiplicative}. \end{proof}

\medskip 

\subsection{Polynomiality of the theta series}
\label{sub:theta}

Inspired by the work of Huafeng Zhang (\cite[Theorem 9.5]{Zhang}, which originally established Theorem \ref{thm:t} below for quantum affine algebras), we expand on the proof of Proposition \ref{prop:multiplicative} in the following sense. Fix $i \in I$ and assume that there exists a Cartan series 
\begin{equation}
\label{eqn:def t}
T_i(x) \in \U[[x^{-1}]] 
\end{equation}
(initially defined for quantum affine algebras in \cite{FH}) which is group-like for \eqref{eqn:coproduct u}
\begin{equation}
\label{eqn:coproduct t}
\Delta(T_i(x)) = T_i(x) \otimes T_i(x)
\end{equation}
and commutes with the positive and negative halves of the quantum loop algebra according to the following analogues of \eqref{eqn:ssg} and \eqref{eqn:ssl}
\begin{align}
&T_i(x) E(z_{j1},\dots,z_{jn_j})_{j \in I} = E(z_{j1},\dots,z_{jn_j})_{j \in I} T_i(x)  \prod_{a=1}^{n_i} \left(1- \frac {z_{ia}}x \right) \label{eqn:ssg t} \\
&F(z_{j1},\dots,z_{jn_j})_{j \in I} T_i(x)  = T_i(x) F(z_{j1},\dots,z_{jn_j})_{j \in I}  \prod_{a=1}^{n_i} \left(1- \frac {z_{ia}}x \right) \label{eqn:ssl t}
\end{align}
When the zeta functions \eqref{eqn:zeta intro} have sufficiently generic coefficients (which happens for instance in the case of quantum affine algebras), it is well-known that $T_i(x)$ can be written as an appropriate product of the power series $\{\ph^+_j(xa)\}_{j \in I, a \in \BK^*}$. Even if the aforementioned genericity fails (which happens for instance in the case of quantum toroidal algebras), one can still formally add the series $T_i(x)$ to the algebra $\U$, all the while imposing \eqref{eqn:coproduct t}, \eqref{eqn:ssg t}, \eqref{eqn:ssl t}.

\medskip 

\begin{theorem}
\label{thm:t}

For any $\bp \in \rr$, we have
\begin{equation}
\label{eqn:t}
\Delta_{\bp}(T_i(x)) = (T_i(x) \otimes 1) \Theta_{i,\bp}(x) (1 \otimes T_i(x))
\end{equation}
for some
\begin{equation}
\label{eqn:tt}
\Theta_{i,\bp}(x) \in 1 + \sum_{\bn \in \nn\backslash \b0} \sum_{d=1}^{n_i} \frac {\U_{-\bn} \otimes \U_{\bn}}{x^d}
\end{equation}
called a \emph{theta series}. Thus, every hdeg graded piece of $\Theta_{i,\bp}$ is polynomial in $x^{-1}$.

\end{theorem} 

\medskip

\begin{proof} As in the proof of Proposition \ref{prop:multiplicative}, we have 
\begin{equation}
\label{eqn:zero}
\Delta_{\bp}(T_i(x)) = T_i(x) Y \otimes X
\end{equation}
where $Y \otimes X \in \CS_{(-\binfty,\bp)}^- \otimes \CS_{[\bp,\binfty]}^+$ satisfies for any $E \in \CS_{(-\binfty,\bp)}^+ \cap \CS_{\bn}$ the equation
\begin{equation}
\label{eqn:eqn}
X \langle E', Y \rangle = \Big[T_i(x) E'\Big]_{\geq \bp} = \left[E'\prod_{a=1}^{n_i} \left(1- \frac {z_{ia}}x \right) \right]_{\geq \bp} T_i(x)
\end{equation}
The formula above immediately implies \eqref{eqn:tt}, once one observes that no $x^{-d}$ with $d > n_i$ can contribute to the bracket in the right-hand side. Moreover, if $\bn \neq \b0$ then $x^0$ can also not appear, because $[E']_{\geq \bp} = 0$ for any element $E'$ of slope $<\bp$. \end{proof}

\medskip 

\subsection{Simple modules}
\label{sub:simple modules}

For any loop weight $\bpsi$, we will define a horizontally graded simple module 
\begin{equation}
\label{eqn:simple}
\CA^{\geq \bp} \curvearrowright L^{\bp}(\bpsi)
\end{equation}
generated by a single vector $\vac$ modulo the relations
\begin{align}
&\ph^+_i(z) \cdot \vac = \psi_i(z) \vac, \quad \forall i \in I \label{eqn:simple relation 1} \\
&E \cdot \vac = 0, \ \ \quad \qquad \qquad \forall E \in \CS^+_{\geq \bp} \quad \text{of hdeg}>\b0 \label{eqn:simple relation 2}
\end{align}
The construction of this simple module is quite standard: define the Verma module
\begin{equation}
\label{eqn:verma action}
W^{\bp}(\bpsi) = \CS^-_{< \bp} \vac 
\end{equation}
with $F \in \CS^-_{<\bp}$ acting by left multiplication and $E \in \CS^+_{\geq \bp}$ acting by
\begin{multline} 
E \cdot F \vac \stackrel{\eqref{eqn:dd1}}= \langle E_1, F_1 \rangle  \langle E_3, S(F_3) \rangle F_2 E_2 \cdot \vac \\ \stackrel{\eqref{eqn:epsilon}}= \langle E_2, S(F_2) \rangle F_1 E_1 \cdot \vac \stackrel{\text{\eqref{eqn:simple relation 1}-\eqref{eqn:simple relation 2}}}=  \langle E\psi, S(F_2) \rangle F_1\vac \label{eqn:action verma}
\end{multline}
where for $E = E(z_{i1},\dots,z_{in_i})_{i \in I}$, we write 
$$
E \psi = E(z_{i1},\dots,z_{in_i})_{i \in I}\prod_{i \in I} \prod_{a=1}^{n_i} \psi_i(z_{ia})
$$
While the expression $E\psi$ is a power series in the variables $z_{ia}^{-1}$, only finitely many terms in the expansion pair non-trivially with any given $S(F_2)$, since the vertical degree of the latter is bounded above for any $F$. The following result is quite standard, and it can be construed as the definition of  $L^{\bp}(\bpsi)$.

\medskip 

\begin{proposition}
\label{prop:unique}
	
There is a surjective $\CA^{\geq \bp}$-module homomorphism 
\begin{equation}
\label{eqn:unique}
W^{\bp}(\bpsi) \stackrel{\pi}\twoheadrightarrow L^{\bp}(\bpsi)
\end{equation}
that sends $\vac$ to $\vac$, whose kernel is the maximal $-\nn \backslash \b0$ horizontally graded $\CA^{\geq \bp}$-submodule of $W^{\bp}(\bpsi)$. Thus, the simple module $L^{\bp}(\bpsi)$ is unique up to isomorphism. 

\end{proposition}

\medskip 

\begin{proof} Sending $\vac \mapsto \vac$ yields an $\CA^{\geq \bp}$-module homomorphism 
$$
\pi : W^{\bp}(\bpsi) \rightarrow L^{\bp}(\bpsi)
$$
which is surjective due to the simplicity of $L^{\bp}(\bpsi)$. Moreover, $\text{Ker }\pi$ is a horizontally graded $\CA^{\geq \bp}$-submodule of $W^{\bp}(\bpsi)$, which is maximal if and only if $L^{\bp}(\bpsi)$ is simple. \end{proof}

\medskip 

\subsection{Rational loop weights}
\label{sub:rational}
 
The following description of the maximal  $-\nn \backslash \b0$ horizontally graded $\CA^{\geq \bp}$-submodule of $W^{\bp}(\bpsi)$ is proved as in \cite[Proposition 4.5]{N Cat}:
$$
\text{Ker }\pi = J^{\bp}(\bpsi) \vac 
$$
where 
\begin{equation}
\label{eqn:j psi}
J^{\bp}(\bpsi) = \left\{F \in  \CS^-_{(-\binfty,\bp)} \Big| \ \langle E \psi, S(F) \rangle = 0 , \ \forall E \in \CS^+_{\geq \bp} \right \}
\end{equation}
Thus, we conclude the following description of the underlying vector space of simple modules
\begin{equation}
\label{eqn:simple modules}
L^{\bp}(\bpsi) = \CS^-_{<\bp} \Big/ J^{\bp}(\bpsi)
\end{equation}
This will allow us to determine which simple modules $L^{\bp}(\bpsi)$ lie in category $\CO$, by analogy with \cite[Theorem 3.11]{HJ}.

\medskip

\begin{proposition}
\label{prop:rational}
	
$L^{\bp}(\bpsi)$ lies in category $\CO$ if and only if $\bpsi$ is \emph{rational}, by which we mean that the constituent power series $\psi_i(z)$ are expansions of rational functions.

\end{proposition}

\medskip

\begin{proof} The ``only if" statement is an easy exercise (see \cite[Theorem 4.11]{N Cat} for what is essentially the general proof in a particular setup), and so we skip it. Meanwhile, the ``if" statement is equivalent to the following fact: finitely many linear conditions on the coefficients of $F \in \CS_{<\bp}^-$ of any fixed horizontal degree $-\bn$ would imply
\begin{equation}
\label{eqn:r1}
\langle E \psi, S(F) \rangle = 0 , \quad \forall E \in \CS^+_{\geq \bp}
\end{equation}
By \cite[Proposition 3.26]{N Cat}, equation \eqref{eqn:r1} follows from
\begin{equation}
\label{eqn:r2}
\langle e_{i_1,d_1} \dots e_{i_n,d_n} \psi, S(F) \rangle = 0
\end{equation}	
for all $i_1,\dots,i_n \in I$ such that $\bs^{i_1} + \dots + \bs^{i_n} = \bn$ and all $d_1,\dots,d_n \geq -M$ for large enough $M$ (depending on $\bn$). By \cite[Lemma 3.10]{N Cat}, condition \eqref{eqn:r2} is equivalent to
\begin{equation}
\label{eqn:r3}
\int_{1 \ll |z_1| \ll \dots \ll |z_n|} \frac {z_1^{d_1}\dots z_n^{d_n} F(z_1,\dots,z_n)}{\prod_{1\leq a < b \leq n} \zeta_{i_bi_a} \left(\frac {z_b}{z_a} \right)} \prod_{a=1}^n \psi_{i_a}(z_a) = 0
\end{equation}	
for all $d_1,\dots,d_n \geq -M$ (the meaning of $1\ll$ is that we expand the rational functions above near $\infty$ and not near 0; the intuition is that $\psi_i(z)$ might produce poles other than $0$ and $\infty$, which we want to avoid). In turn, \eqref{eqn:r3} is implied by
\begin{align*}
&F(z_1,\dots,z_n) \text{ is divisible by } z_1^M Q_1(z_1), \\
\text{or }&F(z_1,\dots,z_n) \text{ is divisible by } z_{2}^M\zeta_{i_2i_1} \left(\frac {z_2}{z_1} \right) Q_2(z_{2}),  \\
& \dots \\
\text{or }&F(z_1,\dots,z_n) \text{ is divisible by } z_n^M\zeta_{i_ni_1} \left(\frac {z_n}{z_1} \right)\dots \zeta_{i_ni_{n-1}} \left(\frac {z_n}{z_{n-1}} \right) Q_n(z_n)
\end{align*}
where $Q_a(z_a)$ denotes the denominator of $\psi_{i_a}(z_a)$, and divisibility is defined in the ring of polynomials, not Laurent polynomials. Indeed, if the $a$-th condition above holds, then we can expand \eqref{eqn:r3} near $z_a \sim 0$ and obtain an answer of $0$ due to the lack of poles at $z_a = 0$. However, it is clear that the above divisibility conditions impose finitely many linear conditions on the coefficients of $F$, since the definition of $\CS_{<\bp}^-$ means that the degree of $F$ in any variable is bounded below. \end{proof}

\medskip 

\subsection{Decompositions}
\label{sub:decompositions}

A particular case of \eqref{eqn:j psi} occurs for the loop weight
$$
z^{-\br} = (z^{-r_i})_{i \in I}
$$
for any $\br = (r_i)_{i \in I} \in \zz$, namely
\begin{equation}
\label{eqn:j r}
J^{\bp}(z^{-\br}) =  \left\{F \in  \CS^-_{<\bp} \ \Big| \ \langle Ez^{-\br}, S(F) \rangle = 0 , \ \forall E \in \CS^+_{\geq \bp} \right \}
\end{equation}
The corresponding simple module \eqref{eqn:simple modules} has underlying vector space
\begin{equation}
\label{eqn:graded modules}
L^{\bp}(z^{-\br}) = \CS^-_{<\bp} \Big/ J^{\bp}(z^{-\br})
\end{equation}
which inherits a grading from the vertical degree of $\CS^-$ (generalizing the grading constructed in \cite{FH}). Moreover, it is clear that shuffle elements $F$ of large enough vertical degree vanish in $L^{\bp}(z^{-\br})$. We will also encounter the following vector space
\begin{equation}
\label{eqn:j neq}
J^{\neq 0}(\bpsi) =  \left\{F \in \CS^- \ \Big| \ \langle E\psi, S(F) \rangle = 0, \forall E \in \CS_{\geq N\bone} \text{ for }N\text{ large enough} \right\}
\end{equation} 
where in the right-hand side, $N$ is allowed to be large enough in comparison to the horizontal degree of $F$. As in formula \eqref{eqn:r3}, we have $F \in J^{\neq 0}(\bpsi)$ if and only if
\begin{equation}
\label{eqn:r4}
\int_{1 \ll |z_1| \ll \dots \ll |z_n|} \frac {z_1^{d_1}\dots z_n^{d_n} F(z_1,\dots,z_n)}{\prod_{1\leq a < b \leq n} \zeta_{i_bi_a} \left(\frac {z_b}{z_a} \right)} \prod_{a=1}^n \psi_{i_a}(z_a)  = 0
\end{equation}
for all $i_1,\dots,i_n \in I$ and all $d_1,\dots,d_n$ large enough. 

\medskip 

\begin{proposition}
\label{prop:two contours}

We have $F \in J^{\neq 0}(\bpsi)$ if and only if
\begin{equation}
\label{eqn:hart}
\int_{z_n} \dots \int_{z_1}  \frac {z_1^{d_1}\dots z_n^{d_n} F(z_1,\dots,z_n)}{\prod_{1\leq a < b \leq n} \zeta_{i_bi_a} \left(\frac {z_b}{z_a} \right)} \prod_{a=1}^n \psi_{i_a}(z_a) = 0
\end{equation}
for all $i_1,\dots,i_n \in I$ and all $d_1,\dots,d_n \in \BZ$. Above, we use the notation
\begin{equation}
	\label{eqn:two contours}
	\int_w G(w) = \int_{1\ll |w|} G(w) - \int_{1 \gg |w|}  G(w)
\end{equation}
to refer to the difference between the constant terms of $G(w)$ calculated near $\infty$ and near $0$ (the notation is inspired by the case of $\BK = \BC$, in which $\int_w$ can be calculated as the difference of contour integrals over two circles around $\infty$ and around $0$).

\end{proposition}

\medskip 

\begin{proof} It is clear that $\int_w G(w) = 0$ for a Laurent polynomial $G$. Thus, the left-hand side of \eqref{eqn:hart} is automatically 0 if $F$ satisfies any one of the divisibility properties in the proof of Proposition \ref{prop:rational}. This allows us to arbitrarily increase the exponents $d_1,\dots,d_n$ in formula \eqref{eqn:hart} without changing the zero-ness of the left-hand side. We conclude that \eqref{eqn:hart} holds for all $d_1,\dots,d_n \in \BZ$ if and only if it holds for all $d_1,\dots,d_n$ large enough. But if $d_1,\dots,d_n$ are large enough, the integrals in \eqref{eqn:hart} have zero contribution from $|z_1|,\dots,|z_n| \ll 1$, so  \eqref{eqn:hart} is equivalent to \eqref{eqn:r4}. \end{proof}

\medskip 

\noindent While the quotient
\begin{equation}
\label{eqn:neq modules}
L^{\neq 0}(\bpsi) = \CS^- \Big/ J^{\neq 0}(\bpsi) 
\end{equation}
is not a $\CA^{\geq \bp}$-module, we will see in the next Subsection that it is a module for a shifted version of the quantum loop algebra $\U$. Recall the coproduct \eqref{eqn:coproduct shuffle minus}, and let us write it as follows
\begin{equation}
\label{eqn:delta prime}
\Delta(F) = \sum_{\b0 \leq \bm \leq \bn} F'(z_{i1},\dots , z_{im_i}) \otimes  F''(z_{i,m_i+1}, \dots, z_{in_i}) \prod_{i \in I} \prod_{a=1}^{m_i} \ph^-_i(z_{ia})
\end{equation}
(in other words, we absorb the denominator of \eqref{eqn:coproduct shuffle minus} in the sum of tensors $F' \otimes F''$, at the cost of allowing infinite sums). For any loop weight $\bpsi = (\psi_i(z))_{i \in I}$, let  $\br = \textbf{ord }\bpsi$ denote the $I$-tuple of orders of the poles of the functions $\psi_i(z)$ at $z=0$. 

\medskip 

\begin{proposition}
\label{prop:decomposition}

For any loop weight $\bpsi$, let $\br = \emph{\textbf{ord }}\bpsi$. The assignment \footnote{While the sum $F' \otimes F''$ is infinite for any given $F$, we note that the map \eqref{eqn:decomposition} is well-defined because all but finitely many of the $F'$ that appear have vertical degree bounded below by any arbitrarily large number, and thus vanish in $L^{\bp}(z^{-\br})$.}
$$
\CS^-_{<\bp} \rightarrow \CS^-_{<\bp}  \otimes \CS^-, \qquad F \mapsto F' \otimes F''
$$
(with $F',F''$ as in \eqref{eqn:delta prime}) induces an isomorphism of vector spaces
\begin{equation}
\label{eqn:decomposition}
L^\bp(\bpsi) \xrightarrow{\sim} L^\bp(z^{-\br}) \otimes L^{\neq 0}(\bpsi)
\end{equation}
which intertwines the action of $\ph^+_i(z)$ on the LHS with $\ph^+_i(z) \otimes \ph^+_i(z)$ on the RHS.

\end{proposition}

\medskip

\begin{proof} We will adapt the proof of \cite[Proposition 4.8]{N Char}. In order to show that the map \eqref{eqn:decomposition} is well-defined and injective, we must prove that 	
\begin{equation}
\label{eqn:that}
F \in J^{\bp}(\bpsi) \quad \Leftrightarrow \quad F' \otimes F'' \in J^{\bp}(z^{-\br}) \otimes \CS^- + \CS^-_{<\bp} \otimes J^{\neq 0}(\bpsi)
\end{equation}
For the implication $\Rightarrow$, we must show that
\begin{multline}
\label{eqn:this}
\left \langle E'(z_{i1},\dots,z_{im_i})_{i \in I} \prod_{i \in I} \prod_{a=1}^{m_i} z_{ia}^{-r_i}, S(F') \right \rangle \\ \left \langle E''(z_{i,m_i+1},\dots,z_{in_i})_{i \in I} \prod_{i \in I} \prod_{a=m_i+1}^{n_i} \psi_i(z_{ia}), S(F'') \right \rangle = 0
\end{multline}
for all $E' \in \CS_{\geq \bp|\bm}$ and $E'' \in \CS_{\geq N \bone|\bn-\bm}$ with $N$ large enough. By \eqref{eqn:bialgebra 2} and the anti-automorphism property of the antipode, the condition above is equivalent to \footnote{Note that we must use the following straightforward consequence of \eqref{eqn:coproduct shuffle plus} and \eqref{eqn:bialgebra 1}:
\begin{equation}
\label{eqn:consequence 1}
\langle E'', S(\ph F'') \rangle = \varepsilon(\ph) \langle E'',S(F'')\rangle
\end{equation}
for any $E'' \in \CS^+$, $F'' \in \CS^-$ and $\ph = \ph^-_{i_1,d_1}\ph^-_{i_2,d_2}\dots$. We also note that $\ph$ is on opposite sides of $F''$ in the RHS of \eqref{eqn:delta prime} and in the LHS of \eqref{eqn:consequence 1}. This is not an issue, since commuting $\ph$ past $F''$ happens at the cost of multiplying $F' \otimes F''$ by a power series in $\{z_{ia}/z_{jb}\}_{i,j \in I, a \leq m_i, b > m_j}$ with non-zero constant term, which does not change whether $F' \otimes F'' \in J^{\bp}(z^{-\br}) \otimes \CS^- + \CS^-_{<\bp} \otimes J^{\neq 0}(\bpsi)$.}
$$
\left \langle E'(z_{i1},\dots,z_{im_i}) * E''(z_{i,m_i+1},\dots,z_{in_i}) \prod_{i \in I} \left( \prod_{a=1}^{m_i} z_{ia}^{-r_i} \prod_{a=m_i+1}^{n_i} \psi_i(z_{ia}) \right), S(F)\right \rangle = 0
$$
However, consider the following fact: since $\psi_i(w) = O(w^{-r_i})$ near 0, then we can find a polynomial $Q_i(w)$ such that $Q_i(w)\psi_i(w) \in w^{-r_i} + w^N\BC[w]$ for arbitrarily large $N$ (which we choose large enough so that any shuffle element of vertical degree $\geq N$ is automatically in $J^\bp(z^{-\br})$). Thus, the equation above is implied by
$$
\left \langle E'(z_{i1},\dots,z_{im_i})\prod_{i \in I} \prod_{a=1}^{m_i} Q_i(z_{ia}) * E''(z_{i,m_i+1},\dots,z_{in_i}) \prod_{i \in I} \prod_{a=1}^{n_i} \psi_i(z_{ia}), S(F)\right \rangle = 0
$$
which in turn holds because $F \in J^{\bp}(\bpsi)$ and $E' \prod_{i,a} Q_i(z_{ia}) * E'' \in \CS^+_{\geq \bp}$.
	
\medskip
	
\noindent For the implication $\Leftarrow$ of \eqref{eqn:that}, recall that any $E \in \CS_{\geq \bp}^+$ can be written as
\begin{equation}
	\label{eqn:e spherical}
	E \stackrel{\eqref{eqn:spherical}}= \text{Sym} \left[ \nu(z_1,\dots,z_n) \prod_{1 \leq a < b \leq n} \zeta_{i_ai_b} \left(\frac {z_a}{z_b} \right) \right] \in \CS_{\geq \bp}^+
\end{equation}
for some Laurent polynomial $\nu$ and some $i_1,\dots,i_n \in I$. In the formula above, we recall that $z_a$ is a placeholder for the variable $z_{i_a\bullet_a}$, for minimal positive integers $\bullet_1,\dots,\bullet_n$ such that $\bullet_a < \bullet_b$ if $a<b$ and $i_a=i_b$. By \cite[Lemma 3.10]{N Cat}, we have
\begin{equation}
	\label{eqn:integral 1}
	\langle E\psi, S(F) \rangle =	\int_{1 \ll |z_1| \ll \dots \ll |z_n|} \frac {\nu(z_1,\dots,z_n)F(z_1,\dots,z_n)}{ \prod_{1 \leq a < b \leq n} \zeta_{i_bi_a} \left(\frac {z_b}{z_a} \right)} \prod_{a=1}^n \psi_{i_a}(z_a)
\end{equation}
Using \eqref{eqn:two contours}, we rewrite \eqref{eqn:integral 1} as follows
\begin{multline}
	\langle E\psi, S(F) \rangle =	\label{eqn:integral 2}
	\sum_{\{1,\dots,n\} = \{a_1<\dots <a_k\} \sqcup \{b_1 < \dots < b_{\ell}\}} \\ \int_{z_{b_\ell}} \dots \int_{z_{b_1}} \int_{1 \gg |z_{a_k}| \gg \dots \gg |z_{a_1}|}  \frac {\nu(z_1,\dots,z_n)F(z_1,\dots,z_n)}{ \prod_{1 \leq a < b \leq n} \zeta_{i_bi_a} \left(\frac {z_b}{z_a} \right)} \prod_{a=1}^n \psi_{i_a}(z_a)
\end{multline}
Let us rewrite the expression above in a way that makes it clear that the integrand is expanded as $|z_{a_1}|,\dots,|z_{a_k}|$ $\ll$ $|z_{b_1}|,\dots,|z_{b_{\ell}}|$:
\begin{equation}
	\label{eqn:integral 3}
	\langle E\psi, S(F) \rangle =	\sum_{\{1,\dots,n\} = \{a_1<\dots <a_k\} \sqcup \{b_1 < \dots < b_{\ell}\}} \int_{z_{b_\ell}} \dots \int_{z_{b_1}} \int_{1 \gg |z_{a_k}| \gg \dots \gg |z_{a_1}|} 
\end{equation}
$$ 
\frac {\nu(z_{a_1},\dots,z_{a_k} \otimes z_{b_1},\dots,z_{b_{\ell}})F(z_{a_1},\dots,z_{a_k} \otimes z_{b_1},\dots,z_{b_{\ell}})}{ \prod_{1 \leq a < b \leq n} \zeta_{i_bi_a} \left(\frac {z_b}{z_a} \right)} \prod_{s=1}^{k} \psi_{i_{a_s}}(z_{a_s}) \prod_{t=1}^{\ell} \psi_{i_{b_t}}(z_{b_t})
$$
Since $\CS^+_{\geq \bp}$ is closed under multiplication by color-symmetric polynomials, the right-hand side of \eqref{eqn:integral 3} vanishes for all $\nu$ if and only if it vanishes for all $\nu$ multiplied by arbitrary color-symmetric polynomials in $z_{a_1},\dots,z_{a_k}$. Thus, we may replace 
$$
\psi_{i_{a_s}}(z_{a_s}) = z_{a_s}^{-r_{i_{a_s}}} \Big (\text{non-zero constant} + O(z_{a_s}) \Big) \quad \text{by} \quad z_{a_s}^{-r_{i_{a_s}}} 
$$
in \eqref{eqn:integral 3}, since the constant terms in the variables $z_{a_1},\dots,z_{a_k}$ are calculated only near 0. Thus, we conclude that $F \in J^{\bp}(\bpsi)$ if and only if
\begin{equation}
	\label{eqn:integral 3 bis}
	\sum_{\{1,\dots,n\} = \{a_1<\dots <a_k\} \sqcup \{b_1 < \dots < b_{\ell}\}} \int_{z_{b_\ell}} \dots \int_{z_{b_1}} \int_{1 \gg |z_{a_k}| \gg \dots \gg |z_{a_1}|} 
\end{equation}
$$ 
\frac {\nu(z_{a_1},\dots,z_{a_k} \otimes z_{b_1},\dots,z_{b_{\ell}})F(z_{a_1},\dots,z_{a_k} \otimes z_{b_1},\dots,z_{b_{\ell}})}{ \prod_{1 \leq a < b \leq n} \zeta_{i_bi_a} \left(\frac {z_b}{z_a} \right)} \prod_{s=1}^{k} z_{a_s}^{-r_{i_{a_s}}} \prod_{t=1}^{\ell} \psi_{i_{b_t}}(z_{b_t}) = 0
$$
for all $\nu$ as in \eqref{eqn:e spherical}. For the remainder of this proof, the symbol $\sum$ will stand for summing over all partitions $\{1,\dots,n\} = \{a_1<\dots<a_k\} \sqcup \{b_1 < \dots < b_\ell\}$, for various $k$ and $\ell$. Using relations \eqref{eqn:uug} and \eqref{eqn:coproduct shuffle plus}, it is straightforward to obtain
	$$
	\Delta(E) = \sum \ph E' \otimes E''
	$$
	where (above and henceforth) the symbol $\sum$ stands for summing over partitions $\{1,\dots,n\} = \{a_1<\dots<a_k\} \sqcup \{b_1 < \dots < b_\ell\}$ for various $k$, $\ell$, and we write
	\begin{align*}
		&\ph = \ph_{i_{b_1}}^+(z_{b_1}) \dots \ph_{i_{b_\ell}}^+(z_{b_\ell}) \\
		&E' = \text{Sym} \left[ \nu'(z_{a_1},\dots,z_{a_k}) \prod_{1 \leq s < t \leq k} \zeta_{i_{a_s} i_{a_t}} \left(\frac {z_{a_s}}{z_{a_t}} \right) \right] \prod_{1 \leq s \leq k, 1 \leq t \leq \ell}^{a_s < b_t} \frac {\zeta_{i_{a_s}i_{b_t}} \left(\frac {z_{a_s}}{z_{b_t}} \right)}{\zeta_{i_{b_t}i_{a_s}} \left(\frac {z_{b_t}}{z_{a_s}} \right)} \\ 
		&E'' = \text{Sym} \left[ \nu''(z_{b_1},\dots,z_{b_\ell}) \prod_{1 \leq s < t \leq \ell} \zeta_{i_{b_s} i_{b_t}} \left(\frac {z_{b_s}}{z_{b_t}} \right) \right] 
\end{align*} 
where $\nu(z_1,\dots,z_n) = \nu'(z_{a_1},\dots,z_{a_k}) \otimes \nu''(z_{b_1},\dots,z_{b_{\ell}})$.  Therefore, condition \eqref{eqn:integral 3} can be translated into the fact that $F \in J^{\bp}(\bpsi)$ if and only if \footnote{Note that we must use the following straightforward consequence of \eqref{eqn:coproduct shuffle minus} and \eqref{eqn:bialgebra 2}:
\begin{equation}
\label{eqn:consequence 2}
\langle \ph E', S(F') \rangle = \varepsilon(\ph) \langle E',S(F')\rangle
\end{equation}
for any $E' \in \CS^+$, $F' \in \CS^-$ and $\ph = \ph^+_{i_1,d_1}\ph^+_{i_2,d_2}\dots$.}
\begin{multline}
\label{eqn:integral 4}
\sum \left \langle E' \prod_{s=1}^{k} z_{a_s}^{-r_{i_{a_s}}} ,  S(F')  \right \rangle \\
\int_{z_{b_\ell}} \dots \int_{z_{b_1}} \frac {\nu''(z_{b_1},\dots,z_{b_{\ell}})F''(z_{b_1},\dots,z_{b_{\ell}})}{ \prod_{1 \leq s<t \leq \ell} \zeta_{i_{b_t}i_{b_s}} \left(\frac {z_{b_t}}{z_{b_s}} \right)}  \prod_{t=1}^{\ell} \psi_{i_{b_t}}(z_{b_t}) = 0
\end{multline}
By \eqref{eqn:interact 1}, we have $E \in \CS_{\geq \bp}^+ \Rightarrow E' \in \CS_{\geq \bp}^+$. Therefore, the display above allows us to immediately conclude the $\Leftarrow$ implication of \eqref{eqn:that}: if $F' \in J^{\bp}(z^{-\br})$ then the first line of \eqref{eqn:integral 4} vanishes, while if $F'' \in J^{\neq 0}(\bpsi)$ then the second line of \eqref{eqn:integral 4} vanishes due to Proposition \ref{prop:two contours}. We have thus shown that the map \eqref{eqn:decomposition} is well-defined and injective. To prove that it is surjective, we adapt the final paragraph in the proof of \cite[Theorem 1.3]{N Cat}, as follows. Consider any 
\begin{equation}
\label{eqn:consider any}
\tilde{F}' \otimes \tilde{F}'' \in \CS^-_{<\bp} \otimes \CS^-
\end{equation}
and we seek to construct $F \in \CS^-_{<\bp}$ such that
\begin{equation}
\label{eqn:jt}
F' \otimes F'' \equiv \tilde{F}' \otimes \tilde{F}'' \quad \text{mod} \quad J^{\bp}(z^{-\br}) \otimes \CS^- + \CS^-_{<\bp} \otimes J^{\neq 0}(\bpsi)
\end{equation}
with $F' \otimes F''$ as in \eqref{eqn:delta prime}. We do so by increasing induction on $|\hdeg \tilde{F}'|$ and by decreasing induction on $\vdeg \tilde{F}'$ to break ties (the base case of the latter induction is trivial because all shuffle elements of high enough vdeg lie in $J^{\bp}(z^{-\br})$). Thus, we assume that one can always pick $F$ such that \eqref{eqn:jt} holds whenever $\deg \tilde{F}'$ is smaller than a fixed amount, and we will construct $F$ when $\deg \tilde{F}'$ is equal to said amount. First of all, Proposition \ref{prop:two contours} shows that
$$
\tilde{F}''(z_{i1},\dots,z_{in_i}) \in J^{\neq 0}(\bpsi) \quad \Leftrightarrow \quad \tilde{F}''(z_{i1},\dots,z_{in_i}) \prod_{i \in I} \prod_{a=1}^{n_i} z_{ia} \in J^{\neq 0}(\bpsi)
$$
which implies that multiplication by $\prod_{i \in I} \prod_{a=1}^{n_i} z_{ia}$ is an automorphism of the finite-dimensional vector space $\CS_{-\bn}/(J^{\neq 0}(\bpsi) \cap \CS_{-\bn})$ for all $\bn$. Therefore, we can assume that $\tilde{F}''$ in \eqref{eqn:consider any} lies in $\CS^-_{<-N\bone}$ for a henceforth fixed $N \gg \vdeg \tilde{F}'$. Let
$$
F =  \tilde{F}'(z_{i1},\dots,z_{im_i})  * \left[ \tilde{F}''(z_{i,m_i+1},\dots,z_{in_i}) \prod_{i,j \in I} \prod_{a \leq m_i, b > m_j} \left(\frac {(-1)^{\delta_{ij}} z_{ia}^{\#_{ji}}}{c_{ji} z_{jb}^{\#_{ji}}} \right) \right]
$$
If $N$ is large enough, then the shuffle element in square brackets lies in $\CS^-_{<\bp}$, and therefore so does $F$. When we compute the coproduct of $F$ as in \eqref{eqn:delta prime}, either

\medskip 

\begin{itemize}[leftmargin=*]
	
\item some of the variables of $F'$ come from the shuffle element $\tilde{F}''$. However, since the latter shuffle element has degree in every variable at least $N$ ($\pm$ a constant), this would force $\vdeg F'$ to also be large enough and thus $F' \in J^{\bp}(z^{-\br})$;

\medskip 

\item  all the variables of $F'$ are among the variables of $\tilde{F}'$, but $\hdeg F' < \hdeg \tilde{F}'$. Then we invoke the induction hypothesis to achieve \eqref{eqn:jt};

\medskip 

\item the variables of $F'$ are precisely the ones of $\tilde{F'}$, i.e. $\hdeg F' = \hdeg \tilde{F}'$. Then 
$$
F' \otimes F'' = \tilde{F}' \otimes \tilde{F}''
$$
plus terms whose first tensor factor has vdeg greater than that of $\tilde{F}'$ (and thus can be accounted for by the induction hypothesis).  

\end{itemize} 

\medskip 

\end{proof}

\subsection{Factorization of $q$-characters} 
\label{sub:q-char}

Since the decomposition \eqref{eqn:decomposition} preserves the action of the positive loop Cartan subalgebra, it implies an equality of $q$-characters
\begin{equation}
\label{eqn:decomposition q-char}
\chi_q(L^\bp(\bpsi)) = \chi_q(L^\bp(z^{-\textbf{ord }\bpsi})) \cdot \chi_q(L^{\neq 0}(\bpsi))
\end{equation}
As shown in \cite{N Char}, the first term above can be completely calculated in terms of the graded dimensions of slope subalgebras, as follows. We will not reprove the result below, as it follows the proof of Theorem 1.3 of \loccit almost verbatim.

\medskip 

\begin{theorem}
\label{thm:factor}

For any $\bp \in \rr$ and $\br \in \zz$, we have
\begin{align*} 
&L^{\bp}(z^{-\br}) \qquad = \CS^-_{(-\binfty,\bp)} \Big/ J^{\bp}(z^{-\br})  \\ =& \bigoplus_{d=0}^{\infty} L^{\bp}(z^{-\br})_{d} = \bigoplus_{d=0}^{\infty} \CS^-_{(-\binfty,\bp)|d} \Big/ J^{\bp}(z^{-\br})_{d} 
\end{align*} 
The refined $q$-character 
\begin{equation}
\label{eqn:refined q-char}
\chi^{\br}_{\emph{ref}} =  \sum_{\bpsi \text{ loop weight}, \ d\in \BZ} \dim_{\BK}\left(\bpsi\text{-eigenspace in }L^{\bp}(z^{-\br})_{d} \right)  [\bpsi] v^d
\end{equation}
is given by the formula
\begin{equation}
\label{eqn:factor}
\chi^{\br}_{\emph{ref}} = \mathop{\prod_{t < 0 \text{ s.t. }d := - \bar{\bp}(t)\cdot \bm(t) \in \BZ}}_{\text{and } - \bp \cdot \bm(t) < d \leq (-\bp+\br)\cdot \bm(t)} \sum_{k = 0}^{\infty} \dim_{\BK}(\CB^-_{\bar{\bp}(t)|-k\bm(t)}) \left[ \bga^{-k\bm(t)} \right] v^{kd}
\end{equation}
where the product is determined by any catty-corner curve $\bar{\bp} : (-\infty,0) \rightarrow \rr$ such that $\bar{\bp}(0) = \bp$; this curve is assumed to be generic in the sense that for all $t<0$,
$$
\Big\{ \bn \in \nn | \bar{\bp}(t) \cdot \bn \in \BZ \Big\} = \begin{cases} 0 &\text{or} \\ \BN  \bm(t) &\text{for some }\bm(t) \in \nn \backslash \b0 \end{cases} 
$$

\end{theorem} 

\medskip 

\noindent In \eqref{eqn:factor}, we identify the weight $[\bga^{\bn}]$ with a constant loop weight, i.e. an $I$-tuple of constant power series whose constant term is given by the scalars $\gamma_{\bs^i,\bn} \in \BK^*$. Because of this, we note that \eqref{eqn:refined q-char} is simply a character and not a $q$-character: it does not actually depend on loop weights, only on the underlying weights.

\medskip 

\subsection{Eigenspaces}
\label{sub:eigenspaces}

As for the second term in \eqref{eqn:decomposition q-char}, it was shown in \cite{N Cat} that
\begin{equation}
\label{eqn:q-character non-zero}
\chi_q(L^{\neq 0}(\bpsi)) = [\bpsi] \sum_{\bn = (n_i)_{i \in I} \in \nn} \sum_{\bx \in (\BK^*)^{(\bn)}} \dim_{\BK}(L^{\neq 0}(\bpsi)_{\bx}) \prod_{i \in I} \prod_{a=1}^{n_i} A_{i,x_{ia}}^{-1}
\end{equation}
where the sum runs over $\bx = (x_{i1},\dots,x_{in_i})_{i \in I} \in (\BK^*)^{(\bn)} = \prod_{i \in I} (\BK^*)^{n_i}/S_{n_i}$, and 
\begin{equation}
\label{eqn:fm}
A_{i,x}^{-1} = \left[ \frac {\zeta_{ij} \left(\frac xz \right)}{\zeta_{ji} \left(\frac zx \right)} \right]_{j \in I}
\end{equation}
The right-hand side of \eqref{eqn:q-character non-zero} features the finite-dimensional vector spaces
\begin{equation}
\label{eqn:l x}
L^{\neq 0}(\bpsi)_{\bx} = \CS_{-\bn} \Big / J^{\neq 0}(\bpsi)_{\bx}
\end{equation}
where $J^{\neq 0}(\bpsi)_{\bx}$ denotes the set of $F(z_{i1},\dots,z_{in_i})_{i \in I} \in \CS_{-\bn}$ such that
\begin{equation}
\label{eqn:j x} 
\underset{z_n = x_n}{\text{Res}} \dots \underset{z_1 = x_1}{\text{Res}} \frac {F(z_1,\dots,z_n)(\text{any monomial in the }z_a)}{\prod_{1\leq a < b \leq n} \zeta_{i_ai_b} \left(\frac {z_a}{z_b} \right)} \prod_{a=1}^n \psi_{i_a}(z_a) = 0
\end{equation}
for all $i_1,\dots,i_n \in I$ such that $\bs^{i_1}+\dots+\bs^{i_n} = \bn$ and all orderings $x_1,\dots,x_n$ of the coordinates of $\bx = (x_{i1},\dots,x_{in_i})_{i \in I}$ such that $x_a = x_{i_a\bullet_a}$ for various $\bullet_a \geq 1$ (see \eqref{eqn:relabeling} for the meaning of $F(z_1,\dots,z_n)$ in formula \eqref{eqn:j x}). Thus, we realize the $q$-character \eqref{eqn:q-character non-zero} as the dimension of certain explicit vector spaces. More generally,
\begin{equation}
\label{eqn:eigenspace decomposition}
L^{\neq 0}(\bpsi) = \bigoplus_{\bn \in \nn} L^{\neq 0}(\bpsi)_{\bn}, \qquad L^{\neq 0}(\bpsi)_{\bn} = \bigoplus_{\bx \in (\BK^*)^{(\bn)}} L^{\neq 0}(\bpsi)_{\bx}
\end{equation} 
where each summand in the RHS is the generalized eigenspace of $L^{\neq 0}(\bpsi)$ on which 
\begin{align}
&\kappa_i^+ \text{ acts by } \ell_i \gamma_{\bs^i,-\bn} \label{eqn:kappa acts} \\
&p_{i,d} \text{ acts by } y_{i,d} - \sum_{j \in I} \alpha_{ij}^{(d)}(x_{j1}^d+\dots+x_{jn_j}^d) \label{eqn:p acts}
\end{align}
where we set $\psi_i(z) = \ell_i \exp \left(\sum_{d=1}^{\infty} \frac {y_{i,d}}{z^d} \right)$ and recall the notation in Subsection \ref{sub:cartan}. In particular, formula \eqref{eqn:p acts} follows from the fact that the loop Cartan elements $p_{i,d}$ act on the shuffle algebra $\CS^-$ by multiplication with appropriate linear combinations of color-symmetric Laurent polynomials, which descend to $L^{\neq 0}(\bpsi)_{\bn}$ as the formulas
\begin{equation}
\label{eqn:p acts again} 
p_{i,d} \left(F \text{ mod }J^{\neq 0}(\bpsi)_{\bn} \right) =  F \cdot \left(y_{i,d} - \sum_{j \in I} \alpha_{ij}^{(d)}(z_{j1}^d+\dots+z_{jn_j}^d) \right)\text{ mod }J^{\neq 0}(\bpsi)_{\bn}
\end{equation}
for any $F = F(z_{i1},\dots,z_{in_i})_{i \in I} \in \CS_{-\bn}$. 

\medskip 

\subsection{Shifted quantum loop algebras}
\label{sub:shifted}

In the present Subsection, we adapt the material of \cite{HN}, which in turn generalizes the representation theory of shifted quantum loop algebras from \cite{H3}. For any $\br \in \zz$, define a shifted quantum loop algebra
\begin{equation}
\label{eqn:shifted algebra}
\U^\br = \BK \Big \langle e_{i,d}, f_{i,d}, \ph^+_{i,d'}, \ph^-_{i,d'}\Big \rangle_{i \in I, d \in \BZ, d' \geq 0} \Big/ \Big(\text{\eqref{eqn:rel quantum 1}-\eqref{eqn:rel quantum 8} and modified \eqref{eqn:rel quantum 9}} \Big)
\end{equation}
where the appropriate modification of \eqref{eqn:rel quantum 9} is to replace $\ph^-_j(w)$ by $w^{-r_j} \ph^-_j(w)$ in the right-hand side. Clearly, we have $\U^\b0 = \U$. It was shown in \cite{HN} that for any loop weight $\bpsi$ with $\br = \textbf{ord }\bpsi$, one can construct a simple module
\begin{equation}
\label{eqn:shifted module}
\U^\br \curvearrowright L^{\text{sh}}(\bpsi) = \CS^- \Big/ J^{\text{sh}}(\bpsi)
\end{equation}
where
\begin{equation}
\label{eqn:j shifted}
J^{\text{sh}}(\bpsi)= \left\{F \in \CS^- \ \Big| \ \langle E \psi, F_1 * S(F_2) \rangle = 0 , \ \forall E \in \CS^+ \right \}
\end{equation}
Explicitly, the vanishing condition in the formula above reads
\begin{multline}
	\label{eqn:psi pairing shifted}
	\left\langle E(z_{i1},\dots,z_{in_i})_{i \in I} \prod_{i \in I} \prod_{a=1}^{n_i} \psi_{i}(z_{ia}), \right. \\ \left. \sum_{\b0 \leq \bm \leq \bn} F_1(z_{i1},\dots,z_{im_i})_{i \in I} * S(F_2(z_{i,m_i+1},\dots,z_{in_i})_{i \in I}) \right\rangle = 0
\end{multline}
and the pairing is calculated by expanding in the range
\begin{equation}
	\label{eqn:expansion}
	z_{i1},\dots,z_{i m_i} \sim 0 \qquad \text{and} \qquad z_{i,m_i+1},\dots,z_{in_i} \sim \infty
\end{equation}
This is the reason why \eqref{eqn:psi pairing shifted} does not vanish identically, despite the fact that $F_1 * S(F_2) = 0$ in any topological Hopf algebra due to the properties of the antipode.

\medskip

\begin{lemma} 
	\label{lem:iso}
	
	For any loop weight $\bpsi$, we have $L^{\neq 0}(\bpsi) = L^{\emph{sh}}(\bpsi)$ as vector spaces.
	
\end{lemma} 

\medskip 

\begin{proof} We will show that $J^{\neq 0}(\bpsi) = J^{\text{sh}}(\bpsi)$.	By plugging $E = e_{i_1,d_1} \dots e_{i_n,d_n}$ (such elements span $\CS^-$) in \eqref{eqn:psi pairing shifted}, we see that an element $F \in \CS^-$ lies in $J^{\text{sh}}(\bpsi)$ iff
$$
\sum_{\{1,\dots,n\} = \{a_1 < \dots < a_m\} \sqcup \{b_1 < \dots < b_{n-m}\}}  (-1)^{m} \int_{|z_{a_m}| \ll \dots \ll |z_{a_1}| \ll 1 \ll |z_{b_1}| \ll \dots \ll |z_{b_{n-m}}|}  
$$
$$
\frac {z_1^{d_1} \dots z_n^{d_n} F(z_1,\dots,z_n)}{\prod_{1\leq a < b \leq n} \zeta_{i_bi_a} \left(\frac {z_b}{z_a} \right)} \prod_{a=1}^n \psi_{i_a}(z_a) = 0
$$
for all $i_1,\dots,i_n \in I$ and $d_1,\dots,d_n \in \BZ$. This is none other than condition \eqref{eqn:hart}. \end{proof} 
	
\medskip 

\noindent In \eqref{eqn:shifted module}, elements of $L^{\text{sh}}(\bpsi)$ are given by $F \vac$ for various $F \in \CS^-$. The action is given as follows: $\CS^-$ acts by left multiplication, while $\CS^+$ acts by the formula
\begin{equation}
\label{eqn:action verma 2}
E \cdot F \vac =  F_2\vac \cdot \langle E_1,F_1\rangle \langle E_2 \psi, S(F_3) \rangle
\end{equation}
for all $E \in \CS^+$ and $F \in \CS^-$ (viewed as halves of the shifted algebra $\U^\br$), see \cite{HN}.

\medskip

\subsection{Regular$^{\neq 0}$ loop weights}
\label{sub:regular}

A rational loop weight $\bpsi$ is called \emph{regular$^{\neq 0}$} if 
$$
\textbf{ord }\bpsi = \b0
$$
i.e. each constituent rational function $\psi_i(z)$ is regular and non-zero at $z = 0$ (beside being regular and non-zero at $z = \infty$, as all loop weights are). The following result generalizes a well-known feature of simple modules of quantum affine algebras.

\medskip 

\begin{proposition}
\label{prop:regular}
	
If $\bpsi$ is regular$^{\neq 0}$, then the action $\CA^{\geq \bp}\curvearrowright L^{\bp}(\bpsi)$ extends to
\begin{equation}
\label{eqn:regular}
\U = \CA^{\geq \bp} \otimes \CA^{\leq \bp} \curvearrowright L^{\bp}(\bpsi)
\end{equation}
As we will show in the proof below, the module $L^{\bp}(\bpsi)$ does not depend on $\bp$.
	
\end{proposition}

\medskip 

\begin{proof} Formula \eqref{eqn:epsilon} shows that $J^{\bp}(z^{\b0})$ is the kernel of the counit $\varepsilon$, and so
$$
L^{\bp}(z^{\b0}) = \BK
$$
concentrated in horizontal degree $\b0$. Then we have that
\begin{equation}
\label{eqn:iso 1}
L^{\bp}(\bpsi) \stackrel{\text{Proposition \ref{prop:decomposition}}}\cong L^{\neq 0}(\bpsi) \stackrel{\text{Lemma \ref{lem:iso}}}\cong L^{\text{sh}}(\bpsi)
\end{equation}
The shifted quantum loop algebra for $\br = \b0$ is none other than $\U$, so the isomorphism \eqref{eqn:iso 1} will establish \eqref{eqn:regular} as soon as we show that the actions of $\CA^{\geq \bp}$ on the two sides are compatible. This is clear for the action of $F \in \CS^-_{<\bp}$, as such elements act on both sides of the equation by left multiplication. To show that the action of $E \in \CS^+_{\geq \bp}$ on the two sides of \eqref{eqn:iso 1} matches,  compare \eqref{eqn:action verma} with \eqref{eqn:action verma 2}
$$
F_1 \vac \cdot \langle E \psi , S(F_2) \rangle =  F_2\vac \cdot \langle E_1,F_1\rangle \langle E_2 \psi, S(F_3) \rangle
$$
The equation above holds because of \eqref{eqn:epsilon} and
$$
E \in \CS^+_{\geq \bp}, \ F \in \CS^-_{<\bp} \quad \stackrel{\eqref{eqn:interact 1}-\eqref{eqn:interact 3}}\Longrightarrow \quad E_1\in \CS^+_{\geq \bp}, \ F_1 \in \CS^-_{<\bp}
$$
\end{proof}

\medskip

\noindent As a consequence of Proposition \ref{prop:regular}, we will denote the simple modules \eqref{eqn:iso 1} as
\begin{equation}
\label{eqn:simple module without slope}
\U \curvearrowright L(\bpsi)
\end{equation}
for any regular$^{\neq 0}$ loop weight $\bpsi$, without any reference to the defining slope $\bp \in \rr$.

\medskip

\subsection{$R$-matrices}
\label{sub:r-matrices}

Consider now two modules $\U \curvearrowright V,W$ in category $\CO$, for example simple modules associated to regular$^{\neq 0}$ loop weights, as per Proposition \ref{prop:regular}. The various coproducts $\Delta_{\bp}$ give rise to a host of module structures
\begin{equation}
\label{eqn:structures}
\U \curvearrowright V \otimes_{\bp} W
\end{equation}
on the tensor product $V \otimes W$; the fact that the completion $\ho$ in which $\Delta_{\bp}$ takes values acts by finite sums on any element of $V \otimes W$ is an immediate consequence of the fact that $V,W$ are in category $\CO$, as we saw in Proposition \ref{prop:tensor product}. As $\bp$ varies, the module structures \eqref{eqn:structures} are all related by the evaluation of the universal $R$-matrices of Subsection \ref{sub:universal 2} in $V\otimes W$. For instance, we may consider \eqref{eqn:consider 1}
$$
_{\bp^2}\CR_{\bp^1} \in \U \ \bar{\otimes} \ \U \quad \leadsto \quad _{\bp^2}R_{\bp^1} \in \text{End}(V) \otimes \text{End}(W)
$$
for any $\bp^1, \bp^2$ in $\rr$, and then \eqref{eqn:intertwine 3} implies that we have a $\U$-intertwiner
\begin{equation}
\label{eqn:isomorphism 1}
_{\bp^2}R_{\bp^1} : V \otimes_{\bp^1} W \rightarrow  V \otimes_{\bp^2} W
\end{equation}
To do the same for the universal $R$-matrix of \eqref{eqn:consider 2}, we need to modify the construction as is usually done in the theory of integrable systems for quantum affine algebras. The reason is that while elements of the completion \eqref{eqn:hat} act by finite sums on $V \otimes W$ for any $V,W$ in category $\CO$, the universal $R$-matrix \eqref{eqn:consider 2} actually has infinitely many terms in any horizontal degree, corresponding to arbitrarily high vertical degrees in the first tensor factor. Therefore, we consider
$$
\U \curvearrowright V^u = V \quad \text{with the action of any }X \in \U \text{ renormalized by }u^{\vdeg X}
$$
Therefore
\begin{multline*} 
_{\bar{\bp}^2}\CR_{\bp^1} = \sum_{k} A_k \otimes B_k \in \U \ \widehat{\otimes} \ \U \quad \leadsto \\ \leadsto \quad _{\bar{\bp}^2}R_{\bp^1}(u) = \sum_k A_k u^{\vdeg A_k} \otimes B_k \in \text{End}(V) \otimes \text{End}(W)((u))
\end{multline*}
With this modification, formula \eqref{eqn:intertwine 4} implies that we have a $\U$-intertwiner
\begin{equation}
\label{eqn:isomorphism 2}
_{\bar{\bp}^2}R_{\bp^1}(u) : V^u \otimes_{\bp^1} W \rightarrow V^u \otimes^{\text{op}}_{\bp^2} W((u))
\end{equation}
When $V = L(\bpsi)$ and $W = L(\bpsi')$ are the simple modules \eqref{eqn:iso 1} associated to arbitrary regular$^{\neq 0}$ loop weights $\bpsi$ and $\bpsi'$, it is customary to renormalize the above intertwiner so that the matrix coefficient of the highest weight vector is 1:
\begin{equation}
\label{eqn:reduced r-matrix}
_{\bar{\bp}^2}R'_{\bp^1}(u) = \frac {_{\bar{\bp}^2}R_{\bp^1}(u)}{\langle \varnothing \otimes \varnothing | _{\bar{\bp}^2}R_{\bp^1}(u) | \varnothing \otimes \varnothing \rangle}
\end{equation}
The following generalizes a classic result of Drinfeld (\cite{Dr 0}) for finite type $\fg$.

\medskip 

\begin{theorem}
\label{thm:rational}

For any simple modules $V = L(\bpsi)$ and $W = L(\bpsi')$ associated to regular$^{\neq 0}$ loop weights $\bpsi$ and $\bpsi'$, the coefficients of the renormalized $\U$-intertwiner $_{\bar{\bp}^2}R'_{\bp^1}(u)$ of \eqref{eqn:reduced r-matrix} are (Laurent series expansions of) rational functions in $u$.

\end{theorem}

\medskip 

\noindent We note that the theorem above also requires the following technical condition, which holds automatically if one replaces the ground field $\BK$ by its algebraic closure: the zeta functions that define $\U$ are fully factored, in the sense of \eqref{eqn:fully factored}, and the rational loop weights $\bpsi,\bpsi'$ are also fully factored as follows:
\begin{equation}
\label{eqn:factored loop weight}
\psi_i(z) = \ell_i\frac {(z - s_{i|1})\dots (z-s_{i|\flat_i})}{(z - t_{i|1})\dots (z-t_{i|\flat_i})} \quad \text{and} \quad \psi'_i(z) = \ell_i'\frac {(z - s'_{i|1})\dots (z-s'_{i|\flat'_i})}{(z - t'_{i|1})\dots (z-t'_{i|\flat'_i})}
\end{equation}
for various $\ell_i,\ell_i',s_{i|\bullet},t_{i|\bullet},s'_{i|\bullet},t'_{i|\bullet} \in \BK^*$.

\medskip 

\begin{proof} For any $\br = (r_i)_{i \in I} \in \zz$, we consider the following shift automorphism 
\begin{equation}
	\label{eqn:shift}
\sigma_{\br} : \U \rightarrow \U, \qquad \sigma_{\br}(e_{i,d}) = e_{i,d+r_i}, \ \sigma_{\br}(f_{i,d}) = f_{i,d-r_i}, \ \sigma_{\br}(\ph^\pm_{i,d}) = \ph^\pm_{i,d} 
\end{equation}
In terms of the shuffle algebra realization of $\Upm \cong \CS^\pm$, we have
\begin{align*} 
&\sigma_{\br}(E) = E(z_{i1},\dots,z_{in_i})_{i \in I} \prod_{i \in I} \prod_{a=1}^{n_i} z_{ia}^{r_i} \\ 
&\sigma_{\br}(F) = F(z_{i1},\dots,z_{in_i})_{i \in I} \prod_{i \in I} \prod_{a=1}^{n_i} z_{ia}^{-r_i} 
\end{align*}
for all $E \in \CS_{\bn}$, $F \in \CS_{-\bn}$. Due to formula \eqref{eqn:hart}, it is clear that $\sigma_{\br}$ sends $J(\bpsi) := J^{\neq 0}(\bpsi) \subset \CS^-$ to itself, and then \eqref{eqn:iso 1} implies that $\sigma_{\br}$ descends to a linear map
\begin{equation}
\label{eqn:nabla}
\nabla^{\br} : L(\bpsi) \rightarrow L(\bpsi)
\end{equation}
The following result is straightforward, and we leave it as an exercise to the reader (it generalizes the particular case of quantum toroidal $\fgl_1$, where the role of $\nabla^{\br}$ was played by the nabla operator of \cite{BG}, and Lemma \ref{lem:nabla} can be found in \cite[(46)]{GN}).

\medskip

\begin{lemma}
\label{lem:nabla}

For any $x \in \U$, we have
\begin{equation}
\label{eqn:gn}
\sigma_{\br}(x) = \nabla^{\br} x  \nabla^{-\br}
\end{equation}
as endomorphisms of $L(\bpsi)$. 

\end{lemma}

\medskip 

\noindent As explained in Subsection \ref{sub:eigenspaces}, the finite-dimensional vector spaces $L(\bpsi)_{\bx}$ of \eqref{eqn:l x} are the generalized eigenspaces for the action of the loop Cartan subalgebra $\BK[\ph^\pm_{i,d}]$ in $L(\bpsi)$. By analogy with \eqref{eqn:kappa acts} and \eqref{eqn:p acts}, the eigenvalue of $\nabla^{\br}$ in $L(\bpsi)_{\bx}$ is
\begin{equation}
\label{eqn:scalar}
\bx^{-\br} = \prod_{i \in I} \prod_{a=1}^{n_i} x_{ia}^{-r_i}
\end{equation}
Thus, we conclude the following Jordan decomposition
\begin{equation}
\label{eqn:jordan}
\nabla^{\br} = \bx^{-\br} + \text{nilpotent operator}
\end{equation}
where $\bx^{-\br}$ denotes (by abuse of notation) the diagonal operator which acts in any generalized eigenspace $L(\bpsi)_{\bx}$ of \eqref{eqn:l x} by the scalar \eqref{eqn:scalar}. Moreover, it is clear from the definition of slope subalgebras in Subsections \ref{sub:slope}-\ref{sub:slope subalgebras} that we have
\begin{equation}
\label{eqn:slope shift}
\sigma_{\br}(\CB_{\boldsymbol{s}}) = \CB_{\boldsymbol{s}+\br}
\end{equation}
for all $\boldsymbol{s} \in \rr$ and $\br \in \zz$. We now have all the tools we need to prove that the intertwiner \eqref{eqn:reduced r-matrix} is a rational function of $u$ for $V = L(\bpsi)$ and $W = L(\bpsi')$. To this end, note that formula \eqref{eqn:factorization r-matrix 3} implies that $_{\bar{\bp}^2}R_{\bp^1}(u)$ is a product of three operators:
	
\medskip 

\begin{itemize}[leftmargin=*]
	
\item $X_1(u) = \prod_{t \in [t_2,\infty)}^{\rightarrow} P_{\bp'(t)}(u)$, for a catty-corner curve $\bp' : [t_2,\infty) \rightarrow \rr$,

\medskip 

\item $X_2(u) = P_{\binfty}(u) = P'_{\binfty}P''_{\binfty}(u)$ with the notation of Subsection \ref{sub:infinite slope 2}, and

\medskip 

\item $X_3(u) = \prod_{t \in (-\infty,t_1)}^{\rightarrow} P_{\bp(t)}^{\text{op}}(u)$, for a catty-corner curve $\bp : (-\infty,t_1) \rightarrow \rr$.

\medskip 

\end{itemize}

\noindent In the formulas above, if for any $\boldsymbol{s} \in \rr$ the slope universal $R$-matrix \eqref{eqn:slope r-matrix} is
$$
\CP_{\boldsymbol{s}} = \sum_k A_k \otimes B_k \in \CB^+_{\boldsymbol{s}} \ \bar{\otimes} \ \CB^{-}_{\boldsymbol{s}}
$$
then we write
$$
P_{\boldsymbol{s}}(u) = \sum_k A_k u^{\vdeg A_k} \otimes B_k \in \text{End}(L(\bpsi)) \otimes \text{End}(L(\bpsi'))((u))
$$
Similarly, $P'_{\binfty}$ and $P''_{\binfty}(u)$ are the images in $\text{End}(L(\bpsi)) \otimes \text{End}(L(\bpsi'))((u))$ of the canonical tensors of the pairings \eqref{eqn:first pairing} and \eqref{eqn:second pairing}, respectively; note that we must assume these pairings to be non-degenerate to even define the universal $R$-matrices, see Subsection \ref{sub:infinite slope 2}. As the former of these pairings does not depend on $u$, we will only deal with the latter, which is explicitly
\begin{equation}
\label{eqn:double prime}
P''_{\binfty}(u) = \exp \left( \sum_{i,j \in I} \sum_{d=1}^{\infty} \frac {\beta^{(d)}_{ij}}d p_{i,d} u^d \otimes p_{j,-d} \right)
\end{equation}
The scalars $\beta_{ij}^{(d)}$ that appear in the formula above must satisfy the equations
\begin{equation}
\label{eqn:equations}
\sum_{\bullet \in I} \alpha^{(d)}_{i \bullet} \beta^{(d)}_{\bullet j} = \beta^{(d)}_{i \bullet} \alpha^{(d)}_{\bullet j} = \delta_{ij}, \quad \forall i,j \in I, \forall d \geq 1 
\end{equation}
with the notation as in \eqref{eqn:notation},
in order for \eqref{eqn:double prime} to give rise to the canonical tensor of the pairing \eqref{eqn:second pairing}. By \eqref{eqn:p acts again}, the action of $p_{i,d}$ on $L(\bpsi)_{\bn}$ is given by
$$
p_{i,d} (F \text{ mod }J(\bpsi)_{\bn}) = F \left( \sum_{\bullet=1}^{\flat_i} (t_{i|\bullet}^d - s_{i|\bullet}^d) - \sum_{k \in I} \sum_{a=1}^{n_k} \alpha_{ki}^{(d)} z_{ka}^d \right) \text{ mod }J(\bpsi)_{\bn}
$$
for any $F(z_{k1},\dots,z_{kn_k})_{k \in I} \in \CS_{-\bn}$. By analogy, we have 
$$
p_{j,-d} (F' \text{ mod }J(\bpsi')_{\bn'}) = F' \left( \sum_{\bullet=1}^{\flat'_j} ({t'_{j|\bullet}}^{-d} - {s'_{j|\bullet}}^{-d}) - \sum_{k' \in I} \sum_{a'=1}^{n'_{k'}} \alpha_{jk'}^{(d)} z_{k'a'}^{'-d} \right) \text{ mod }J(\bpsi')_{\bn'}
$$
for any $F'(z'_{k1},\dots,z'_{kn'_k})_{k \in I} \in \CS_{-\bn'}$, where the $\alpha$'s and the $s,t$'s are defined in \eqref{eqn:notation} and \eqref{eqn:factored loop weight}, respectively (we use different notation for the variables of $F$ and $F'$ in the formulas above in order to emphasize the fact that they represent elements of the different modules $L(\bpsi)$ and $L(\bpsi')$). Putting the formulas above together, we conclude that \eqref{eqn:double prime} sends a tensor $(F \text{ mod }J(\bpsi)_{\bn}) \otimes (F' \text{ mod }J(\bpsi')_{\bn'})$ to \footnote{Note that term in the expression below involving the product of $(t_{i|\bullet}^d - s_{i|\bullet}^d)$ and $({t'_{j|\bullet}}^{-d} - {s'_{j|\bullet}}^{-d})$ is missing because we renormalized the $R$-matrix in the right-hand side of \eqref{eqn:reduced r-matrix}.}
\begin{align*} 
\exp &\left( \sum_{i,j,k' \in I} \sum_{d = 1}^{\infty} \sum_{\bullet=1}^{\flat_i} \sum_{a'=1}^{n'_{k'}} \frac {\beta_{ij}^{(d)} \alpha_{jk'}^{(d)}}d \left[ \left(\frac {s_{i|\bullet} u}{z'_{k'a'}} \right)^d - \left(\frac {t_{i|\bullet} u}{z'_{k'a'}} \right)^d \right]  \right. \\ 
&+\sum_{i,j,k \in I} \sum_{d = 1}^{\infty} \sum_{\bullet=1}^{\flat'_j} \sum_{a=1}^{n_k} \frac {\beta_{ij}^{(d)} \alpha_{ki}^{(d)}}d \left[ \left(\frac {z_{ka} u}{s'_{j|\bullet}}\right)^d - \left(\frac {z_{ka} u}{t'_{j|\bullet}}\right)^d \right] \\
&+ \left.\sum_{i,j,k,k' \in I} \sum_{d=1}^{\infty} \sum_{a=1}^{n_k} \sum_{a'=1}^{n_{k'}'} \frac {\beta_{ij}^{(d)} \alpha_{ki}^{(d)} \alpha_{jk'}^{(d)}}d \left(\frac {z_{ka}u}{z'_{k'a'}} \right)^d \right) \stackrel{\eqref{eqn:equations}}= \\
\exp &\left( \sum_{i \in I} \sum_{d = 1}^{\infty} \sum_{\bullet=1}^{\flat_i} \sum_{a'=1}^{n'_i} \frac 1d \left[ \left(\frac {s_{i|\bullet} u}{z'_{ia'}} \right)^d - \left(\frac {t_{i|\bullet} u}{z'_{ia'}} \right)^d \right]  \right. \\ 
&+\sum_{j \in I} \sum_{d = 1}^{\infty} \sum_{\bullet=1}^{\flat'_j} \sum_{a=1}^{n_j} \frac 1d \left[ \left(\frac {z_{ja} u}{s'_{j|\bullet}}\right)^d - \left(\frac {z_{ja} u}{t'_{j|\bullet}}\right)^d \right] \\
&+ \left.\sum_{k,k' \in I} \sum_{d=1}^{\infty} \sum_{a=1}^{n_k} \sum_{a'=1}^{n_{k'}'} \frac {\alpha_{kk'}^{(d)}}d \left(\frac {z_{ka}u}{z'_{k'a'}} \right)^d \right) \stackrel{\eqref{eqn:consequence}}=
\end{align*} 
$$
\prod_{i \in I} \prod_{\bullet=1}^{\flat_i} \prod_{a'=1}^{n_i'} \frac {z_{ia'}' - t_{i|\bullet} u}{z_{ia'}' - s_{i|\bullet} u} \prod_{j \in I} \prod_{\bullet=1}^{\flat'_j}\prod_{a=1}^{n_j}  \frac {z_{ja}u - t'_{j|\bullet}}{z_{ja}u - s'_{j|\bullet}} \prod_{k,k' \in I} \prod_{a=1}^{n_k}  \prod_{a'=1}^{n'_{k'}}\prod_b \frac {z'_{k'a'}-z_{ka}s_{kk'|b}u}{z'_{k'a'}-z_{ka}s^{-1}_{k'k|b}u}
$$
where in the latter product, we have $b \in \{1,\dots,\#_{kk'}+\#_{k'k}+\delta_{kk'}\}$, as in \eqref{eqn:consequence}. The formula above for the operator \eqref{eqn:double prime} is a rational function in $u$, precisely as we needed to show.

\medskip 

\noindent As for the operators $X_1(u)$ and $X_3(u)$, it suffices to show that the former is rational in $u$, as the latter is treated analogously. To this end, we choose a catty-corner curve $\bp'$ which has the property that 
\begin{equation}
\label{eqn:catty}
\bp'(t+1) = \bp'(t)+\br
\end{equation} 
for all $t \geq t_2$, for some henceforth fixed $\br \in \BZ_{>0}^I$. If we write
$$
\prod_{t \in [t_2,t_2+1)}^{\rightarrow} \CP_{\bp'(t)} = \sum_k A_k \otimes B_k \in \U^+ \ \bar{\otimes} \ \U^-
$$
where the sum over $k$ is finite in any horizontal degree, then formulas \eqref{eqn:slope shift} and  \eqref{eqn:catty} imply that
$$
\prod_{t \in [t_2,\infty)}^{\rightarrow} \CP_{\bp'(t)} = \prod_{\ell = 0}^{\infty} \left( \sum_k \sigma_{\br}^{\ell}(A_k) \otimes \sigma_{\br}^{\ell}(B_k) \right) = \sum_{k_0,k_1,k_2,\dots} \left[ {\prod_{\ell=0}^{\infty}}^* \sigma_{\br}^{\ell}(A_{k_\ell} \otimes B_{k_\ell}) \right]
$$
(by definition, the products denoted $\prod^*$ above have the property that $A_{k_\ell} = B_{k_\ell} = 1$ for all but finitely many $\ell$'s, which resolves any possible convergence issue). We may then use formula \eqref{eqn:gn}, together with the obvious fact that $\vdeg \sigma_{\br}^{\ell}(A) = \vdeg A + \ell \br \cdot \hdeg A$, to deduce from the display above the following equality of endomorphisms of $L(\bpsi) \otimes L(\bpsi')$ with coefficients in $\BK((u))$:
$$
X_1(u) = \sum_{k_0,k_1,k_2,\dots} \left[ {\prod_{\ell=0}^{\infty}}^* (\nabla^{\ell \br} \otimes \nabla^{\ell \br} ) (A_{k_\ell} u^{\vdeg A_k + \ell \br \cdot \hdeg A_k} \otimes B_k) (\nabla^{-\ell \br} \otimes \nabla^{-\ell \br}) \right]
$$
In other words, in any horizontally graded component of $L(\bpsi) \otimes L(\bpsi')$, the operator $X_1(u)$ is equal to a finite linear combination of sums of the form
\begin{align*} 
\sum_{0 \leq \ell < \ell' < \dots < \ell'' < \ell'''} &(\nabla^{\ell \br} \otimes \nabla^{\ell \br}) (A_k u^{\vdeg A_k+\ell \br \cdot \hdeg A_k} \otimes B_k) \\ &(\nabla^{(\ell'-\ell) \br} \otimes \nabla^{(\ell'-\ell) \br})(A_{k'} u^{\vdeg A_{k'}+\ell' \br \cdot \hdeg A_{k'}} \otimes B_{k'}) \\ &\dots \\
&(\nabla^{(\ell'''-\ell'') \br} \otimes \nabla^{(\ell'''-\ell'') \br})(A_{k'''} u^{\vdeg A_{k'''}+\ell''' \br \cdot \hdeg A_{k'''}} \otimes B_{k'''}) (\nabla^{-\ell''' \br} \otimes \nabla^{-\ell''' \br})
\end{align*}
for various indices $k,k',\dots,k'',k'''$. However, we claim that any infinite sum as in the display above is actually a rational function in $u$. This follows from the general claim that for any finite-dimensional vector space $S$, linear maps $T \in GL(S)$ and $X_1,\dots,X_m \in \text{End}(S)$ and integers $a_1,\dots,a_m \in \BZ_{>0}$, $b \in \BZ$, we have that
$$
\sum_{\ell_1,\dots,\ell_m = 0}^{\infty} T^{\ell_1} X_1 T^{\ell_2} X_2 \dots T^{\ell_m} X_m T^{-\ell_1-\dots-\ell_m} u^{a_1\ell_1+\dots+a_m\ell_m+b}
$$
is a rational function in $u$ (which one proves by writing the Jordan decomposition of $T$ into a diagonal plus a nilpotent matrix, thus allowing one to reduce the above statement to the elementary case when $T$ is diagonal). \end{proof}

\medskip 

\noindent One could also define $u$-twisted versions of the intertwiner \eqref{eqn:isomorphism 1}, but it would actually be Laurent polynomial in $u$:
\begin{equation}
\label{eqn:isomorphism 1 bis}
_{\bp^2}R_{\bp^1} : V^u \otimes_{\bp^1} W \rightarrow V^u \otimes_{\bp^2} W[u^{\pm 1}]
\end{equation}
There is a host of interesting problems that can be asked about the $R$-matrices above, such as the Bethe ansatz, calculation of transfer matrices and XXZ-type Hamiltonians, see \cite{FH} and numerous other works. The infrastructure we have set up in the present paper shows that it is reasonable to ask these questions in the generality of quantum loop algebras for all $(I,\BK,\zeta_{ij}(x))$, which goes beyond the quantum affine or toroidal algebras that have been studied so far (\cite{FJMM, FJMM2}).

\medskip 

\begin{remark}
	
It would be interesting to study the version of \eqref{eqn:isomorphism 1}-\eqref{eqn:isomorphism 2} when one of the slopes is $\binfty$. In more detail, the tensors \eqref{eqn:partial 1}-\eqref{eqn:partial 2} give rise to operators
\begin{align}
&\CR_{\bp} \in \U \ \bar{\otimes} \ \U \quad \leadsto \quad R_{\bp}(u) \in \emph{End}(V) \otimes \emph{End}(W)((u)) \label{eqn:remark 1} \\
&_{\bar{\bp}}\CR \in \U \ \bar{\otimes} \ \U \quad \leadsto \quad _{\bar{\bp}}R(u) \in \emph{End}(V) \otimes \emph{End}(W)((u)) \label{eqn:remark 2} 
\end{align}	
which produce $\U$-intertwiners
\begin{align}
&R_{\bp}(u) : V^u \otimes_{\bp} W \rightarrow V^u \otimes W((u)) \label{eqn:isomorphism 3} \\
&_{\bar{\bp}}R(u) : V^u \otimes W \rightarrow V^u \otimes^{\emph{op}}_{\bp} W((u)) \label{eqn:isomorphism 4} 
\end{align}
where $\otimes$ without any subscript denotes the Drinfeld coproduct. The latter coproduct was used in \cite{H1, H2} to define the fusion product of the modules $V$ and $W$, and it would be very interesting to compare the fusion products of \loccit with $V \otimes_{\b0} W$.

\end{remark}

\end{document}